\documentclass[EJP]{ejpecp}  

\usepackage[T1]{fontenc}
\usepackage[latin1]{inputenc}
\usepackage[english]{babel}
\usepackage{stmaryrd}
\usepackage{enumerate}

\SHORTTITLE{Second order backward SDE with random terminal time} 

\TITLE{Second order backward SDE with random terminal time \thanks{This work benefits from the financial support of the ERC Advanced Grant 321111 and the Chairs \textit{Financial Risk} and \textit{Finance and Sustainable Development} and the NSFC Grant No.11801365.}}

\AUTHORS{
  Yiqing~Lin\footnote{School of Mathematical Sciences, MOE-LSC, Shanghai Jiaotong University, 200240 Shanghai, China.   
  	\BEMAIL{yiqing.lin@sjtu.edu.cn}} 
  \and 
  Zhenjie~Ren\footnote{CEREMADE, Universit\'e Paris-Dauphine, PSL Research University, F-75016 Paris, France. 
  	\BEMAIL{ren@ceremade.dauphine.fr}}  
  \and Nizar~Touzi\footnote{CMAP, Ecole Polytechnique, F-91128 Palaiseau, France.  
  	\BEMAIL{nizar.touzi@polytechnique.edu}} 
  \and Junjian~Yang\footnote{Fakult\"at f\"ur Mathematik und Geoinformation, TU Wien, A-1040 Vienna, Austria. \BEMAIL{junjian.yang@tuwien.ac.at}}}


\AMSSUBJ{60H10; 60H30}

\SUBMITTED{May 08, 2018} 
\ACCEPTED{July 23, 2020}

\ARXIVID{1802.02260}  
\HALID{al-02293001}

\VOLUME{0}
\YEAR{2016}
\PAPERNUM{0}
\DOI{10.1214/YY-TN}

\ABSTRACT{Backward stochastic differential equations extend the martingale representation theorem to the nonlinear setting. This can be seen as path-dependent counterpart of the extension from the heat equation to fully nonlinear parabolic equations in the Markov setting. This paper extends such a nonlinear representation to the context where the random variable of interest is measurable with respect to the information at a finite stopping time. We provide a complete wellposedness theory which covers the semilinear case (backward SDE), the semilinear case with obstacle (reflected backward SDE), and the fully nonlinear case (second order backward SDE).}

\newcommand{\Om}{\Omega}
\newcommand{\om}{\omega}

\newcommand{\cB}{\mathcal{B}}
\newcommand{\cC}{\mathcal{C}}
\newcommand{\cD}{\mathcal{D}}
\newcommand{\cE}{\mathcal{E}}
\newcommand{\cF}{\mathcal{F}}
\newcommand{\cG}{\mathcal{G}}
\newcommand{\cH}{\mathcal{H}}
\newcommand{\cI}{\mathcal{I}}
\newcommand{\cL}{\mathcal{L}}
\newcommand{\cM}{\mathcal{M}}
\newcommand{\cN}{\mathcal{N}}
\newcommand{\cP}{\mathcal{P}}
\newcommand{\cQ}{\mathcal{Q}}
\newcommand{\cT}{\mathcal{T}}
\newcommand{\cU}{\mathcal{U}}
\newcommand{\cY}{\mathcal{Y}}
\newcommand{\cZ}{\mathcal{Z}}

\newcommand{\dbD}{\mathbb{D}}
\newcommand{\dbE}{\mathbb{E}}
\newcommand{\dbF}{\mathbb{F}}
\newcommand{\dbG}{\mathbb{G}}
\newcommand{\dbH}{\mathbb{H}}
\newcommand{\dbI}{\mathbb{I}}
\newcommand{\dbK}{\mathbb{K}}
\newcommand{\dbL}{\mathbb{L}}
\newcommand{\dbN}{\mathbb{N}}
\newcommand{\dbP}{\mathbb{P}}
\newcommand{\dbQ}{\mathbb{Q}}
\newcommand{\dbR}{\mathbb{R}}
\newcommand{\dbS}{\mathbb{S}}
\newcommand{\dbU}{\mathbb{U}}
\newcommand{\dbY}{\mathbb{Y}}

\newcommand{\sint}{\stackrel{\mbox{\tiny$\bullet$}}{}}
\newcommand{\Ind}{\mathbb{1}}
\newcommand{\rmD}{{\rm D}}
\newcommand{\prog}{\operatorname{Prog}}
\DeclareMathOperator*{\esssup}{\textnormal{ess\,sup}}
\DeclareMathOperator*{\essinf}{\textnormal{ess\,inf}}
\DeclareMathOperator{\sgn}{\mathrm{sgn}}

\begin{document}

\section{Introduction}

Let $(\Om,\cF,\{\cF_t\}_{t\ge 0},\dbP)$ be a filtered probability space, supporting a $d-$dimensional Brownian motion $W$. The martingale representation theorem states that any integrable $\cF_\tau-$measurable random variable $\xi$, for some $\dbF-$stopping time $\tau$, can be represented as $\xi=\dbE\xi+(Z\cdot W)_\tau+N_\tau$, for some square integrable $\dbF-$predictable process $Z$, and some martingale $N$ with $N_0=0$ and $[N,W]=0$. In particular when $\dbF$ is the (augmented) canonical filtration of the Brownian motion, $N=0$. This result can be seen as the path-dependent counterpart of the heat equation. Indeed, a standard density argument reduces to the case $\xi=g(W_{t_0},\ldots,W_{t_n})$ for an arbitrary partition $0=t_0<\ldots<t_n=T$ of $[0,T]$, where the representation follows from a backward resolution of the heat equation $\partial_tv+\frac12\Delta v=0$ on each time interval $[t_{i-1},t_i]$, $i=1,\ldots,n$, and the $Z$ process is identified to the space gradient of the solution.

As a first extension of the martingale representation theorem, the seminal work of Pardoux \& Peng \cite{PP90} introduced the theory of backward stochastic differential equations in finite horizon.
In words, this theory provides a representation of an $\cF_T-$measurable random variable $\xi$ with appropriate integrability as $\xi=Y_T$ with 
 $$ Y_{t}=Y_0-\int_0^{t} f_s(Y_s,Z_s)ds+(Z\cdot W)_{t}+N_{t}, \quad t\geq 0, $$ 
 where $f$ is a given random field. 
In the Markov setting where $\xi=g(W_T)$ and $f_t(\om,y,z)=f\big(t,W_t(\om),y,z\big)$, $t\ge 0$, it turns out that $Y_t(\om)=v(t,W_t(\om))$ for some deterministic function $v:\dbR_+\times\dbR^d\longrightarrow\dbR$, which is easily seen to correspond to the semilinear heat equation 
  $$\partial_tv+\frac12\Delta v+f(.,v,Dv)=0, $$
  by the fact that the $Z$ process again identifies the space gradient of $v$.

It was extended further to the random horizon setting by \cite{Peng91}, Darling \& Pardoux \cite{DP97}. On one hand, these results provide a representation for an $\cF_\tau-$measurable random variable $\xi$ with appropriate integrability as $\xi=Y_\tau$ with 
 $$ Y_{t\wedge\tau}=Y_0-\int_0^{t\wedge\tau} f_s(Y_s,Z_s)ds+(Z\cdot W)_{t\wedge\tau}, \quad t\ge 0, $$ 
where $f$ is a given random field. On the other hand, they give probabilistic interpretation to solutions of semilinear elliptic PDEs. As our interest in this paper is on the random horizon setting, we refer the interested reader to the related works by El Karoui \& Huang \cite{EH97}, Briand \& Hu \cite{BH98}, Briand \& Carmona \cite{BC00}, Bender \& Kohlmann \cite{BK00}, Royer \cite{Roy04}, Bahlali, Elouaflin \& N'zi \cite{BEN04}, Hu and Tessitore \cite{HT07}, Popier \cite{Pop07}, Briand and Confortola \cite{BC08}, Wang, Ran and Chen \cite{WRC07}, Papapantoleon, Possama\"i and Saplaouras \cite{PPS18}. We also mention the related works of Hamad\`ene, Lepeltier \& Wu \cite{HLW99}, Chen \& Wang \cite{CW00} and Hu and Schweizer \cite{HS11}, which study BSDEs with infinite horizon. 

Our main interest in this paper is on the extension to the fully nonlinear second order parabolic equations, as initiated in the finite horizon setting by Soner, Touzi \& Zhang \cite{STZ12}, and further developed by Possama\"{\i}, Tan \& Zhou \cite{PTZ15}, see also the first attempt by Cheridito, Soner, Touzi \& Victoir \cite{CSTV07}, and the closely connected BSDEs in a nonlinear expectation framework of Hu, Ji, Peng \& Song \cite{HJPS14a, HJPS14b} (called GBSDEs). This extension is performed on the canonical space of continuous paths with canonical process denoted by $X$. The key idea is to reduce the fully nonlinear representation to a semilinear representation which is required to hold simultaneously under an appropriate family $\cP$ of singular semimartingale measures on the canonical space. Namely, an $\cF_T-$ random variable $\xi$ with appropriate integrability is represented as 
$\xi=Y_T,$ where
 \begin{align*}
   Y_t=Y_0-\int_0^t F_s(Y_s,Z_s,\widehat\sigma_s)ds+(Z\sint X)_t+U^\dbP_t,
   ~~t\ge 0,~~\dbP-\mbox{a.s. for all}~~\dbP\in\cP.
 \end{align*}
Here, $\widehat\sigma_s^2ds=d\langle X\rangle_s$, and $U^\dbP$ is a supermartingale with $U^\dbP_0=0$, $[U^\dbP,X]=0$, $\dbP-$a.s. for all $\dbP\in\cP$ satisfying the minimality condition $\sup_{\dbP\in\cP}\dbE^\dbP[U^\dbP_T]=0$. Loosely speaking, in the Markov setting where $Y_t(\om)=v\big(t,X_t(\om)\big)$ for some deterministic function $v$, the last representation implies that $v$ is a supersolution of a semilinear parabolic PDE parameterized by the diffusion coefficient 
$$ -\partial_t v - \frac{1}{2}{\rm Tr}\big[\sigma\sigma^{\top}D^2v\big]-F(t,x,v,Dv,\sigma) \ge 0, $$ 
and the minimality condition induces the fully nonlinear parabolic PDE 
$$ -\partial_t v - \sup_{\sigma}\left\{\frac12{\rm Tr}\big[\sigma\sigma^{\top}D^2v\big]+F(t,x,v,Dv,\sigma)\right\} = 0. $$

Our main contribution is to extend the finite horizon fully nonlinear representation of \cite{STZ12} and \cite{PTZ15} to the context of a random horizon defined by a finite $\dbF-$stopping time. In view of the formulation of  second order backward SDEs as backward SDEs holding simultaneously under a non-dominated family of singular measures, we review --and in fact complement-- the corresponding theory of backward SDEs, and we develop the theory of reflected backward SDEs, which is missing in the literature, and which plays a crucial role in the wellposedness of second order backward SDEs. 

Finally, we emphasize that backward SDEs and their second order extension provide a Sobolev-type of wellposedness as uniqueness holds within an appropriate integrability class of the solution $Y$ and the corresponding ``space gradient'' $Z$. Also, our extension to the random horizon setting allows in particular to cover the elliptic fully nonlinear second order PDEs with convex dependence on the Hessian component.

The paper is organized as follows. Section \ref{sect:pre} sets the notations used throughout the paper. Our main results are contained in Section \ref{sect:mainresults}, with proofs reported in the remaining sections. Namely, Section \ref{sect:rbsde} contains the proofs related to backward SDEs and the corresponding reflected version, while Sections \ref{sect:2bsde} and \ref{sect:existence} focus on the uniqueness and the existence, respectively, for the second order backward SDEs.

\section{Preliminaries}\label{sect:pre}

\subsection{Canonical space} \label{sect:canonical}

Fix $d\in\mathbb{N}$, and let $\Omega 
= 
\big\{\om\in\cC\big([0, \infty) ; \dbR^d\big) :  \om_0={\bf 0} 
\big\}$
be the space of continuous paths starting from the origin equipped with the distance defined by 
$ \|{\om}-{\om}'\|_\infty:=\sum_{n\ge 0}2^{-n}\big(\sup_{0\leq t\leq n}\|{\om}_{t}-{\om}'_{t}\|\wedge 1\big). $ 
Denote by $X$ the canonical process. Let $\mathcal{M}_1$ be the collection of all probability measures on $(\Omega, \cF)$, equipped with the topology of weak convergence.  
Denote by $\dbF:=(\cF_t)_{t\geq 0}$ the raw filtration generated by the canonical process $ X$.  
Denote by $\dbF^+:=(\cF^+_t)_{t\geq 0}$ the right limit of $(\cF_t)_{t\geq 0}$. For each $\mathbb{P}\in \mathcal{M}_1$, we denote by $\dbF^{+,\dbP}$ the augmented filtration of $\dbF^+$ under $\mathbb{P}$. 
The filtration $\dbF^{+,\dbP}$ is the coarsest filtration satisfying the usual conditions.  
We denote by $\dbF^U:=\big(\cF^U_t\big)_{t\ge 0}$ and $\dbF^{+,U}:=\big(\cF^{+,U}_t\big)_{t\ge 0}$ the (right-continuous) universal completed filtration defined by 
$$ \cF_t^U:=\bigcap_{\dbP\in \mathcal{M}_1}\cF_t^\dbP  \quad  \mbox{ and } \quad \cF^{+,U}_t := \bigcap_{\dbP\in \mathcal{M}_1}\cF^{+,\dbP}_t. 
$$
Clearly, $\dbF^{+,U}$ is right-continuous. 
Similarly, for $\cP\subseteq \mathcal{M}_1$, we introduce $\dbF^\cP:=\big(\cF^\cP_t\big)_{t\ge 0}$ and $\dbF^{+,\cP}:=\big(\cF^{+,\cP}_t\big)_{t\ge 0}$, where
	\begin{align*}
      \cF_t^\cP:=\bigcap_{\dbP\in \cP}\cF_t^\dbP  \quad \mbox{and} \quad
      \cF^{+,\cP}_t := \bigcap_{\dbP\in \cP}\cF^{+,\dbP}_t. 
	\end{align*}
For any family $\cP\subseteq\cM_1$, we say that a property holds $\cP-$quasi-surely, abbreviated as $\cP-$q.s., if it holds $\dbP$-a.s. for all $\dbP\in\cP$.

Define $\cP_{loc}$ the subset of $\cM_1$ such that, for each $\dbP\in \cP_{loc}$, $X$ is $\dbP$-local martingale whose quadratic variation $\langle X \rangle$ is absolutely continuous in $t$ with respect to the Lebesgue measure.  
Note that the $d\times d$-matrix-valued  processes $\langle X \rangle$ can be defined pathwisely, and we may introduce the corresponding $\dbF$-progressively measurable density processes
  \begin{align*}
  	\widehat{a}_t
  	:=
  	\limsup_{n\rightarrow \infty} n \big(\langle X \rangle_t-\langle X\rangle_{t-\frac1n}\big),
  \end{align*}
so that 
  $$ \langle X\rangle_t = \int_0^t\widehat{a}_sds, \quad t\geq 0, \quad \mathbb{P}\mbox{-}a.s., \quad \mbox{for all } \dbP\in\cP_{loc}. $$
For later use, we observe that, as $\widehat a_t\in\dbS^d_+$, the set of $d\times d$ nonnegative-definite symmetric matrices, we may define a measurable\footnote{\label{inverse}Any matrix $S\in\dbS^d_+$ has a decomposition $S = Q^{\top}_S\Lambda_S Q_S$ for some orthogonal matrix $Q_S$, and a diagonal matrix $\Lambda_S$, with Borel-measurable maps $S\mapsto Q_S$ and $S\mapsto \Lambda_S$, as this decomposition can be obtained by e.g.~the Rayleigh quotient iteration. This implies the Borel measurability of the generalized inverse map $\dbS^d_+\ni S\longmapsto S^{-1}:=Q^{\top} \Lambda^{-1} Q\in \dbS^d_+$, where $\Lambda^{-1}$ is the diagonal element defined by $\Lambda^{-1}_{ii}:=\Lambda^{-1}_{ii}\Ind_{\{\Lambda_{ii}\neq 0\}}$, $i=1,\ldots,d$.} generalized inverse $\widehat a_t^{-1}$, and a measurable square root $\widehat a_t^\frac12=:\widehat \sigma_t$.

\subsection{Spaces and norms}

Let $p>1$ and $\alpha\in\dbR$.

\vspace{3mm}

\noindent {\bf (i)} \textit{One-measure integrability classes:} for a probability measure $\dbP\in\cM_1$, let $\tau$ be an $\dbF^{+, \mathbb{P}}$-stopping time. We denote:

\begin{itemize}
	\item $\dbL^p_{\alpha,\tau}(\dbP)$ is the space of $\dbR$-valued and $\cF^{+, \mathbb{P}}_\tau$-measurable random variables $\xi$, such that 
	        \begin{align*}
	        	\|\xi\|^p_{\dbL^p_{\alpha,\tau}(\dbP)} 
	        	:= 
	        	\dbE^{\dbP}\big[\big|e^{\alpha\tau}\xi\big|^p\big] <\infty. 
	        \end{align*}
	\item $\dbD^p_{\alpha,\tau}(\dbP)$ is the space of $\dbR$-valued, $\dbF^{+, \mathbb{P}}$-adapted processes $Y$ with c\`adl\`ag paths, 
	         such that \footnote{If the stopping time $\tau$ is finite, the norm is indeed 
	         	$\|Y\|^p_{\dbD^p_{\alpha,\tau}(\dbP)} := \dbE^{\dbP}\Big[\sup_{0\leq t\leq \tau}\big|e^{\alpha t}Y_{t}\big|^p\Big] <\infty. $}
         	\begin{align*}
         		\|Y\|^p_{\dbD^p_{\alpha,\tau}(\dbP)} 
         		:= \dbE^{\dbP}\bigg[\sup_{0\leq t<\infty}\big|e^{\alpha (t\wedge\tau)}Y_{t\wedge \tau}\big|^p\bigg] <\infty. 
         	\end{align*}
    \item $\dbH^p_{\alpha,\tau}(\dbP)$ is the space of $\dbR^d$-valued,  $\dbF^{+, \mathbb{P}}$-progressively measurable processes $Z$ such that
            \begin{align*}
            	\|Z\|^p_{\dbH^p_{\alpha,\tau}(\dbP)}
            	:= \dbE^{\dbP}\bigg[\bigg(\int_0^\tau \big|e^{\alpha t}\widehat\sigma_t^{\top}Z_t\big|^2dt\bigg)^{\frac{p}{2}}\bigg] <\infty. 
            \end{align*}
    \item $\dbN^p_{\alpha,\tau}(\dbP)$ is the space of $\dbR$-valued, $\dbF^{+, \mathbb{P}}$-adapted martingales $N$ such that 
            \begin{align*}
            	\|N\|^p_{\dbN^p_{\alpha,\tau}(\dbP)} 
            	:= \dbE^{\dbP}\bigg[\bigg(\int_0^\tau e^{2\alpha t}d[N]_t\bigg)^{\frac{p}{2}}\bigg] <\infty. 
            \end{align*}
    \item $\dbI^p_{\alpha,\tau}(\dbP)$ is the set of scalar $\dbF^{+, \mathbb{P}}$-predictable processes $K$ with c\`adl\`ag nondecreasing paths, s.t.
            \begin{align*}
            	\|K\|^p_{\dbI^p_{\alpha,\tau}(\dbP)} 
            	:= \dbE^{\dbP}\bigg[\bigg(\int_0^\tau e^{\alpha t}dK_t\bigg)^p\bigg] <\infty. 
            \end{align*}
    \item $\dbU^p_{\alpha,\tau}(\dbP)$ is the set of c\`adl\`ag $\dbF$-supermartingales $U$, with Doob-Meyer decomposition $U=N-K$ into the difference of a martingale and a predictable non-decreasing process, such that 
            \begin{align*}
                \|U\|^p_{\dbU^p_{\alpha,\tau}(\dbP)} 
                := \|N\|^p_{\dbN^p_{\alpha,\tau}(\dbP)} + \|K\|^p_{\dbI^p_{\alpha,\tau}(\dbP)} <\infty. 	
            \end{align*}
\end{itemize}

\noindent {\bf (ii)} \textit{Integrability classes under dominated nonlinear expectation:} Let us enlarge the canonical space to $\overline \Om:=\Om\times\Om$ and denote by $(X,W)$ the coordinate process in $\overline \Om$. 
Denote by $\overline\dbF$ the filtration generated by $(X,W)$. For each $\mathbb{P}\in \cP_{loc}$, we may construct a probability measure $\overline\dbP$ on $\overline\Om $ such that $\overline\dbP\circ X^{-1}=\dbP$, $W$ is a $\overline \dbP$-Brownian motion and $dX_t = \widehat\sigma_t dW_t$, $\overline\dbP$-a.s. From now on, we abuse the notation, and keep using $\dbP$ to represent $\overline\dbP$ on $\overline \Om$. Denote by $\cQ_L(\mathbb{P})$ the set of all probability measures $\dbQ^\lambda$ such that
  $$ \rmD_t^{\dbQ^\lambda|\dbP} := \frac{d\dbQ^\lambda}{d\dbP}\bigg|_{\overline\cF_t} 
                                 = \exp{\bigg(\int_0^t\lambda_s\cdot dW_s - \frac{1}{2}\int_0^t|\lambda_s|^2ds\bigg)}, \quad t\geq 0, 
  $$
for some $\dbF^{+, \mathbb{P}}$-progressively measurable process $\lambda=(\lambda)_{t\geq 0}$ uniformly bounded by $L$. By Girsanov's theorem,
 $ W^{\lambda} := W-\int_0^{\cdot}\lambda_s ds $ is a $\dbQ^\lambda$-Brownian motion on any finite horizon, and thus $X^{\lambda}:= X -\int_0^\cdot \widehat\sigma_t \lambda_t dt $ is a $\dbQ^\lambda$-martingale on any finite horizon. For $\dbP\in\cP_{loc}$, we denote 
  \begin{align*}
  	\cE^\dbP[\cdot] := \sup_{\dbQ\in\cQ_L(\dbP)}\dbE^{\dbQ}[\cdot],
  \end{align*}
  and we introduce the subspace $\cL^p_{\alpha, \tau}(\dbP)=\bigcap_{\dbQ\in\cQ_L(\dbP)}\dbL^p_{\alpha, \tau}(\dbQ)$ of random variable $\xi$ such that 
  \begin{align*}
  	\sup_{\dbQ\in\cQ_L(\dbP)}\|\xi\|_{\dbL^p_{\alpha,\tau}(\dbQ)} = \cE^\dbP\big[|e^{\alpha\tau}\xi|^p\big] < \infty.
  \end{align*}
We define similarly the subspaces $\cD^p_{\alpha, \tau}(\dbP)$, $\cH^p_{\alpha, \tau}(\dbP)$, $\cN^p_{\alpha, \tau}(\dbP)$, and the subsets $\cI^p_{\alpha, \tau}(\dbP)$, $\cU^p_{\alpha, \tau}(\dbP)$.

\vspace{3mm}

\noindent {\bf (iii)} {\it Integrability classes under non-dominated nonlinear expectation:} Let $\cP\subseteq\cP_{loc}$ be a subset of probability measures, and denote 
	\begin{align*}
		\cE^\cP[\cdot] := \sup_{\dbP\in\cP}\cE^{\dbP}[\cdot].
	\end{align*}
Let $\dbG:=\{\cG_t\}_{t\geq 0}$ be a filtration with $\cG_t\supseteq\cF_t$ for all $t\geq 0$, so that $\tau$ is also a $\dbG$-stopping time. We define the subspace $\cL^p_{\alpha,\tau}(\cP, \dbG)$ as the collection of all $\cG_\tau$-measurable $\mathbb{R}$-valued random variables $\xi$, such that 
	\begin{align*}
		\|\xi\|^p_{\cL_{\alpha,\tau}^p(\mathcal P)} := \cE^\cP\big[\big|e^{\alpha\tau}\xi\big|^p\big] <\infty. 
	\end{align*}
We define similarly the subspaces $\cD^p_{\alpha, \tau}(\cP, \dbG)$ and $\cH^p_{\alpha, \tau}(\cP, \dbG)$ by replacing $\mathbb{F}^{+, \mathbb{P}}$ with $\mathbb{G}$.

\section{Main results} \label{sect:mainresults}

\subsection{Random horizon backward SDE}

For a probability measure $\dbP\in\cP_{loc}$, an $\mathbb{F}$-stopping time $\tau$, which may be infinite, an $\cF^{+, \mathbb{P}}_\tau$-measurable random variable $\xi$, and a {\it generator} $F:\dbR_+\times\Omega\times\dbR\times\dbR^d\times\dbS^d\longrightarrow\dbR\cup\{\infty\}$, $\prog\otimes\cB(\dbR)\otimes\cB(\dbR^d)\otimes\cB(\dbS^d)$-measurable
\footnote{By $\prog$ we denote the $\sigma$-algebra generated by progressively measurable processes. Consequently, for every fixed $(y,z)\in\dbR\times\dbR^d$, the process $\big(F_t(y,z,\widehat\sigma_t)\big)_{t\geq 0}$ is progressively measurable.}, we set
	\begin{align*}
		f_t(\om,y,z) := F_t\big(\om,y,z,\widehat\sigma_t(\om)\big),
	     \quad 
		(t,\om,y,z)\in\dbR_+\times\Om\times\dbR\times\dbR^d,
	\end{align*}
	and we consider the following backward stochastic differential equation (BSDE): for $t$, $t'\in \mathbb{R}_+$, $t\leq t'$,
	\begin{equation}  \label{bsde1w}
	   \left\{
	      \begin{aligned}
	         Y_{t\wedge \tau} = Y_{t'\wedge\tau}+ \int^{t' \wedge \tau}_{t\wedge\tau} \Big(f_s(Y_{s}, Z_s)ds - Z_s \cdot dX_s - dN_s\Big), \quad \dbP\mbox{-a.s.} \\
	         Y_\tau=\xi\quad {\rm on} \quad \{\tau<\infty\}.
	      \end{aligned}
	  \right.
	\end{equation}

Here, $Y$ is a c\`adl\`ag adapted scalar process, $Z$ is a predictable $\dbR^d$-valued process, and $N$ a c\`adl\`ag $\dbR$-valued martingale with $N_0=0$ orthogonal to $X$, i.e., $[X, N]=0$. We recall that $dX_s=\widehat\sigma_sdW_s$, $\dbP-$a.s.

By freezing the pair $(y,z)$ to $0$, we set $f^0_t:=f_t(0,0)$. 

\vspace{2mm}

\begin{assumption}\label{assum:bsdeF}
	The generator satisfies the following conditions.
	\begin{enumerate}[$(i)$]
		\item $F$ {\rm Lipschitz:} there is a constant $L\geq 0$, such that for all $(y_1, z_1)$, $(y_2, z_2)\in \dbR\times\dbR^d$, $\sigma\in\dbS^d$,
		        \begin{align*}
		        	\big|F_t(y_1, z_1,\sigma)-F_t(y_2, z_2,\sigma)\big| \le L\big(|y_1-y_2|+\big|\sigma^{\top}(z_1-z_2)\big|\big), 
		        	\quad dt\otimes d\dbP\mbox{-a.e.}
		        \end{align*}
		\item $F$ {\rm Monotone:} there is a constant $\mu\in\dbR$, such that for all $z\in \mathbb{R}^d$, $(y_1, y_2)\in \mathbb{R}^2$, $\sigma\in\dbS^d$,
		        \begin{align*}
		        	(y_1-y_2)\big(F_t(y_1, z,\sigma)-F_t(y_2, z,\sigma)\big) \leq -\mu |y_1-y_2|^2, 
		        	\quad dt\otimes d\dbP\mbox{-a.e.}
		        \end{align*}
	\end{enumerate}
\end{assumption}	

\begin{assumption}\label{assum:bsde-integ}
	$\tau$ is a stopping time, $\xi$ is $\cF_\tau-$measurable, and 
	\begin{align*} 
	\Vert\xi{\mathbb 1}_{\{\tau<\infty\}}\Vert_{\cL^q_{\rho,\tau}(\mathbb{P})}<\infty , 
       	\quad \mbox{and} \quad
	  \overline{f}^\dbP_{\rho,q,\tau} := \cE^\dbP\bigg[\bigg(\int_0^\tau \big|e^{\rho t}f_t^0\big|^2 ds\bigg)^{\frac{q}{2}}\bigg]^{\frac{1}{q}}<\infty,
	\end{align*}
	for some $\rho>-\mu,~q>1.$
\end{assumption}	

\vspace{2mm}

\begin{remark}
	In the context of a bounded stopping time $\tau\le T$, the monotonicity assumption can be deduced from  the Lipschitz assumption by the following standard argument. Set $\widetilde Y_t:=e^{\lambda t}Y_t$ and apply It\^o's formula. It is straightforward that the wellposedness of the backward SDE \eqref{bsde1w} is equivalent to a similar wellposedness problem with terminal data $\tilde\xi:=e^{\lambda\tau}\xi$ and nonlinearity 
	  $$ \widetilde F_t(\tilde y,\tilde z,\sigma):=-\lambda\tilde{y}+e^{\lambda t}F_t\big(e^{-\lambda t}\tilde{y},e^{-\lambda t}\tilde{z},\sigma\big). $$ 
	Clearly, $\widetilde F$ inherits the Lipschitz property of $F$, and satisfies the monotonicity condition for sufficiently large $\lambda$. Finally, $\tilde\xi$ is in the same integrability class as $\xi$ for bounded $\tau$. We emphasize that the above mentioned technique applies throughout this paper, and thus when pulling back to the context of finite horizon, the monotonicity assumption could be removed.
	
	However, if one applies the previous argument in the case as $\tau$ is not bounded, then $\tilde \xi$ would fit different integrability condition from $\xi$. Therefore, the monotonicity condition is necessary.
\end{remark}

\vspace{2mm}

\begin{theorem} [Existence and uniqueness] \label{thm:bsde}
	Under Assumptions \ref{assum:bsdeF} and \ref{assum:bsde-integ}, the backward SDE \eqref{bsde1w} has a unique\footnote{The solution is unique modulo the norms of the corresponding spaces.} solution $(Y,Z,N)\in\cD ^{p}_{\eta, \tau}(\mathbb{P})\times\cH^{p}_{\eta, \tau}(\mathbb{P})\times\cN^{p}_{\eta, \tau}(\mathbb{P})$, for all $p\in(1,q)$ and $\eta\in[-\mu,\rho)$, with
	  \begin{equation} \label{estznk}
	  	 \|Y\|^p_{\cD^p_{\eta,\tau}(\dbP)}  + \|Z\|^p_{\cH^p_{\eta,\tau}(\dbP)} + \|N\|^p_{\cN^p_{\eta,\tau}(\dbP)} 
	  	   \leq \mbox{Const}\Big(\|\xi{\mathbb 1}_{\{\tau<\infty\}}\|^p_{\cL^q_{\rho,\tau}(\dbP)} + \big(\overline{f}^\dbP_{\rho,q,\tau}\big)^p\Big). 
	  \end{equation}
\end{theorem}

Except for the estimate \eqref{estznk}, whose proof is reported in Section \ref{sect:specialbsde}, the wellposedness part of the last result is a special case of Theorem \ref{thm:rbsde} below, with obstacle $S\equiv-\infty$.

\vspace{2mm}

\begin{remark}
	The norm, with which we  propose the integrability condition on the coefficients (Assumption \ref{assum:bsde-integ}) and the solution space in Theorem \ref{thm:bsde}, is novel. It is mainly motivated by the following reasons.
	\begin{itemize}
		\item In the initial investigation on the random horizon backward SDE by Peng \cite{Peng91} and Darling \& Pardoux \cite{DP97}, it requires a similar integrability condition as Assumption \ref{assum:bsde-integ} with $\bar\rho:=\rho+L^2/2$ instead of $\rho$ and $\dbE^\dbP$ instead of $\cE^\dbP$. The following Example \ref{eg:condition} illustrates the relevance of our assumption in the simple case of a linear generator.
		In the  works generalizing the result in \cite{DP97}, see e.g.~\cite{BH98,Roy04}, to our knowledge, it is always assumed that $\mu>0$, i.e., the generator is strictly monotone, and the coefficients $\xi, f^0$ are bounded, which is a special case of our Assumption \ref{assum:bsde-integ}. For $\mu=0$, i.e., the generator $f$ is monotone, Royer \cite{Roy04} provided the existence and uniqueness under assumptions that the generator $f$ depending only on $z$ is bounded and $\xi$ is bounded. This result was later generalized by Hu \& Tessitore \cite{HT07}, Briand \& Confortola \cite{BC08} and Papapantoleon et al. \cite{PPS18} to a more general setting. 
		Our Theorem \ref{thm:bsde} generalizes these previous results by allowing for $\mu\leq 0$, thanks to the new norms under which we set up the wellposedness result.
		 \vspace{2mm}
		\item The backward SDE can be viewed as a nonlinear representation of a random variable by an It\^o process with a particular generator function. For the sake of applications, we would like that the representation is a `one-to-one mapping' between the random variable space and the solution space of backward SDE. Here, on the one hand, according to Theorem \ref{thm:bsde}, given {$\xi{\mathbb 1}_{\{\tau<\infty\}}\in \bigcup_{q>1, \rho>-\mu}\cL^q_{\rho,\tau}$,} we may find the solution in $\bigcup_{q>1, \rho>-\mu} \cD ^{q}_{\rho, \tau}(\mathbb{P})\times\cH^{q}_{\rho, \tau}(\mathbb{P})\times\cN^{q}_{\rho, \tau}(\mathbb{P})$. On the other hand, given $Y_0\in \dbR$, $(Z, U)\in \bigcup_{q>1, \rho>-\mu}  \cH^{q}_{\rho, \tau}(\mathbb{P})\times\cN^{q}_{\rho, \tau}(\mathbb{P})$, we may construct an It\^o process (by solving an ODE) such that $Y_\tau{\mathbb 1}_{\{\tau<\infty\}}\in \bigcup_{q>1, \rho>-\mu}\cL^q_{\rho,\tau}$. This builds up the desired one-to-one correspondence. 
	\end{itemize}
	Again, we remind that, unlike in \cite{RTY18}, the application of the new norm in Assumption \ref{assum:bsde-integ} is not to pursue a weaker integrability condition for the wellposedness of backward SDE.
\end{remark}

\vspace{2mm}

\begin{example} \label{eg:condition}
	Let $\dbP:=\dbP_0$, be the Wiener measure on $\Omega$, so that $X$ is a $\dbP_0-$Brownian motion. Let $\tau:={\rm H}_1$, where ${\rm H}_x:=\inf\{t>0\,:\,X_t\geq x\}$, $\xi:=|X_{1\wedge\tau}|$, and $f_t(\omega,y,z):=-\mu y +Lz$ for some constants $0<\mu<1\leq L$. Notice that $f^0=0$, and $\xi\in\cL^2_{0,\tau}(\dbP_0)$ by direct verification:
	  \begin{align*}
	  	\cE^{\dbP_0}\big[|\xi|^2\big] 
	  	  \leq \sup_{\dbQ\in\cQ_L(\dbP_0)}\dbE^{\dbP_0}\Big[\rmD^{\dbQ|\dbP_0}_1|\xi|^2\Big]
	  	  \leq \sup_{\dbQ\in\cQ_L(\dbP_0)}\dbE^{\dbP_0}\Big[\big(\rmD^{\dbQ|\dbP_0}_1\big)^2\Big]^{\frac12}\dbE^{\dbP_0}\big[|\xi|^4\big]^{\frac12}
	  	  <\infty.
	  \end{align*}
	We next show that Darling \& Pardoux's condition is not satisfied. To see this, observe that the event set $A:=\big\{\omega\in\Omega\,:\, \sup_{0\leq t\leq 1}X_t<1, \, X_1\in\left[\tfrac{1}{2},\tfrac{3}{4}\right]\big\}$ satisfies $\dbP_0[A]>0$, and therefore
	\begin{align*}
	  \dbE^{\dbP_0}\big[e^{2L^2\tau}|\xi|^2\big] \ge \frac{1}{4}\dbE^{\dbP_0}\big[e^{2L^2\tau}\Ind_A\big] 
	   &\ge \frac{1}{4}\dbE^{\dbP_0}\Big[\Ind_A \dbE^{\dbP_0}\big[e^{2L^2 {\rm H}_{1-X_1}}\big|X_1\big]\Big] \\
	   &\ge \frac{1}{4}\dbE^{\dbP_0}\Big[\Ind_A \dbE^{\dbP_0}\big[e^{2L^2 {\rm H}_{1/4}}\big]\Big] = \infty.
	\end{align*}
\end{example}

\vspace{2mm}

We also have the following comparison and stability results, which are direct consequences of Theorem \ref{thm:rbsdecomp-stab} below, obtained by setting the obstacle to $-\infty$ therein, together with the estimate \eqref{estznk} in Theorem \ref{thm:bsde}. 
	
\vspace{2mm}	
	
\begin{theorem}\label{thm:bsdecomp-stab}
	Let $(f,\xi)$, $(f',\xi')$ be two sets of parameters satisfying the conditions of Theorem \ref{thm:bsde} with some stopping time $\tau$, and the corresponding solutions $(Y,Z,N)$, $(Y',Z',N')$.
	\begin{enumerate}[{\bf (i)}]
		\item {\rm Stability}. Denoting $\delta\xi:=\xi-\xi'$, $\delta Y:=Y-Y'$, $\delta Z:=Z-Z'$, $\delta U:=U-U'$ and $\delta f= f-f'$, we have for all $1<p<p'< q$ and $-\mu<\eta<\eta'<\rho$:
		\begin{align*}
		\|\delta Y\|^p_{\cD^p_{\eta,\tau}(\dbP)} 
		\le C_{p,p',\eta}\bigg\{\|\delta\xi{\mathbb 1}_{\{\tau<\infty\}}\|^p_{\cL^{p'}_{\eta,\tau}(\dbP)} 
		+ \cE^{\dbP}\bigg[\bigg(\int_0^\tau \big|e^{\eta t}\delta f_t(Y_t,Z_t)\big|dt\bigg)^{p'}\bigg]^{\frac{p}{p'}}\bigg\} 
		\end{align*}
		and
		\begin{align*}
	 	\|\delta Z\|^p_{\cH^p_{\eta,\tau}(\dbP)}\! + \!\|\delta N\|^p_{\cN^p_{\eta,\tau}(\dbP)} 
		\le
		C_{p,\eta,\eta'} \bigg\{\|\delta Y\|^p_{\cD^p_{\eta',\tau}(\dbP)}
		\!+\!\cE^{\dbP}\bigg[\bigg(\int_0^\tau \big|e^{\eta t}\delta f_t(Y_t,Z_t)\big|dt\bigg)^{p}\bigg]\bigg\}. 
		\end{align*} 
		\item {\rm Comparison}. Assume $\xi \leq \xi'$, $\dbP$-a.s. on $\{\tau<\infty\}$, and $f(y, z) \leq f'(y, z)$ for all $(y, z) \in \dbR\times\dbR^d$, $dt\otimes\dbP$-a.e. Then, $Y_{\tau_0}\leq Y'_{\tau_0}$, $\dbP$-a.s. for all finite stopping time $\tau_0\leq\tau$, $\dbP$-a.s. 
	\end{enumerate}
\end{theorem}

\vspace{2mm}

\begin{remark} \label{rem:supersolution}
	Following \cite{EPQ97} we say that $(Y,Z)$ is a supersolution (resp.~subsolution) of the BSDE with parameters $(f,\xi)$ if the martingale $N$ in \eqref{bsde1w} is replaced by a supermartingale (resp.~submartingale).
	A direct examination of the proof of the last comparison result reveals that the conclusion is unchanged if $(Y,Z)$ is a subsolution of BSDE$(f,\xi)$, and $(Y',Z')$ is a supersolution of BSDE$(f',\xi')$.
\end{remark}

\subsection{Random horizon reflected backward SDE}

We now consider an obstacle defined by $(S_t)_{t\geq 0}$, and we search for a representation similar to \eqref{bsde1w} with the additional requirement that $Y\ge S$. This is achieved at the price of pushing up the solution $Y$ by substracting a supermartingale $U$ with minimal action. We then consider the following reflected backward stochastic differential equation (RBSDE): for $t$, $t'\in \mathbb{R}_+$, $t\leq t'$,
   \begin{equation} \label{RBSDEeq1}
      \left\{
        \begin{aligned}
           Y_{t \wedge\tau} = Y_{t'\wedge\tau} +\int_{t\wedge \tau}^{t' \wedge \tau}\Big(f_s(Y_s,Z_s)ds - Z_s\cdot dX_s - dU_s\Big), \quad 
              Y \ge S,\quad \dbP\mbox{-a.s.} \\
           \dbE^\dbP\bigg[\int_0^{t\wedge\tau} 1\wedge \big((Y_{r-}-S_{r-})\big)dU_r\bigg] = 0, \quad \mbox{for all}~~t\ge 0, \\
           Y_\tau=\xi\quad {\rm on}\quad \{\tau<\infty\}.
        \end{aligned}
      \right.
   \end{equation} 
where $U_{\wedge t}$ is a c\`adl\`ag $\dbP$-supermartingale, for all $t\ge 0$, starting from $U_0=0$, orthogonal to $X$, i.e.~$[X,U]=0$. The last minimality requirement is the so-called Skorokhod condition.\footnote{This condition coincides the standard Skorokhod condition in the literature. Indeed, by using the corresponding Doob-Meyer decomposition $U=N-K$ into a martingale $N$ and a nondecreasing process $K$, and recalling that $Y\ge S$ , it follows that  $0=\dbE^\dbP\big[\int_0^{\tau\wedge t} \big(1\wedge(Y_{r-}-S_{r-})\big)dU_r\big] = \dbE^\dbP\big[-\int_0^{\tau\wedge t}\big(1\wedge (Y_{r-}-S_{r-})\big)dK_r\big]$ is equivalent to $\int_0^\tau (Y_{r-}-S_{r-})dK_r=0$, $\dbP-$a.s. by the arbitrariness of $t\ge 0$.}

\vspace{2mm}

\begin{theorem}[Existence and uniqueness] \label{thm:rbsde}
	Let Assumptions \ref{assum:bsdeF} and \ref{assum:bsde-integ} hold true, and let $S$ be a c\`adl\`ag $\dbF^{+,\dbP}$-adapted process with $\|S^+\|_{\cD^q_{\rho,\tau}(\dbP)}<\infty$. Then, the reflected backward SDE \eqref{RBSDEeq1} has a unique solution $(Y,Z,U)\in\cD^{p}_{\eta,\tau}(\dbP)\times\cH^{p}_{\eta,\tau}(\dbP)\times\cU^{p}_{\eta,\tau}(\dbP)$, for all $p\in(1,q)$ and $\eta\in[-\mu,\rho)$. 
\end{theorem}

\vspace{2mm}

The existence part of this result is proved in Section \ref{sect:rbsde-existence}. The uniqueness is a consequence of claim \textbf{(i)} of the following stability and comparison results.

\vspace{2mm}

\begin{theorem} \label{thm:rbsdecomp-stab}
	Let $(f,\xi,S)$ and $(f',\xi',S')$ be two sets of parameters satisfying the conditions of Theorem \ref{thm:rbsde}, with corresponding solutions $(Y,Z,U)$ and $(Y',Z',U')$.
	\begin{enumerate}[{\bf (i)}]
		\item {\rm Comparison.} Assume $\xi\le\xi'$, $\dbP$-a.s. on $\{\tau<\infty\}$, $f(y,z)\leq f'(y,z)$ for all $(y,z)\in\dbR\times\dbR^d$, and $S\leq S'$, $dt\otimes\dbP$-a.e. Then, $Y_{\tau_0}\leq Y'_{\tau_0}$, $\dbP$-a.s., for all finite stopping time $\tau_0\leq\tau$, $\dbP$-a.s.
		\item {\rm Stability.} Let $S=S'$, and denote $\delta\xi:=\xi-\xi'$, $\delta Y:=Y-Y'$, $\delta Z:=Z-Z'$, $\delta U:=U-U'$ and $\delta f= f-f'$. Then, for all $1<p<p'< q$ and $-\mu\leq\eta<\eta'<\rho$, we have:
		\begin{align*}
		& \|\delta Y\|^p_{\cD^p_{\eta,\tau}(\dbP)} 
		+ \|\delta Z\|^p_{\cH^p_{\eta,\tau}(\dbP)} 
		+ \|\delta U\|^p_{\cN^p_{\eta,\tau}(\dbP)} 
		\\
		& \hspace{1mm}\leq C_{p,p',\eta,\eta'}
		\Big\{\Delta_\xi  + \Delta_f  \\
		& \hspace{21mm} +\Big(\Delta_\xi^{\frac12}  + \Delta_f^{\frac12} \Big)
		\Big(\big(\overline{f}^\dbP_{\eta',p,\tau}\big)^{\frac{p}{2}}
		+ \big(\overline{f'}^\dbP_{\eta',p,\tau}\big)^{\frac{p}{2}}
		+ \|Y\|_{\cD^p_{\eta',\tau}(\dbP)}^{\frac{p}{2}} 
		+ \|Y'\|_{\cD^p_{\eta',\tau}(\dbP)}^{\frac{p}{2}}
		\Big)
		\Big\} 
		\end{align*}
		where 
		\begin{align*}
			\Delta_\xi:=\big\|\delta\xi{\mathbb 1}_{\{\tau<\infty\}}\big\|^p_{\cL^{p'}_{\eta',\tau}(\dbP)} 
			 \quad \mbox{and} \quad 
			\Delta_f:=\cE^{\dbP}\bigg[\bigg(\int_0^\tau e^{\eta' s}\big|\delta f_s(Y_s,Z_s)\big|ds\bigg)^{p'}\bigg]^{\frac{p}{p'}}.
		\end{align*}
	Moreover, $\delta\overline{U}:=\int_0^{\cdot\wedge \tau} e^{\eta s}d\delta U_s$ satisfies
		\begin{align*}
	       \big\|\delta\overline{U}\big\|_{\cD_{0,\tau}^p}^p
	        \leq C_{p,L,\eta,\eta'}\Big(\|\delta Y\|^p_{\cD^p_{\eta',\tau}(\dbP)} + \|\delta Z\|^p_{\cH^p_{\eta',\tau}(\dbP)} +\Delta_f\Big).
		\end{align*}
	\end{enumerate}
\end{theorem}

The proof of \textbf{(ii)} is reported in Section \ref{sect:rbsde-stability}, while \textbf{(i)} is proved at the end of Section \ref{sect:rbsde-existence}.

Notice that the stability result is incomplete as the differences $\delta Y$, $\delta Z$ and $\delta U$ are controlled by the norms of $Y$ and $Y'$. However, in contrast with the estimate \eqref{estznk} in the backward SDE context, we have unfortunately failed to derive a similar control of $(Y,Z,U)$ by the ingredients $\xi$, $f^0$ and $S$ in the present context of random horizon reflected backward SDE due to the presence of the orthogonal martingale $N$ in the general filtration, see also \cite{BPTZ18}.

\subsection{Random horizon second order backward SDE}

Following Soner, Touzi \& Zhang \cite{STZ12}, we introduce second order backward SDE as a family of backward SDEs defined on the supports of a convenient family of singular probability measures. For this reason, we introduce the subset of $\cP_{loc}$:
	\begin{equation} \label{cP0}
		\cP_0 = \Big\{\dbP\in\cP_{loc}: f^0_t(\omega)<\infty, 
		\mbox{ for Leb}\!\otimes\!\dbP\mbox{-a.e.}~(t,\omega)\in\dbR_+\times\Omega\Big\},
	\end{equation}
  where we recall that $f^0_t(\omega)=F_t\big(\omega,0,0,\widehat\sigma_t(\omega)\big)$. 
Note that in the context of stochastic control, which is the major application of second order backward SDE, the set $\cP_0$ defined above is the set of all admissible controls of volatility. We also define for all finite stopping times $\tau_0$:
   \begin{equation*}
   	  \cP_\dbP(\tau_0) := \big\{\dbP'\in\cP_0:~\dbP'=\dbP~\mbox{on}~\cF_{\tau_0}\big\},
   	    \quad \mbox{and} \quad
   	  \cP_\dbP^+(\tau_0) := \bigcup_{h>0}\cP_\dbP\big(\tau_0+h\big).
   \end{equation*}
We remark that the definition of $\cP_\dbP^+(\tau_0)$ differs slightly from the one in \cite{STZ12,STZ13}, in which the authors studied second order backward SDEs under the extra uniform continuity condition.

For a finite $\mathbb{F}$-stopping time $\tau$, the second order backward SDE (2BSDE, hereafter) is defined by
	\begin{equation}  \label{2bsdel}
		Y_{t\wedge\tau} = \xi + \int_{t\wedge\tau}^\tau\Big(F_s (Y_s,  Z_s,\widehat \sigma_s )ds - Z_s \cdot dX_s - dU_s\Big),
		 \quad \cP_0\mbox{-q.s.}
	\end{equation}
  for some supermartingale $U$ together with a convenient minimality condition.

\vspace{2mm}

\begin{definition} \label{def:2bsde}
	Let $p>1$ and $\eta\in\dbR$.
	A process $(Y,Z)\in\cD ^p_{\eta,\tau}\big(\mathcal P_0, \mathbb{F}^{+,\cP_0}\big)\times \mathcal H^p_{\eta,\tau}\big(\mathcal P_0,\mathbb{F}^{\cP_0}\big)$ is said to be a solution of the 2BSDE \eqref{2bsdel}, if for all $\dbP\in\cP_0$, the process
	 \begin{align*}
	 	U^\dbP_{t\wedge\tau} := Y_{t\wedge\tau}- Y_0 + \int_0^{t\wedge\tau}\Big(F_s (Y_s,  Z_s,\widehat \sigma_s )ds - Z_s \cdot dX_s\Big),
	 	  \quad t\geq 0, \quad \dbP\mbox{-a.s.}
	 \end{align*}
	is a c\`adl\`ag $\dbP$-local supermartingale starting from $U^\dbP_0=0$, orthogonal to $X$, i.e.~$[X,U^\dbP]=0$, $\dbP-$a.s.~and satisfying the minimality condition
	 \begin{align*}
	 	U^\dbP_{s\wedge\tau} = \esssup^\dbP_{\dbP'\in\cP_\dbP^+(s\wedge\tau)}\dbE^{\dbP'}\big[ U^{\dbP'}_{t\wedge\tau} \big| \cF^{+,\dbP'}_{s\wedge\tau}\big],
	 	~~\dbP\mbox{-a.s.} \quad \mbox{for all }~0\leq s\leq t.
	 \end{align*}
\end{definition}

\vspace{2mm}

\begin{remark}
	Notice that the last definition relaxes slightly \eqref{2bsdel} by allowing for a dependence of $U$ on the underlying probability measure. This dependence is due to the fact that the stochastic integral $Z\sint X:=\int_0^\cdot Z_s\cdot dX_s$ is defined $\dbP-$a.s. under all $\dbP\in\cP_0$, and should rather be denoted by $(Z\sint X)^\dbP$ in order to emphasize the $\dbP-$dependence. 
	
	By Theorem 2.2 in Nutz \cite{Nut12}, the family $\{(Z\sint X)^\dbP\}_{\dbP\in \cP_0}$ can be aggregated as a medial limit $(Z\sint X)$ under the acceptance of Zermelo-Fraenkel set theory with axiom of choice together with the continuum hypothesis into our framework. In this case, $(Z\sint X)$ can be chosen as an $\mathbb{F}^{+, \cP_0}$-adapted process, and the family $\{U^\dbP\}_{\dbP\in\cP_0}$ can be aggregated into the resulting medial limit $U$, i.e., $U=U^\dbP$, $\dbP-$a.s. for all $\dbP\in\cP_0$. 
\end{remark}

The following assumption requires the additional notations:
  $$ \xi^{t,\om}(\om') := \xi(\om\otimes_t\om'),
       \quad 
     f^{0,t,\om}_s(\om') := F_{t+s}\big(\om\otimes_t\om',0,0,\widehat\sigma_s(\om')\big),
       \quad 
     \vec{\tau}^{t,\om}:=\tau^{t,\om}-t,
  $$
  which involve the paths concatenation operator 
   $$ (\om\otimes_t\om')_s :=\Ind_{\{s\le t\}}\om_s+\Ind_{\{s>t\}}\big(\om_t+\om'_{s-t}\big), $$ 
  and
   \begin{align*}
   	  \cP(t,\om) := \Big\{\dbP\in\cP_{loc}:~f^{0,t,\om}_s(\om')<\infty, \quad \mbox{for~~Leb}\otimes\dbP\mbox{-a.e.}~(s,\om')\in\dbR_+\times\Om\Big\},
   \end{align*}
  so that $\cP_0=\cP(0,{\bf 0})$. 
  
 \vspace{2mm} 
  
\begin{assumption} \label{assum:2bsde-integ}
	We assume that
	\begin{enumerate}[$(i)$]
		\item $\tau$ is a stopping time with 
		        $$ \displaystyle \lim_{n\to\infty}\!\!\cE^{\cP_0}\big[\Ind_{\{\tau\geq n\}}\big] = 0, $$ 
		      $\xi$ is  $\cF_\tau-$measurable, and there are constants $\rho>-\mu$, and $q>1$ such that
		      \begin{align*}
		      	\big\|\xi\big\|_{\cL^q_{\rho,\tau}(\cP_0)} < \infty, \quad \mbox{and} \quad 
		      	\overline{F}^0_{\rho,q,\tau} := \cE^{\cP_0}\bigg[\bigg(\int_0^\tau\big|e^{\rho t}f_t^0\big|^2dt\bigg)^{\frac{q}{2}}\bigg]^{\frac1q} < \infty.
		      \end{align*}
		\item Furthermore, the following dynamic version of $(i)$ holds for all $(t, \omega)\in \llbracket 0, \tau\rrbracket$:
		      \begin{align*}
		      	\big\|\xi^{t, \omega}\big\|_{\mathcal L^q_{\rho,\vec{\tau}^{t,\om}}(\mathcal{P}(t, \omega))} < \infty, 
		      	 \quad \mbox{and} \quad
		      	\overline{F}_{\rho,q}^{0,t,\omega} 
		      	:= \cE^{\cP(t,\omega)}\bigg[\bigg(\int_0^{\vec{\tau}^{t,\om}}\big|e^{\rho s}f^{0,t,\om}_s\big|^2ds \bigg)^{\frac{q}{2}} \bigg]^{\frac1q} 
		      	< \infty.
		      \end{align*}
    \end{enumerate}
\end{assumption}

\vspace{2mm}

\begin{theorem} \label{mainthm}
	Under Assumptions \ref{assum:bsdeF} and \ref{assum:2bsde-integ} $(i)$, the 2BSDE \eqref{2bsdel} has at most one solution $(Y,Z)\in \mathcal D^p_{\eta,\tau}\big(\mathcal P_0, \mathbb{F}^{+, \cP_0}\big)\times \mathcal H^p_{\eta,\tau}\big(\mathcal P_0, \mathbb{F}^{\cP_0}\big)$, for all $p\in(1,q)$ and $\eta\in[-\mu,\rho)$, with
	 \begin{equation} \label{apriori-2bsde}
	 	\|Y\|^p_{\cD^p_{\eta,\tau}(\cP_0)}+ \|Z\|^p_{\cH^p_{\eta,\tau}(\cP_0)} 
	 	\leq C_{p,q,\eta,\rho}\Big(\|\xi\|^p_{\cL^q_{\rho,\tau}(\cP_0)} + \big(\overline{F}^0_{\rho,q,\tau}\big)^p\Big).
	 \end{equation}
	Under the additional Assumption \ref{assum:2bsde-integ} $(ii)$, such a solution $(Y,Z)$ for the 2BSDE \eqref{2bsdel} exists.

	If $\cP_0$ is saturated\footnote{We say that the family $\cP_0$ is saturated if, for all $\dbP\in\cP_0$, we have $\dbQ\in\cP_0$ for every probability measure $\dbQ\sim\dbP$ on $(\Omega,\cF)$ such that $X$ is $\dbQ-$local martingale. The assertion follows by the same argument as in \cite[Theorem 5.1]{PTZ15}.}, then $U^\dbP$ is a $\dbP$-a.s. non-increasing process for all $\dbP\in\cP_0$.
\end{theorem}

\vspace{2mm}

Similar to Soner, Touzi \& Zhang \cite{STZ12}, the following comparison result for second order backward SDEs is a by-product of our construction; the proof is provided in Proposition \ref{prop:representation}.

\vspace{2mm}

\begin{proposition}\label{prop:2BSDEcomp}
	Let $(Y,Z)$ and $(Y',Z')$ be solutions of 2BSDEs with parameters $(F,\xi)$ and $(F',\xi')$, respectively, which satisfy Assumptions \ref{assum:bsdeF} and \ref{assum:2bsde-integ}. Suppose further that $\xi\leq\xi'$ and $F_t\big(y,z,\widehat\sigma_t\big)\le F'_t\big(y,z,\widehat\sigma_t\big)$ for all $(y,z)\in\dbR\times\dbR^d$, $dt\otimes\cP_0$-q.s. Then, we have $Y\leq Y'$, $dt\otimes\cP_0$-q.s.~on $\llbracket 0,\tau\rrbracket$.
\end{proposition}

\section{Wellposedness of random horizon reflected BSDEs}  \label{sect:rbsde}

Throughout this section, we fix a probability measure $\dbP\in\cP_{loc}$, and we omit the dependence on $\dbP$ in all of our notations. We also observe that $\cQ_L:=\cQ_L(\dbP)$ is stable under concatenation. 

For all $\dbQ^\lambda\in\cQ_L$, it follows from Girsanov's Theorem that 
\begin{itemize}
	\item $W^\lambda:=W-\int_0^\cdot \lambda_sds$ is a $\dbQ^\lambda$-Brownian motion, $X^\lambda:= X-\int_0^\cdot\widehat\sigma_s\lambda_sds$ is a $\dbQ^\lambda$-local martingale, and we may rewrite the RBSDE as 
	       \begin{align*}
	       	  dY_t = -f^\lambda_t(Y_t,Z_t)dt + Z_t\cdot dX^\lambda_t+ dU_t, \quad \mbox{where} \quad 
	       	   f^\lambda_t(y,z):=f_t(y,z)-\widehat\sigma^{\top}_tz\cdot\lambda_t
	       \end{align*}
        	satisfies the Assumption \ref{assum:bsdeF} with Lipschitz coefficient $2L$.
   	\item $U$ remains a $\dbQ^\lambda$-supermartingale, with the same Doob-Meyer decomposition as under $\dbP$.
\end{itemize}

\subsection{Auxiliary inequalities}

We first state a Doob-type inequality. For simplicity, we write $\cE[\cdot]:= \cE^\dbP[\cdot]$.

\begin{lemma}  \label{supDoob}
	Let $(M_t)_{0\leq t\leq \tau}$ be a uniformly integrable martingale under some $\widehat{\dbQ}\in\cQ_L$. 
	Then, 
	\begin{align*}
		\cE\bigg[\sup_{0\leq t\leq \tau}|M_t|^p\bigg] 
		\leq \frac{q}{q-p}\big(\cE\big[|M_\tau|^q\big]\big)^{\frac{p}{q}}, \quad \mbox{for all}~~0<p<q.
	\end{align*}
\end{lemma}

\begin{proof}
	Let $x>0$ and $T_x^n:=\tau\wedge n\wedge \inf\{t\geq 0,\, |M_t|>x\},$ with the convention $\inf\emptyset =\infty$. From the definition of concatenation and the optional sampling theorem, we obtain for all $\dbQ\in\cQ_L$:
	  \begin{align*}
	  	\dbE^\dbQ\big[|M_{T_x^n}|^q\big] 
	  	  &= \dbE^\dbQ\Big[\big|\dbE^{\widehat\dbQ}\big[M_{\tau}\big|\cF_{T_x^n}\big]\big|^q\Big] 
	  	   \le \dbE^{\dbQ}\Big[\dbE^{\widehat\dbQ}\big[|M_{\tau}|^q\big|\cF_{T_x^n}\big]\Big] \\
	  	  & = \dbE^{\dbQ\otimes_{T_x^n}\widehat\dbQ}\big[|M_{\tau}|^q\big]
	  	    \le \cE\big[|M_{\tau}|^q\big]=:c,
	  \end{align*}
	as $\dbQ\otimes_{T_x^n}\widehat\dbQ\in\cQ_L$. Then, denoting $M_*:=\sup_{0\leq t\leq \tau}|M_t|$, we see that
	\begin{align*}
	  x^q\dbQ\left[M_*>x\right] \le x^q\dbQ[T_x\le \tau]  
	   &= \lim_{n\rightarrow \infty} x^q\dbQ[T_x^n\le \tau] \\
	   & \leq \lim_{n\rightarrow \infty} \dbE^\dbQ\big[|M_{T_x^n}|^q\Ind_{\{T_x^n\le \tau\}}\big] 
	     \leq \lim_{n\rightarrow\infty}\dbE^\dbQ\big[|M_{T_x^n}|^q\big]\leq c, 
	\end{align*}
	and we deduce that
	\begin{align*}
	   \dbE^\dbQ\big[M_*^p\big] 
	    & = \dbE^\dbQ\bigg[\int_0^\infty  \Ind_{\{M_*>x\}}p x^{p-1}dx\bigg]
	      = \int_0^\infty  \dbQ\left[M_*>x\right]p x^{p-1}dx  \\
	    & \le \int_0^\infty  [1\wedge (cx^{-q})]p x^{p-1}dx 
	      = \frac{qc^{\frac{p}{q}}}{q-p}. 
	\end{align*}
	The required inequality follows from the arbitrariness of $\dbQ\in\cQ_L$. 
\end{proof}

\vspace{3mm}

The following result is well-known, we report its proof for completeness as we could not find a reference for it. 
We shall denote $\sgn(x):=\Ind_{\{x>0\}}-\Ind_{\{x<0\}}$, for all $x\in\dbR$.

\vspace{2mm}

\begin{proposition}  \label{TanakaIneq}
	For any semimartingale $X$, we have 
	  $$ |X_t| - |X_0| \geq \int_0^t\sgn(X_{s-})dX_s, \quad t\ge 0. $$
\end{proposition}

\begin{proof}
	Consider a decreasing sequence of $C^2$, symmetric convex functions $\varphi_n$ on $\dbR$, such that $\varphi_n(x)=|x|$ on $(-\frac{1}{n^2},\frac{1}{n^2})^c$, 
	and $\varphi'_n(x)$ increases to $1$ for $x>0$ and  $\varphi'_n(x)$ decreases to $-1$ for $x<0$, i.e., $\varphi'_n(x)$ converges to $\sgn(x)$. 
	By It\^o's formula and convexity of $\varphi_n$, we obtain that 
	  \begin{align*}
	  	\varphi_n(X_t)-\varphi_n(X_0) 
	  	 = \int_0^t\varphi'_n(X_{s-})dX_s &+ \frac{1}{2}\int_0^t\varphi''_n(X_{s-})d[X^c]_s \\
	  	  &+ \sum_{0<s\leq t}\big\{\Delta\varphi_n(X_s)-\varphi'_n(X_{s-})\Delta X_s\big\}.
	  \end{align*}
	By convexity of $\varphi_n$, this implies that $\varphi_n(X_t)-\varphi_n(X_0)\ge\int_0^t\varphi'_n(X_{s-})dX_s$.  The required inequality follows by sending $n\to\infty$ in the above inequality and by applying the dominated convergence theorem for stochastic integrals (see, e.g., \cite[Section IV, Theorem 32]{Pro05}). 
\end{proof}

\subsection{A priori estimates}

\begin{proposition}  \label{EstimationRBSDE}
	Under the conditions of Theorem \ref{thm:rbsde}, let $(Y,Z,U)\in \cD^p_{\beta,\tau}\times\cH^p_{\beta,\tau}\times\cU^p_{\beta,\tau}$ be a solution of RBSDE \eqref{RBSDEeq1}. 
	For each $p\in(1,q)$ and $-\mu\leq\alpha<\beta<\rho$, there exists a constant $C_{p,L,\alpha,\beta}$ such that 
	  \begin{align*}
	  	 \|Z\|^p_{\cH^p_{\alpha,\tau}} + \|U\|^p_{\cU^p_{\alpha,\tau}} 
	  	   \leq C_{p,L,\alpha,\beta} \Big(\big(\overline{f}^\dbP_{\beta,p,\tau}\big)^p+ \|Y\|^p_{\cD^p_{\beta,\tau}}\Big).  
	  \end{align*}
\end{proposition}

\begin{proof}  
	Let $U=N-K$ be the Doob-Meyer decomposition of the supermartingale $U$. 
	
	\vspace{3mm}
	
	\noindent {\bf 1.} We first prove that
	  \begin{align} \label{estimationKeq5}
	    \|Z\|^p_{\dbH^p_{\alpha,\tau}(\dbQ^\lambda)} +\|N\|^p_{\dbN^p_{\alpha,\tau}(\dbQ^\lambda)} 
	     \leq C_p\Big(\big\|Z\sint X^\lambda+U\big\|^p_{\dbN^p_{\alpha,\tau}(\dbQ^\lambda)} +\|K\|^p_{\dbI^p_{\alpha,\tau}(\dbQ^\lambda)}\Big),
   	  \end{align}
	 and 
	  \begin{align}  \label{estimationKeq6}
	    \widetilde{c}_p \Big(\|Z\|^p_{\dbH^p_{\alpha,\tau}(\dbQ^\lambda)} + \|U\|^p_{\dbN^p_{\alpha,\tau}(\dbQ^\lambda)}\Big)
	      & \leq \big\|Z\sint X^\lambda+U\big\|^p_{\dbN^p_{\alpha,\tau}(\dbQ^\lambda)} \nonumber  \\
	      & \leq \widetilde{C}_p\Big(\|Z\|^p_{\dbH^p_{\alpha,\tau}(\dbQ^\lambda)} + \|U\|^p_{\dbN^p_{\alpha,\tau}(\dbQ^\lambda)}\Big).
	  \end{align}
	We only prove \eqref{estimationKeq5}, the second claim follows by similar arguments.
	
	As $[X^\lambda, N]=\widehat\sigma\sint [W^\lambda,N]=0$, we obtain that
	  \begin{align*}
	     \|Z\|^p_{\dbH^p_{\alpha,\tau}(\dbQ^\lambda)} +\|N\|^p_{\dbN^p_{\alpha,\tau}(\dbQ^\lambda)}
	    & \leq c_p\dbE^{\dbQ^\lambda}\bigg[\bigg(\int_0^\tau \!\!e^{2\alpha s}\big(d\big[Z\!\sint\! X^{\lambda}\big]_s +d[N]_s\big) \bigg)^{\frac{p}{2}}\bigg] \\
    	& = c_p\big\|Z\sint X^\lambda + N\big\|^p_{\dbN^p_{\alpha,\tau}(\dbQ^\lambda)}.
	  \end{align*}
	We continue by estimating the right hand side term:
	  \begin{align*}
	    &\hspace{-3mm}
	      \big\|Z\sint X^\lambda + N\big\|^p_{\dbN^p_{\alpha,\tau}(\dbQ^\lambda)} \\
	    & \leq 2^{\frac{p}{2}}\dbE^{\dbQ^\lambda}\bigg[\bigg(\int_0^\tau e^{2\alpha s}d\big[Z\sint X^{\lambda}+U\big]_s 
	                   + \int_0^\tau e^{2\alpha s}d[K]_s\bigg)^{\frac{p}{2}}\bigg] \\
	    & \leq 2^p\bigg(\dbE^{\dbQ^\lambda}\bigg[\bigg(\int_0^\tau  e^{2\alpha s}d\big[Z\sint X^{\lambda}+U\big]_s\bigg)^{\frac{p}{2}}\bigg]
	                   + \dbE^{\dbQ^\lambda}\bigg[\bigg(\int_0^\tau e^{2\alpha s}d[K]_s\bigg)^{\frac{p}{2}}\bigg]\bigg) \\
	    & \leq 2^p\bigg(\dbE^{\dbQ^\lambda}\bigg[\bigg(\int_0^\tau e^{2\alpha s}d\big[Z\sint X^{\lambda}+U\big]_s\bigg)^{\frac{p}{2}}\bigg]
	                   + \dbE^{\dbQ^\lambda}\bigg[\bigg(\int_0^\tau e^{\alpha s}dK_s\bigg)^{p}\bigg]\bigg) \\
	    & = 2^p\Big(\big\|Z\sint X^\lambda+U\big\|^p_{\dbN^p_{\alpha,\tau}(\dbQ^\lambda)}+\|K\|^p_{\dbI^p_{\alpha,\tau}(\dbQ^\lambda)}\Big),
	  \end{align*}
	 where we used the estimate 
	  \begin{align*}
	  	\int_0^\tau e^{2\alpha s}d[K]_s 
	  	 = \sum_{0<s\leq\tau}e^{2\alpha s}(\Delta K_s)^2 \!\le\! \left(\sum_{0<s\leq\tau}e^{\alpha s}(\Delta K_s)\right)^2
	  	 \leq \left(\int_0^\tau e^{\alpha s}dK_s\right)^2,
	  \end{align*}
	 since $K$ is non-decreasing.  
	
	\vspace{3mm}
	
	\noindent {\bf 2.} 
	Denote $U^\lambda:=Z\sint X^\lambda + U=\sigma^{\top}Z\sint W^\lambda + U$. 
	By It\^o's formula, for $t'\in \mathbb{R}_+$, 
	  \begin{align*}
	  	0 \le Y^2_0 
	  	  &= e^{2\alpha (t' \wedge \tau)}Y_{t' \wedge \tau}^2 \\
	  	  &\quad + \int_0^{t' \wedge\tau} e^{2\alpha t}\Big(-2\alpha Y_t^2dt
	  	    +2Y_{t-}\big(f^\lambda_t(Y_t, Z_t)dt-dU^\lambda_t\big) -\big|\widehat\sigma_t^{\top}Z_t\big|^2dt -d[U]_t \Big).
	  \end{align*}
	It follows from Assumption \ref{assum:bsdeF} and Young's inequality that 
	  \begin{align*}
	  	 2yf^\lambda_t(y,z) \leq -2\mu y^2 + 2|y||f^0_t| + 4L|y||\widehat\sigma_t^{\top}z| \le -2\mu y^2+|f^0_t|^2
	  	    +\ell |y|^2+\frac12\big|\widehat\sigma_t^{\top}z\big|^2,
	  \end{align*}
	  with $\ell:=1+8L^2$. 
	Then, as $\alpha + \mu \geq 0$,
	  \begin{align*}
	     & \hspace{-3mm}\int_0^{t' \wedge \tau} e^{2\alpha t}\Big(\frac12\big|\widehat\sigma_t^{\top}Z_t\big|^2dt +d[U]_t\Big) \\
	     &\leq  e^{2\alpha (t' \wedge \tau)}Y_{t' \wedge \tau}^2 + \int_0^{t' \wedge \tau} e^{2\alpha t}\big(\big|f^0_t\big|^2+\ell Y_t^2\big)dt 
	                                                             - 2\int_0^{t' \wedge \tau} e^{2\alpha t}Y_{t-}dU^\lambda_t  \\
	     &\leq \int_0^{t' \wedge \tau} e^{2\alpha t}\big|f^0_t\big|^2dt 
                	+ \Big(1+\frac{\ell}{2(\alpha'-\alpha)}\Big)\sup_{0\leq t<\infty}e^{2\alpha'(t\wedge \tau)}Y^2_{t\wedge \tau} 
	                - 2\int_0^{t' \wedge \tau} e^{2\alpha t}Y_{t-}dU^\lambda_t,
	  \end{align*}
	  for an arbitrary $\alpha'\in(\alpha,\rho)$. 
	Let $t' \rightarrow \infty$. 
	It follows that
	  \begin{equation} \label{estimationKeq7}
	     \begin{aligned}
	         &\hspace{-3mm}\|Z\|^p_{\dbH^p_{\alpha,\tau}(\dbQ^\lambda)} + \|U\|^p_{\dbN^p_{\alpha,\tau}(\dbQ^\lambda)} \\
	         & \leq C_{p,\alpha,\alpha',L} \bigg(\dbE^{\dbQ^\lambda}\bigg[\bigg(\int_0^\tau \big|e^{\alpha t}f^0_t\big|^2dt\bigg)^{\frac{p}{2}}\bigg] 
	              + \|Y\|^p_{\dbD^p_{\alpha',\tau}(\dbQ^\lambda)}	+ E^\lambda \bigg),   
	     \end{aligned}
 	  \end{equation}
  	  where
	  \begin{align*}
	    E^\lambda
	      &:= \dbE^{\dbQ^\lambda}\bigg[\bigg|\int_0^\tau e^{2\alpha t}Y_{t-}dU^\lambda_t\bigg|^{\frac{p}{2}}\bigg]  \\
	      &\,\le \overline{C}_p\bigg(\dbE^{\dbQ^\lambda}\bigg[\sup_{0\leq t\leq \tau}\bigg|\int_0^t e^{2\alpha s}Y_{s-}\big(Z_s\sint dX_s^\lambda+dN_s\big)\bigg|^{\frac{p}{2}} + \bigg|\int_0^\tau e^{2\alpha s}Y_{s-}dK_s\bigg|^{\frac{p}{2}}\bigg]\bigg)  \\
	      &\,\leq C'_p\bigg(\dbE^{\dbQ^\lambda}
	           \bigg[\bigg(\int_0^\tau e^{4\alpha s}Y_{s-}^2d\big[Z\sint X^\lambda+ N\big]_s\bigg)^{\frac{p}{4}}\bigg]
             	+\dbE^{\dbQ^\lambda}\bigg[\bigg|\int_0^\tau e^{2\alpha s}Y_{s-}dK_s\bigg|^{\frac{p}{2}} \bigg] \bigg),
	  \end{align*}
	by the BDG inequality. Since $K$ is non-decreasing, we applying Young's inequality with an arbitrary $\varepsilon>0$ to deduce
		\begin{align*}
		  E^\lambda 
		    &\le C'_p\dbE^{\dbQ^\lambda} \bigg[\sup_{0\le s<\infty} |e^{\alpha' (s\wedge\tau)}Y_{s\wedge \tau}|^{\frac{p}{2}}
		                     \bigg\{\bigg(\int_0^\tau e^{2(2\alpha -\alpha')s}d\big[Z\sint X^\lambda+N\big]_s\bigg)^{\frac{p}{4}} \\
		    & \hspace{64mm} +\bigg(\int_0^\tau e^{(2\alpha-\alpha')s}dK_s\bigg)^{\frac{p}{2}}\bigg\}\bigg] \\
		    &\le \frac{(C'_p)^2}{\varepsilon}\|Y\|^p_{\dbD^p_{\alpha',\tau}(\dbQ^\lambda)} 
		      + \frac{\varepsilon}{2}\dbE^{\dbQ^\lambda}\bigg[\bigg(\int_0^\tau e^{2(2\alpha-\alpha')s}d\big[Z\sint X^\lambda + N\big]_s\bigg)^{\frac{p}{2}}\bigg] \\
		    & \hspace{33.5mm} + \frac{\varepsilon}{2}\dbE^{\dbQ^\lambda}\bigg[\bigg(\int_0^\tau e^{(2\alpha-\alpha')s}dK_s\bigg)^p\bigg]  \\
		    &\leq \frac{(C'_p)^2}{\varepsilon}\|Y\|^p_{\dbD^p_{\alpha',\tau}(\dbQ^\lambda)}
		      + \frac{\varepsilon}{2}(C_p''+1)\|K\|^p_{\dbI^p_{2\alpha-\alpha',\tau}(\dbQ^\lambda)}  \\
		    & \hspace{33.5mm} + \frac{\varepsilon}{2}C_p''\dbE^{\dbQ^\lambda}\bigg[\bigg(\int_0^\tau e^{2(2\alpha-\alpha')s}d\big[U^\lambda\big]_s\bigg)^{\frac{p}{2}}\bigg],
		\end{align*}
     where the last inequality follows from \eqref{estimationKeq5}. Plugging this estimate into \eqref{estimationKeq7}, and using \eqref{estimationKeq6} together with the fact that $2\alpha - \alpha'<\alpha$, we obtain 
	\allowdisplaybreaks
	\begin{equation} \label{estimationKeq9} 
		\begin{aligned} 
		   & \hspace{-3mm}\Big(1-\frac{\overline{C}_pC_{p,\alpha,\alpha',L}C_p''}{2}\varepsilon\Big) \big\|U^\lambda\big\|_{\dbN^p_{\alpha,\tau}(\dbQ^\lambda)} \\
		   & \leq \overline{C}_pC_{p,\alpha,\alpha',L} \bigg(\dbE^{\dbQ^{\lambda}} \bigg[\bigg(\int_0^\tau \!\!\!\big|e^{\alpha s}f^0_s\big|^2ds\bigg)^{\frac{p}{2}}\bigg]   
		               +\bigg(1+\frac{(C'_p)^2}{\varepsilon}\bigg) \|Y\|^p_{\dbD^p_{\alpha',\tau}(\dbQ^\lambda)}  \\
		   &\hspace{68.5mm} +\frac{\varepsilon}{2}(C''_p\!+\!1) \|K\|^p_{\dbI^p_{2\alpha-\alpha',\tau}(\dbQ^\lambda)} \bigg).  
		\end{aligned}
	\end{equation}

	\vspace{3mm}
	
	\noindent {\bf 3.} We shall prove in Step 4 below that for $\delta<\delta'<\rho$:
	   \begin{align} \label{estimationKeq4}
	   	 \|K\|^p_{\dbI^p_{\delta,\tau}(\dbQ^\lambda)} 
	   	 \leq \overline{C}^K_{p,\delta,\delta', L} \bigg(\|Y\|^p_{\dbD^p_{\delta',\tau}(\dbQ^\lambda)} + \|Z\|^p_{\dbH^p_{\delta',\tau}(\dbQ^\lambda)} 
	   	           + \dbE^{\dbQ^\lambda}\bigg[\bigg(\int_0^\tau \big|e^{\delta' s}f^0_s\big|^2ds\bigg)^{\frac{p}{2}}\bigg]\bigg).
	   \end{align}
	Plugging this inequality with $\delta:=2\alpha-\alpha'$ and $\delta':=\alpha$ in \eqref{estimationKeq9}, and using the left hand side inequality of \eqref{estimationKeq6}, we see that we may choose $\varepsilon>0$ conveniently such that 
	   \begin{align}  \label{estimationKeq11}
	   	  \|Z\|^p_{\dbH^p_{\alpha,\tau}(\dbQ^\lambda)}
	   	    \leq C^Z_{p,\alpha,\alpha', L}\bigg(\|Y\|^p_{\dbD^p_{\alpha',\tau}(\dbQ^\lambda)} 
	   	          + \dbE^{\dbQ^\lambda}\bigg[\bigg(\int_0^\tau e^{2\alpha s}|f^0_s|^2ds\bigg)^{\frac{p}{2}}\bigg]\bigg),
	   \end{align}
	 for some constant $C^Z_{p,\alpha,\alpha', L}>0$. 
    Plugging this inequality into \eqref{estimationKeq4} with $(\delta,\delta'):=(\alpha,\alpha')$ induces the estimate 
       \begin{align}  \label{estimationKeq13}
       	  \|K\|^p_{\dbK^p_{\alpha,\tau}(\dbQ^\lambda)}  
       	  \leq C^K_{p,\alpha,\alpha', L}\bigg(\|Y\|^p_{\dbD^p_{\alpha',\tau}(\dbQ^\lambda)} 
       	           +\dbE^{\dbQ^\lambda}\bigg[\bigg(\int_0^\tau e^{2\alpha' s}\big|f^0_s\big|^2ds\bigg)^{\frac{p}{2}}\bigg]\bigg),
       \end{align} 
	  for some constant $C^K_{p,\alpha,\alpha', L}$. 
	Combining with \eqref{estimationKeq9}, and recalling that $2\alpha-\alpha'<\alpha$, in turn, this implies an estimate for $\big\|U^\lambda\big\|^p_{\dbN^p_{\alpha,\tau}(\dbQ^\lambda)}$ which can be plugged into \eqref{estimationKeq5} to provide: 
	   \begin{align}  \label{estimationKeq14}
	   	  \|N\|^p_{\dbN^p_{\alpha,\tau}(\dbQ^\lambda)} 
	   	  \leq C^N_{p,\alpha,\alpha', L}\bigg(\|Y\|^p_{\dbD^p_{\alpha',\tau}(\dbQ^\lambda)} 
	   	        + \dbE^{\dbQ^\lambda}\bigg[\bigg(\int_0^\tau e^{2\alpha' s}\big|f^0_s\big|^2ds\bigg)^{\frac{p}{2}}\bigg]\bigg).
	   \end{align}
	Since the constants in \eqref{estimationKeq11}, \eqref{estimationKeq13} and \eqref{estimationKeq14} do not depend on $\dbQ\in\cQ_L$, the proof of this proposition is completed by taking supremum over the family of measures $\dbQ\in\cQ_L$. 
	
	\vspace{3mm}
	
	\noindent {\bf 4.} We now prove \eqref{estimationKeq4}. By It\^o's formula, we have 
	   \begin{align*}
	   	 e^{\delta (t\wedge \tau)}Y_{t\wedge \tau}+ \int_0^{t\wedge \tau} e^{\delta s}\big(f^\lambda_s(Y_s,Z_s)-\delta Y_s\big)ds 
	   	  = Y_0 + \int_0^{t\wedge \tau} e^{\delta s}(Z_s\cdot dX^\lambda_s+dU_s).
	   \end{align*}
	As $(Z,N)\in \dbH^p_{\delta,\tau}(\dbQ^\lambda)\times\dbN^p_{\delta,\tau}(\dbQ^\lambda)$ and $K$ is nondecreasing, the process 
	   \begin{equation*}
	   	  e^{\delta (t\wedge \tau)}Y_{t\wedge \tau} + \int_0^{t\wedge\tau} e^{\delta s}\big(f^\lambda_s(Y_s,Z_s)-\delta Y_s\big)ds,
	   \end{equation*}
  	   is a supermartingale under $\dbQ^\lambda$. 
	By \cite[Lemma A.1]{BPTZ18} and Assumption \ref{assum:bsdeF}, we obtain that
	\allowdisplaybreaks
	 \begin{align}  \label{estimationKeq1}
	  \dbE^{\dbQ^\lambda} \bigg[\bigg(\int_0^\tau e^{\delta s}dK_s\bigg)^p\bigg] 
	   &\leq C_p \dbE^{\dbQ^\lambda}\bigg[\sup_{0\leq u<\infty}\big(e^{\delta (u\wedge \tau)}Y_{u\wedge\tau} 
	             +\int_0^{u\wedge\tau} e^{\delta s}\big(f^\lambda_s(Y_s,Z_s)-\delta Y_s\big)ds\big)^p \bigg] \nonumber \\
	   &\leq C_{p,\delta, L}\dbE^{\dbQ^\lambda}\bigg[\sup_{0\leq u<\infty}|e^{\delta (u\wedge \tau)}Y_{u\wedge \tau}|^p 
	             + \bigg(\int_0^\tau e^{\delta s}\big|f^0_s\big|ds\bigg)^p \nonumber \\
	   &\hspace{25mm} + \bigg(\int_0^\tau e^{\delta s}|Y_s|ds\bigg)^p + \bigg(\int_0^\tau e^{\delta s}|\widehat\sigma_s^{\top}Z_s|ds\bigg)^p \bigg]. 
	 \end{align}
	Finally, for $\delta'\in(\delta,\rho)$, we observe that
	  \begin{align}  \label{estimationKeq2}
	    \bigg(\int_0^\tau \!\! e^{\delta s}|Y_s|ds\bigg)^p
	      &\leq \sup_{0\leq s<\infty}|e^{\delta' (s\wedge \tau)}Y_{s\wedge \tau}|^p \bigg(\int_0^\tau \!\! e^{-(\delta'-\delta)s}ds\bigg)^p \nonumber \\
	      &\leq \frac{1}{(\delta'-\delta)^p}\sup_{0\leq s<\infty}|e^{\delta' (s\wedge \tau)}Y_{s\wedge \tau}|^p,
	  \end{align}
	  and by the Cauchy-Schwarz inequality
	  \begin{align} \label{estimationKeq3}
	     \bigg(\int_0^\tau e^{\delta s}\big|f^0_s\big|ds\bigg)^p
	       &\leq \bigg(\int_0^\tau \big|e^{\delta' s}f^0_s\big|^2ds\bigg)^{\frac{p}{2}}\bigg(\int_0^\tau e^{-2(\delta'-\delta)s}ds\bigg)^{\frac{p}{2}} \nonumber \\
	       &\leq \frac{1}{(2\delta'-2\delta)^{\frac{p}{2}}}\bigg(\int_0^\tau \big|e^{\delta' s}f_s^0\big|^2ds\bigg)^{\frac{p}{2}},
	  \end{align}
	  and, similarly,
	  \begin{equation} \label{estimationKeq3.5}
	     \bigg(\int_0^\tau e^{\delta s}\big|\widehat\sigma_s^{\top}Z_s\big|ds\bigg)^p
	       \leq \frac{1}{(2\delta'-2\delta)^{\frac{p}{2}}}\bigg(\int_0^\tau \big|e^{\delta' s}\widehat\sigma_s^{\top}Z_s\big|^2ds\bigg)^{\frac{p}{2}} .
	  \end{equation} 
	The required inequality \eqref{estimationKeq4} follows from \eqref{estimationKeq1}, \eqref{estimationKeq2}, \eqref{estimationKeq3} and \eqref{estimationKeq3.5}.   
\end{proof}

\vspace{2mm}

\subsection{Stability of reflected backward SDEs} \label{sect:rbsde-stability}

\begin{proof}[Proof of Theorem \ref{thm:rbsdecomp-stab} {\bf (ii)}]
	Clearly, the process $(\delta Y, \delta Z,\delta U)$ satisfies the following equation 
	  \begin{align}  \label{dynamics-deltaY}
	  	 \delta Y_{t\wedge\tau} 
	  	   = \delta Y_{t' \wedge \tau} + \int_{t\wedge\tau}^{t' \wedge \tau} g_s(\delta Y_s,\delta Z_s)ds - \delta Z_s\cdot dX_s - d\delta U_s,\quad t\leq t',
	  \end{align}
	  where $g_s(\delta Y_s,\delta Z_s):= f_s(Y_s, Z_s) - f_s(Y_s-\delta Y_s,Z_s-\delta Z_s). $
	
	\vspace{3mm}              
	
	\noindent {\bf 1.} In this step, we prove that, for some constant $C_{p,p'}$, 
	   \begin{equation}  \label{estimationDYeq1}
	    \begin{aligned} 
	  	  &\cE\bigg[\sup_{0\leq t<\infty}e^{p\eta' (t\wedge \tau)}|\delta Y_{t\wedge\tau}|^p\bigg]  \\
	  	  & \qquad\leq C_{p,p'}\cE\bigg[e^{p'\eta'\tau}\big|\delta\xi {\mathbb 1}_{\{\tau<\infty\}}\big|^{p'}
	  	      + \bigg(\int_{0}^{\tau} e^{\eta' s}|\delta f_s(Y_s,Z_s)|ds\bigg)^{p'}\bigg]^{\frac{p}{p'}}.
	    \end{aligned}
	   \end{equation}
	It follows from Proposition \ref{TanakaIneq} that
	    \begin{equation}  \label{RBSDEPro1eq1}
	      \begin{aligned}
	         & e^{\eta'(t' \wedge\tau)}|\delta Y_{t' \wedge \tau}| - e^{\eta' (t\wedge\tau)}|\delta Y_{t\wedge\tau}|  \\
	         & \quad \geq\int_{t\wedge\tau}^{t' \wedge \tau} e^{\eta' s}\big(\eta'|\delta Y_s| - \sgn(\delta Y_{s})g_s(\delta Y_s,\delta Z_s)\big)ds \\
	         & \quad\qquad + \int_{t\wedge\tau}^{t' \wedge \tau} e^{\eta' s}\big(\sgn(\delta Y_{s})\delta Z_s\cdot dX_s + \sgn(\delta Y_{s-})d\delta U_s\big).
	      \end{aligned}
	    \end{equation}
	As $f$ and $f'$ satisfy Assumption \ref{assum:bsdeF}, we obtain that 
	  \begin{equation*}
	  	 \sgn(\delta Y_{s})g_s(\delta Y_s,\delta Z_s) \leq |\delta f_s(Y_s,Z_s)| + L\big|\widehat\sigma^{\top}_s\delta Z_s\big| -\mu|\delta Y_s|. 
	  \end{equation*}
	Considering the Doob-Meyer decomposition $U=N-K$ and $U'=N'-K'$, and denoting $\delta N$ and $\delta K$ the corresponding differences, it follows from the Skorokhod condition that
	  \begin{align}
	     \delta Y_{s-} d\delta K_s 
	       &= (Y'_{s-}-Y_{s-})(dK'_s-dK_s) \nonumber \\
	       &= (Y'_{s-}-S_{s-})dK'_s -(Y'_{s-}-S_{s-})dK_s - (Y_{s-}-S_{s-})dK'_{s}  + (Y_{s-}-S_{s-})dK_{s} \nonumber \\
	       &= -(Y'_{s-}-S_{s-})dK_s - (Y_{s-}-S_{s-})dK'_s \le 0,   \label{RBSDEPro1eq4}
	  \end{align}
	  so that 
	  \begin{align*}
	  	 \sgn(\delta Y_{s-})d\delta K_s = \Ind_{\{\delta Y_{s-}\neq 0\}}\frac{\sgn(\delta Y_{s-})}{\delta Y_{s-}}\delta Y_{s-} d\delta K_s \leq 0.
	  \end{align*}
	Then, denoting 
	  \begin{align*}
	  	{\widehat\lambda}_s:=L\sgn(\delta Y_s)\frac{\widehat\sigma^{\top}_s\delta Z_s}{|\widehat\sigma^{\top}_s\delta Z_s|}\Ind_{\{|\widehat\sigma^{\top}_s\delta Z_s|\neq 0\}} 
	  	 \quad \mbox{and} \quad
	  	X^{\widehat\lambda}:=X-\int_0^.\widehat\sigma_s\widehat\lambda_sds,
	  \end{align*}
	   it follows from inequality \eqref{RBSDEPro1eq1} and $-\mu<\eta'$ that 
	  \begin{align*}
	     e^{\eta' (t\wedge\tau)}|\delta Y_{t\wedge\tau}| 
	       &\leq e^{\eta'\tau}|\delta Y_{t' \wedge \tau}| + \int_{t\wedge\tau}^{t' \wedge \tau} e^{\eta' s} |\delta f_s(Y_s,Z_s)|ds \\
	       & \hspace{22mm} - \int_{t\wedge\tau}^{t' \wedge \tau} e^{\eta' s} \Big(\sgn(\delta Y_s)\delta Z_s\cdot dX^{\widehat\lambda}_s + \sgn(\delta Y_{s-})d\delta N_s\Big).
	  \end{align*}
	As $\delta Z\!\in\!\cH^p_{\eta,\tau}(\dbP)$ and $\delta N\!\in\!\cN^p_{\eta,\tau}(\dbP)$, we deduce from the BDG inequality that the last two terms are $\dbQ^{\widehat\lambda}-$uniformly integrable martingales. Then, with $\tau_n:=n\wedge\tau$ and $n\geq t$:
	  \begin{align*}
	  	 e^{\eta'(t\wedge\tau)}|\delta Y_{t\wedge\tau}|
	  	  \leq \lim_{n\to\infty}\dbE^{\dbQ^{\widehat\lambda}}\bigg[e^{\eta'\tau_n}|\delta Y_{\tau_n}| 
	  	           + \int_{t\wedge\tau}^{\tau_n} e^{\eta' s}|\delta f_s(Y_s,Z_s)|ds\bigg|\cF_{t\wedge\tau}^{+,\dbP}\bigg].
	  \end{align*}
	For any $\eta''\in (\eta', \rho)$, $||\delta Y||_{\cD^p_{\eta'',\tau}}<\infty$, so that $\sup_{n\in\mathbb{N}}e^{\eta''\tau_n}\delta Y_{\tau_n}<\infty$, $\dbQ^{\widehat\lambda}$-a.s., which implies
	  \begin{align*}
	     \lim_{n\rightarrow \infty}e^{\eta' \tau_n}|\delta Y_n|
	       &= \lim_{n\rightarrow \infty}e^{\eta' \tau_n}\big|\delta Y_{\tau_n}{\mathbb 1}_{\{\tau<\infty\}}\big|
	           +\lim_{n\rightarrow \infty}e^{\eta' \tau_n}\big|\delta Y_{\tau_n}{\mathbb 1}_{\{\tau=\infty\}}\big| \\
	       &\leq \lim_{n\rightarrow \infty}e^{\eta' \tau_n}\big|\delta Y_{\tau_n}{\mathbb 1}_{\{\tau<\infty\}}\big|
	           +\lim_{n\rightarrow \infty}e^{(\eta' -\eta'')\tau_n}\big|e^{\eta''\tau_n}\delta Y_{\tau_n}{\mathbb 1}_{\{\tau=\infty\}}\big| \\
	       &= \lim_{n\rightarrow \infty}e^{\eta' \tau_n}\big|\delta Y_{\tau_n}{\mathbb 1}_{\{\tau<\infty\}}\big|
	        = e^{\eta' \tau}\big|\delta Y_{\tau}{\mathbb 1}_{\{\tau<\infty\}}\big|.
	  \end{align*}
	From the dominated convergence theorem and monotone convergence theorem and the fact that $e^{\eta't}Y_t$ and $e^{\eta't}Y'_t$ are uniformly integrable.
	  \begin{align*}
	  	 e^{\eta'(t\wedge\tau)}|\delta Y_{t\wedge\tau}| 
	  	  \leq \dbE^{\dbQ^{\widehat\lambda}}\bigg[e^{\eta'\tau}|\delta\xi{\mathbb 1}_{\{\tau<\infty\}}| 
	  	        + \int_{0}^{\tau} e^{\eta' s}|\delta f_s(Y_s,Z_s)|ds\,\bigg|\cF_{t\wedge\tau}^{+,\dbP}\bigg].
	  \end{align*}
	By Lemma \ref{supDoob}, we deduce that for any $p'\in(p,q)$:
	  \begin{align*}
	  	 \cE\bigg[\sup_{0\leq t<\infty}\big|e^{\eta' t}\delta Y_{t\wedge \tau}\big|^p\bigg] 
	  	  \leq \frac{p'}{p'-p}\cE\bigg[\bigg(e^{\eta'\tau}|\delta\xi {\mathbb 1}_{\{\tau<\infty\}}| 
	  	         + \int_{0}^{\tau} e^{\eta' s}|\delta f_s(Y_s,Z_s)|ds\bigg)^{p'}\bigg]^{\frac{p}{p'}},
	  \end{align*}
	  which induces the required inequality \eqref{estimationDYeq1}. 
	
	\vspace{3mm}
	
	\noindent {\bf 2.} Let $-\mu\leq\eta<\eta'$. By It\^o's formula, we have for $t'\in \mathbb{R}_+$,
	  \begin{equation*}
	    \begin{aligned}
	      & \hspace{-5mm} e^{2\eta(t' \wedge \tau)}(\delta Y_{t' \wedge \tau})^2 - (\delta Y_0)^2 \\
	      &= 2\eta\int_0^{t' \wedge \tau} e^{2\eta s}(\delta Y_s)^2ds - 2\int_0^{t' \wedge \tau} e^{2\eta s}\delta Y_sg_s(\delta Y_s,\delta Z_s)ds  \\
	      &\qquad + 2\int_0^{t' \wedge \tau} e^{2\eta s}\delta Y_{s-}\delta Z_s\cdot dX_s + 2\int_0^{t' \wedge \tau} e^{2\eta s}\delta Y_{s-}d\delta N_s  \\
	      &\qquad - 2\int_0^{t' \wedge \tau} e^{2\eta s}\delta Y_{s-}d\delta K_s + \int_0^{t' \wedge \tau} e^{2\eta s}\left(\big|\widehat\sigma^{\top}_s\delta Z_s\big|^2ds+d[\delta U]_s\right).
	    \end{aligned}
	  \end{equation*}
	Again Assumption \ref{assum:bsdeF} implies that 
	  \begin{align*}
	  	 \delta Y_s g_s(\delta Y_s,\delta Z_s)\leq |\delta Y_s||\delta f_s(Y_s,Z_s)|+L|\delta Y_s|\big|\widehat\sigma^{\top}_s\delta Z_s\big| - \mu|\delta Y_s|^2, 
	  \end{align*}
 	  and therefore, together with \eqref{RBSDEPro1eq4} and the fact that $\eta+\mu\geq0$, we obtain that
	  \begin{align*}
	      & \hspace{-5mm} \int_0^{t' \wedge \tau}e^{2\eta s} \big(\big|\widehat\sigma^{\top}_s\delta Z_s\big|^2ds + d[\delta U]_s\big) \\
	        &\leq e^{2\eta (t' \wedge \tau)}(\delta Y_{t' \wedge \tau})^2 
	             - 2\int_0^{t' \wedge \tau} e^{2\eta s}\delta Y_{s-}\big(\delta Z_s\cdot dX_s + d\delta N_s\big) \\
	        &\hspace{31mm} + 2\int_0^{t' \wedge \tau} e^{2\eta s}|\delta Y_s|\big(|\delta f_s(Y_s,Z_s)|+L\big|\widehat\sigma^{\top}_s\delta Z_s\big|\big)ds.
	  \end{align*}
	Then,  
	  \begin{align*}
	     & \hspace{-5mm} \int_0^{\tau} e^{2\eta s} \big(\big|\widehat\sigma^{\top}_s\delta Z_s\big|^2ds + d[\delta U]_s\big) \\
 	     & \leq \sup_{0\leq t< \infty} e^{2\eta (t\wedge \tau)}(\delta Y_{t\wedge \tau})^2 
	          + 2\int_0^{\tau} e^{2\eta s}|\delta Y_s|\big(|\delta f_s(Y_s,Z_s)| + 2L\big|\widehat\sigma^{\top}_s\delta Z_s\big|\big)ds \\
	     & \hspace{39mm} + 2\sup_{0\leq u\leq \tau}\bigg|\int_0^ue^{2\eta s}\delta Y_s\delta Z_s\cdot dX^{\lambda}_s 
	              +\int_0^ue^{2\eta s}\delta Y_{s-}d\delta N_s\bigg|,
	  \end{align*}
	  where $\lambda=(\lambda_s)_{0\leq s\leq\tau}$ is an arbitrary process uniformly bounded by $L$.  By Young's inequality and the fact that $\eta<\eta'$, we have
	  \begin{align*}
	     2\int_0^\tau e^{2\eta s}|\delta Y_s||\delta f_s(Y_s,Z_s)|ds 
	      &\leq 2\Big(\sup_{0\leq s< \infty}e^{\eta (s\wedge \tau)}|\delta Y_{s\wedge \tau}|\Big)\int_0^\tau e^{\eta s}|\delta f_s(Y_s,Z_s)|ds \\ 
	      &\leq \sup_{0\leq s<\infty}e^{2\eta (s\wedge \tau)}|\delta Y_{s\wedge \tau}|^2 + \bigg(\int_0^\tau e^{\eta s}|\delta f_s(Y_s,Z_s)|ds\bigg)^2,
	  \end{align*}
	  and \allowdisplaybreaks
	  \begin{align*}
	     & 2\int_0^\tau e^{2\eta s}|\delta Y_s|\big|\widehat\sigma^{\top}_s\delta Z_s\big|ds \\
	     &\qquad \leq \frac{1}{\varepsilon}\int_0^\tau e^{2\eta s}|\delta Y_s|^2ds + \varepsilon\int_0^\tau e^{2\eta s}\big|\widehat\sigma^{\top}_s\delta Z_s\big|^2ds \\
	     &\qquad \leq \frac{1}{\varepsilon}\Big(\sup_{0\leq s< \infty} e^{2\eta' (s\wedge \tau)}|\delta Y_{s\wedge\tau}|^2\Big) \int_0^\tau e^{-2(\eta'-\eta)s}ds
	            + \varepsilon\int_0^\tau e^{2\eta s}\big|\widehat\sigma^{\top}_s\delta Z_s\big|^2ds  \\
	     &\qquad \leq \frac{1}{2\varepsilon(\eta'-\eta)}\sup_{0\leq s< \infty}e^{2\eta' (s\wedge \tau)}|\delta Y_{s\wedge \tau}|^2
	            + \varepsilon\int_0^\tau e^{2\eta s}\big|\widehat\sigma^{\top}_s\delta Z_s\big|^2ds,
	  \end{align*}
	  for an arbitrary $\varepsilon>0$. 
	Therefore, by choosing an $\varepsilon>0$ conveniently, we obtain
	  \begin{equation}  \label{deltaZetNK}
	   \begin{aligned} 
	     &\hspace{-5mm} \big\|\delta Z\big\|^p_{\dbH^p_{\eta,\tau}(\dbQ^\lambda)} + \big\|\delta U\big\|^p_{\dbN^p_{\eta,\tau}(\dbQ^\lambda)} \\
	     &\leq C_{p,\eta,\eta'}\bigg(\big\|\delta Y\big\|^p_{\dbD^p_{\eta',\tau}(\dbQ^\lambda)} 
	           + \dbE^{\dbQ^\lambda}\bigg[\bigg(\int_0^\tau e^{\eta s}|\delta f_s(Y_s,Z_s)|ds\bigg)^p\bigg]\bigg) \\
	     & \qquad + C_{p,\eta,\eta'}\dbE^{\dbQ^\lambda} \bigg[\sup_{0\leq u\leq \tau}\bigg|\int_0^u e^{2\eta s}\big(\delta Y_{s}\delta Z_{s}\cdot dX^{\lambda}_s 
	                       + \delta Y_{s-}d\delta N_s\big)\bigg|^{\frac{p}{2}}\bigg].
	   \end{aligned}
	  \end{equation}
	  for some constant $C_{p,\eta,\eta'}>0$. By the BDG inequality, Young's inequality and the Cauchy-Schwarz inequality, we obtain 
	  \begin{align*}
	     & \hspace{-3mm} \dbE^{\dbQ^\lambda} \bigg[\sup_{0\leq u\leq \tau}\Big|\int_0^u\!e^{2\eta s}\delta Y_{s-}(\delta Z_s\cdot dX^{\lambda}_s +d\delta N_s)\bigg|^{\frac{p}{2}}\bigg] \\
	     &\leq d_p\dbE^{\dbQ^\lambda} \bigg[\bigg[\int_0^\cdot\! e^{2\eta s}\delta Y_{s-} (\delta Z_s\cdot dX^{\lambda}_s +d\delta N_s)\bigg]_\tau^{\frac{p}{4}}\bigg] \\
	     &= d_p\dbE^{\dbQ^\lambda} \bigg[\bigg(\int_0^\tau e^{4\eta s}|\delta Y_{s-}|^2 \big(\big|\widehat\sigma^{\top}_s\delta Z_s\big|^2ds 
	        + d[\delta N]_s\big)\bigg)^{\frac{p}{4}}\bigg] \\
	     &\leq d'_p\dbE^{\dbQ^\lambda}\bigg[\bigg(\int_0^\tau e^{4\eta s}|\delta Y_s|^2\big|\widehat\sigma^{\top}_s\delta Z_s\big|^2ds\bigg)^{\frac{p}{4}} 
	             + \bigg(\int_0^\tau e^{4\eta s}|\delta Y_{s-}|^2d[\delta N]_s\bigg)^{\frac{p}{4}}\bigg] \\
	     &\leq \frac{d'_p}{2\varepsilon'} \big\|\delta Y\big\|^p_{\dbD^p_{\eta,\tau}(\dbQ^\lambda)} 
	            + \frac{d'_p\varepsilon'}{2}\big\|\delta Z\big\|^p_{\dbH^p_{\eta,\tau}(\dbQ^\lambda)} \\
	     &\hspace{30mm} + d'_p\big\|\delta Y\big\|^{\frac{p}{2}}_{\dbD^p_{\eta,\tau}(\dbQ^\lambda)}\Big(\big\|N\big\|^{\frac{p}{2}}_{\dbN^p_{\eta,\tau}(\dbQ^\lambda)}
	            + \big\|N'\big\|^{\frac{p}{2}}_{\dbN^p_{\eta,\tau}(\dbQ^\lambda)}\Big), 
	  \end{align*}
	  for some $\varepsilon'>0$, where we used $d[\delta N]_s \leq 2(d[N]_s+d[N']_s)$. 
	Plugging this estimate into \eqref{deltaZetNK}, and by the arbitrariness of $\lambda$, we obtain 
	  \begin{align*}
	    \big\|\delta Z\big\|^p_{\cH^p_{\eta,\tau}} + \big\|\delta U\big\|^p_{\cN^p_{\eta,\tau}} 
	      &\leq C'_{p,\eta,\eta'}\Big\{\big\|\delta Y\big\|^p_{\cD^p_{\eta',\tau}} + \cE\Big[\Big(\int_0^\tau e^{\eta s}|\delta f_s(Y_s,Z_s)|ds\Big)^p\Big] \\
	      & \hspace{27mm} + \big\|\delta Y\big\|^{\frac{p}{2}}_{\cD^p_{\eta',\tau}}\Big(\big\|N\big\|^{\frac{p}{2}}_{\cN^p_{\eta,\tau}} 
	                      + \big\|N'\big\|^{\frac{p}{2}}_{\cN^p_{\eta,\tau}} \Big)\Big\}.
	\end{align*}
	Together with \eqref{estimationDYeq1} from Step 1, and Proposition \ref{EstimationRBSDE}, this induces the first estimate in Theorem \ref{thm:rbsdecomp-stab} {\bf (ii)}.
	
	\vspace{3mm}
	
	\noindent {\bf 3.} It remains to verify the announced estimate on $\int_0^te^{\alpha s}d\delta U_s$. 
	Given the dynamics of $\delta Y$ in \eqref{dynamics-deltaY}, it follows from a direct application of It\^o's formula and the use of Assumption \ref{assum:bsdeF} that:
	 \begin{align*}
	    &\hspace{-5mm}\sup_{0\leq u\leq \tau}\bigg|\int_0^{u}e^{\alpha s}d\delta U_s\bigg| \\
	    &\leq 2\sup_{0\leq s<\infty} e^{\alpha (s\wedge \tau)}|\delta Y_{s\wedge \tau}| 
	           + (|\alpha|+L)\int_0^\tau e^{\alpha s}|\delta Y_s|ds + 2L\int_0^\tau e^{\alpha s}\big|\widehat\sigma^{\top}_s\delta Z_s\big|ds \\
	    & \qquad + \int_0^\tau e^{\alpha s}\big|\delta f_s(Y_s,Z_s)\big|ds + \sup_{0\leq u\leq \tau}\bigg|\int_0^{u}e^{\alpha s}\delta Z_s\cdot dX^\lambda_s\bigg|.
	 \end{align*}
	By the BDG inequality and the Cauchy-Schwarz inequality, we obtain for $\beta\in(\alpha,\rho)$:
	\allowdisplaybreaks
	  \begin{align*}
	     & \dbE^{\dbQ^\lambda}\Big[\sup_{0\leq u\leq \tau}\Big|\int_0^{u}e^{\alpha s}d\delta U_s\Big|^p\Big] \\
	     & \quad\leq C_p \bigg(\dbE^{\dbQ^\lambda} \bigg[\sup_{0\leq s<\infty} \big|e^{\alpha (s\wedge \tau)}\delta Y_{s\wedge\tau}\big|^p \bigg] \\
	     & \hspace{18mm} + (|\alpha|+L)^p \bigg(\int_0^\infty e^{-(\beta-\alpha)s}ds \bigg)^p \dbE^{\dbQ^\lambda} \bigg[\sup_{0\leq s<\infty} e^{p\beta (s\wedge\tau)}|\delta Y_{s\wedge\tau}|^p \bigg]  \\
	     & \hspace{18mm} + (2L)^p \bigg(\int_0^\infty e^{-2(\beta-\alpha)s}ds \bigg)^{\frac{p}{2}} \dbE^{\dbQ^\lambda} \bigg[ \bigg(\int_0^\tau e^{2\beta s}\big|\widehat\sigma^{\top}_s\delta Z_s\big|^2ds\bigg)^{\frac{p}{2}}\bigg] \\
	     & \hspace{18mm} + \dbE^{\dbQ^\lambda} \bigg[ \bigg(\int_0^\tau e^{\alpha s}\big|\delta f_s(Y_s,Z_s)\big|ds \bigg)^p \bigg]
	                     + d_p\dbE^{\dbQ^\lambda} \bigg[ \bigg(\int_0^\tau e^{2\alpha s}\big|\widehat\sigma^{\top}_s\delta Z_s\big|^2ds\bigg)^{\frac{p}{2}}\bigg]\bigg) \\
  	     & \quad\leq C_{p,\alpha,\beta,L} \bigg(\big\|\delta Y\big\|_{\dbD^p_{\beta,\tau}(\dbQ^\lambda)}^p + \big\|\delta Z\big\|_{\dbH^p_{\beta,\tau}(\dbQ^\lambda)}^p 
	                             + \dbE^{\dbQ^\lambda}\bigg[\bigg(\int_0^\tau e^{\alpha s} \big|\delta f_s(Y_s,Z_s)\big|ds\bigg)^p\bigg]\bigg),
	\end{align*}
	for some constant $C_{p,\alpha,\beta,L}$.
\end{proof}

\subsection{Wellposedness of reflected backward SDEs}  \label{sect:rbsde-existence}

We start from the so-called Snell envelope defined by the dynamic optimal stopping problem\footnote{Indeed, from $\|S^+\|_{\cD^q_{\rho,\tau}(\dbP)}<\infty$, we have $\lim_{t\rightarrow\infty}\widehat{S}_{t\wedge\tau}\leq 0$, $\mathbb{P}$-a.s. on $\{\tau=\infty\}$.}: 
  \begin{align*}
  	 \widehat{y}_{t\wedge\tau}
  	  := \esssup_{\theta\in \cT_{t,\tau}}\dbE^{\dbP}\Big[\widehat{\xi}\Ind_{\{\theta\geq \tau\}}+\widehat{S}_{\theta}\Ind_{\{\theta<\tau\}}\Big|\cF_{t\wedge\tau}^{+,\dbP}\Big], 
  	   \quad \mbox{with}\,\,\,
  	  \widehat{\xi} := e^{-\mu\tau}\xi{\mathbb 1}_{\{\tau<\infty\}}, \,\,\, \widehat{S}_t := e^{-\mu t}S_t,
  \end{align*} 
  where $\cT_{t,\tau}$ denotes the set of all $\dbF^{+, \mathbb{P}}-$stopping times $\theta$ with $t\wedge\tau\le\theta\le\tau$.
Following the proof of \cite[Proposition 5.1]{KKPPQ97} and the theory of optimal stopping, see e.g., \cite{ElK81}, we deduce that there exists an $X-$integrable process $\widehat{z}$, such that:
  \begin{align*}
    \begin{cases}
      \displaystyle \widehat{y}_{t\wedge\tau} = \widehat{\xi} - \int_{t\wedge\tau}^\tau \widehat{z}_s\cdot dX_s - \int_{t\wedge\tau}^\tau d\widehat{u}_s, \\
        \vspace{2mm}
      \widehat{y}_t \geq \widehat{S}_t, \quad t\geq 0, \quad \dbP\mbox{-a.s.} \\
        \vspace{2mm}
      \displaystyle \dbE^\dbP\bigg[\int_0^{\tau\wedge t} \big(1\wedge(\widehat{y}_{t-}-\widehat{S}_{t-})\big)d\widehat{u}_t \bigg] = 0,~~\mbox{for all}~~t\ge 0,
    \end{cases}
  \end{align*}
  where $\widehat{u}$ is local supermartingale, starting from $\widehat{u}_0=0$, orthogonal to $X$, i.e., $[X, \widehat{u}]=0$. 
In other words, $(\widehat{y},\widehat{z},\widehat{u})$ is a solution of the RBSDE with generator $f\equiv 0$ and obstacle $\widehat{S}$. 
Then, it follows by It\^o's formula that the triple $(y,z,u)$, defined by 
  \begin{align*}
  	y_t:=e^{\mu t}\widehat{y}_t, \quad z_t:=e^{\mu t}\widehat{z}_t, \quad u_t:=\int_0^te^{\mu s}d\widehat{u}_s, \quad t\ge 0,
  \end{align*}
  is a solution of the following RBSDE, for $t$, $t'\in \mathbb{R}_+$, $t\leq t'$, 
  \begin{align*}
    \begin{cases}
       \displaystyle y_{t\wedge\tau} = y_{t' \wedge\tau} - \mu\int_{t\wedge\tau}^{t' \wedge \tau} y_sds - \int_{t\wedge\tau}^{t' \wedge \tau} z_s\cdot dX_s 
                       - \int_{t\wedge\tau}^{t' \wedge \tau} du_s,  ~  \\ 
       \vspace{2mm}
       y_t \geq S_t,\quad t\geq 0, \quad \dbP\mbox{-a.s.} \quad \mbox{and} \quad y_\tau=\xi \ \mbox{on}\ \{\tau<\infty\} \\
       \vspace{2mm}
       \displaystyle \dbE^\dbP\bigg[\int_0^{t\wedge \tau}\big(1\wedge(y_{t-}-S_{t-})\big)du_t\bigg] = 0, ~~\mbox{for all}~~t\ge 0,
    \end{cases}
  \end{align*}
  where $u$ is local supermartingale, starting from $u_0=0$, orthogonal to $X$, i.e., $[X, u]=0$.

\begin{lemma}\label{lem:yz}
	For all $\alpha\in[-\mu,\rho)$ and $p\in(1,q)$, we have
	 \begin{align*}
	 	\big\|y\big\|_{\cD_{\rho,\tau}^p}^p+\big\|z\big\|_{\cH_{\alpha,\tau}^p}^p
	 	\leq C_{p,q,L,\alpha,\rho}\Big(\big\|\xi{\mathbb 1}_{\{\tau<\infty\}}\big\|_{\cL_{\rho,\tau}^q}^q 
	 	+ \big\|S^+\big\|_{\cL_{\rho,\tau}^q}^q\Big)^{\frac{p}{q}}.
	 \end{align*}
\end{lemma}

\begin{proof}
	By the definition of $\widehat{y}$, we have
	  \begin{equation*}
	    \dbE^\dbP\Big[e^{-\mu\tau}\xi{\mathbb 1}_{\{\tau<\infty\}}\Big|\cF^{+,\dbP}_{t\wedge\tau}\Big]  
	      \leq \widehat{y}_{t\wedge \tau} 
	      \leq \esssup_{\theta\in\cT_{t,\tau}}\dbE^{\dbP}\Big[e^{-\mu\tau}|\xi|{\mathbb 1}_{\{\tau<\infty\}}+ e^{-\mu \theta}S^+_{\theta}\Big|\cF_{t\wedge\tau}^{+,\dbP}\Big].
	  \end{equation*}
	Then, for $\alpha\in [-\mu,\rho]$, 
	  \begin{align*}
	    -\dbE^\dbP\Big[e^{\alpha\tau}|\xi|{\mathbb 1}_{\{\tau<\infty\}}\big|\cF_{t\wedge\tau}^{+,\dbP}\Big] 
	      &\leq -e^{(\alpha+\mu)(t\wedge\tau)}\dbE^\dbP\Big[e^{-\mu\tau}|\xi|{\mathbb 1}_{\{\tau<\infty\}}\Big|\cF_{t\wedge\tau}^{+,\dbP}\Big] \\
	      &\leq e^{(\alpha+\mu)(t\wedge\tau)}\widehat{y}_{t\wedge\tau} = e^{\alpha (t\wedge\tau)}y_{t\wedge\tau}  \\
	      &\leq \esssup_{\theta\in\cT_{t,\tau}}\dbE^{\dbP}\Big[e^{\alpha\tau}|\xi|{\mathbb 1}_{\{\tau<\infty\}}+ e^{\alpha\theta}S^+_{\theta}\Big|\cF_{t\wedge\tau}^{+,\dbP}\Big] \\
	      &\leq \dbE^{\dbP}\Big[e^{\alpha\tau}|\xi|{\mathbb 1}_{\{\tau<\infty\}}+ \sup_{0\leq s<\infty}e^{\alpha (s\wedge \tau)}S^+_{s\wedge \tau}\Big|\cF_{t\wedge\tau}^{+,\dbP}\Big], 
	  \end{align*}
	  and therefore 
	  \begin{align*}
	  	e^{\alpha (t\wedge\tau)}|y_{t\wedge\tau}| \leq \dbE^{\dbP}\Big[e^{\alpha\tau}|\xi|{\mathbb 1}_{\{\tau<\infty\}}+ \sup_{0\leq s<\infty}e^{\alpha (s\wedge\tau)}S^+_{s\wedge \tau}\Big|\cF_{t\wedge\tau}^{+,\dbP}\Big]. 
	  \end{align*}
	By Lemma \ref{supDoob}, this implies that
	  \begin{align*}  
	     &\cE \bigg[\sup_{0\leq s<\infty}\big|e^{\alpha (s\wedge\tau)}y_{s\wedge\tau}\big|^p\bigg]  \\
    	 &\quad \leq C_p\cE\bigg[\sup_{0\leq s\leq\tau}\dbE^\dbP\bigg[|e^{\alpha\tau}\xi{\mathbb 1}_{\{\tau<\infty\}}|^p
	           +\sup_{0\leq u<\infty} \big(e^{\alpha (u\wedge\tau)}S_{u\wedge\tau}^+\big)^p\Big|\cF_s^{+,\dbP}\bigg]\bigg] \\
	     &\quad \leq C_{p,p'}\Big( \big\|\xi{\mathbb 1}_{\{\tau<\infty\}}\big\|_{\cL_{\alpha,\tau}^{p'}}^{p'} + \big\|S^+\big\|_{\cD_{\alpha,\tau}^{p'}}^{p'}\Big)^{\frac{p}{p'}}, 
	  \end{align*}
	  for all $1<p<p'$. 
	By our assumption on $\xi$ and $S^+$, we see that we need to restrict to $p'<q$ in order to ensure that the last bound is finite. 
	Moreover, by Proposition \ref{EstimationRBSDE}, we have for some $\alpha' >\alpha$, 
	  \begin{align*}
	    \cE\bigg[\bigg(\int_0^\tau e^{2\alpha t}\big|\widehat\sigma^{\top}_tz_t\big|^2dt\bigg)^{\frac{p}{2}}\bigg]
	      & \leq C_{p,\alpha,\alpha', L} \cE\bigg[\sup_{0\leq t<\infty}\big|e^{\alpha' (t\wedge\tau)}y_{t\wedge\tau}\big|^p\bigg] \\
	      & \leq C_{p,p',\alpha,\alpha', L}\cE\bigg[\big|e^{\alpha'\tau}\xi{\mathbb 1}_{\{\tau<\infty\}}\big|^{p'}
	             +\sup_{0\leq t<\infty} \big(e^{\alpha'(t\wedge\tau)}S_{t\wedge\tau}^+\big)^{p'}\bigg]^{\frac{p}{p'}}.
	  \end{align*}
	By our assumption on $\xi$ and $S^+$, we see that we need to restrict $\alpha$ to the interval $[-\mu,\rho)$ in order to ensure that the last bound is finite.
\end{proof} 

\vspace{3mm}

Now, we construct a sequence of approximating solutions to the RBSDE, using the finite horizon RBSDE result in \cite{BPTZ18} and on the optimal stopping problem above. 

Let $\tau_n:=\tau\wedge n$, and $(Y^n,Z^n,U^n)$ be the solution to the following RBSDE 
 \begin{align*}
   \begin{cases}
     \displaystyle Y^n_{t\wedge\tau} = y_{\tau_n} +\int_{t\wedge\tau_n}^{\tau_n}f_s(Y_s^n,Z_s^n)ds - \int_{t\wedge\tau_n}^{\tau_n} \big(Z^n_s\cdot dX_s + dU^n_s\big), \\
     \vspace{2mm}
     Y^n_{t\wedge\tau_n} \ge S_{t\wedge\tau_n}, ~~ t\geq 0, \quad \dbP\mbox{-a.s.}, \\
     \vspace{2mm}
     \displaystyle  \dbE^{\dbP}\bigg[\int_0^{t\wedge\tau_n}\big(1\wedge(Y^n_{t-}-S_{t-})\big)dU^n_t\bigg] = 0,~~\mbox{for all}~~t\ge 0.
   \end{cases}
 \end{align*}
We extend the definition of $Y^n$ for $t\wedge\tau\geq\tau_n$ by
  \begin{align*}
  	Y^n_{t\wedge\tau} = y_{t\wedge\tau}, \quad Z^n_{t\wedge\tau} = z_{t\wedge\tau}, \quad U_{t\wedge\tau}^n = u_{t\wedge\tau}, 
  \end{align*}
  so that $(Y^n,Z^n,U^n)$ is a solution of the RBSDE with parameters $(f^n,\xi,S)$: for $t$, $t'\in \mathbb{R}_+$,
   \begin{align}  \label{RBSDEn}
   	 \begin{cases}
   	   \displaystyle Y^n_{t\wedge\tau} = Y^n_{t'\wedge\tau}  +\int_{t\wedge\tau}^{t' \wedge \tau}\!\!\!f^n_s(Y_s^n,Z_s^n)ds 
   	                   - \int_{t\wedge\tau}^{t' \wedge\tau} \big(Z^n_s\cdot dX_s + dU^n_s\big), \\
   	     \vspace{2mm}
   	     Y^n \ge S, \quad \dbP\mbox{-a.s.} \quad \mbox{and} \quad Y^n_\tau=\xi~~\mbox{on}\ \{\tau<\infty\}  \\
   	     \vspace{2mm}
   	     \displaystyle \dbE^{\dbP}\bigg[\int_0^{t\wedge\tau}\big(1\wedge(Y^n_{t-}-S_{t-})\big)dU^n_t\bigg] = 0,~~\mbox{for all}~~t\ge 0,
   	 \end{cases}
   \end{align}
   where 
   \begin{align*}
   	  f^n_t(y,z):=f_s(y,z)\Ind_{\{s\leq n\}}-\mu y\Ind_{\{s>n\}}, \quad t\geq 0, \quad (y,z)\in\dbR\times\dbR^d.
   \end{align*}
The following result justifies the existence statement in Theorem \ref{thm:rbsde}.

\vspace{3mm}

\begin{proposition} \label{prop:rbsde-existence}
	For all $\eta\in[-\mu,\rho)$ and $p\in(1,q)$, the sequence $\{(Y^n,Z^n,U^n)\}_{n\in\dbN}$ converges in $\cD_{\eta,\tau}^p\times \cH_{\eta,\tau}^p\times\cU_{\eta,\tau}^p$ to some $(Y,Z,U)$, which is a solution of the random horizon RBSDE with the parameters $(f,\xi,S)$.
\end{proposition}

\begin{proof}
	\noindent \textbf{1.} We first show that $\{(Y^n,Z^n,U^n)\}_{n\in\dbN}$ is a Cauchy sequence in $\cD_{\eta,\tau}^p\times \cH_{\eta,\tau}^p\times\cU_{\eta,\tau}^p$, which induces the convergence of $(Y^n,Z^n,U^n)$ towards some $(Y,Z,U)$ in $\cD_{\eta,\tau}^p\times \cH_{\eta,\tau}^p\times\cU_{\eta,\tau}^p$. 
	
	By the stability result of Theorem \ref{thm:rbsdecomp-stab} (ii), we have the following estimate for the differences $(\delta Y,\delta Z,\delta U):= (Y^n-Y^m,Z^n-Z^m,U^n-U^m)$, $n>m$, 
	  \begin{equation} \label{yzuCauchy}
	  	\begin{aligned}
	  	  & \|\delta Y\|^p_{\cD^p_{\eta,\tau}} + \|\delta Z\|^p_{\cH^p_{\eta,\tau}} + \|\delta U\|^p_{\cN^p_{\eta,\tau}}  \\ 
	  	  & \hspace{10mm} \leq C \Big\{\Delta_f+\Delta_f^{\frac12}\Big(2\big(\overline{f}^\dbP_{\eta',p,\tau}\big)^{\frac{p}{2}}
	  	                           + \|Y^m\|_{\cD^p_{\eta',\tau}}^{\frac{p}{2}} + \|Y^n\|_{\cD^p_{\eta',\tau}}^{\frac{p}{2}}\Big)\Big\},
	  	\end{aligned}
	  \end{equation}
	where, by the Lipschitz property of $f$ in Assumption \ref{assum:bsdeF},
	  \begin{align*}
	  	 \Delta_f^{\frac{p'}{p}}
	  	   &= \cE\bigg[\bigg(\int_{\tau_m}^{\tau_n} e^{\eta's}\big|\delta f_s(Y^m_s,Z^m_s)\big|ds\bigg)^{p'}\bigg] \\
	  	   &\leq C_{p',\eta',L}\bigg(\cE\bigg[\bigg(\int_{\tau_m}^{\tau_n} e^{\eta's}\big|f_s^0\big|ds\bigg)^{p'}\bigg] 
	  	           + \cE\bigg[\bigg(\int_{\tau_m}^{\tau_n} e^{\eta' s}|y_s|ds\bigg)^{p'}\bigg] \\
	  	   & \hspace{55mm}  + \cE\bigg[\bigg(\int_{\tau_m}^{\tau_n} e^{\eta' s}\big|\widehat\sigma^{\top}_sz_s\big|ds\bigg)^{p'}\bigg] \bigg).
	  \end{align*}
	By the Cauchy-Schwartz inequality, we have 
	  \begin{align*}
	  	 \cE\bigg[\bigg(\int_{\tau_m}^{\tau_n} \!\!e^{\eta' s}\big|f_s^0\big|ds\bigg)^{p'}\bigg] 
	  	   &\leq \bigg(\frac{e^{-2(\rho-\eta')m}}{2(\rho-\eta')}\bigg)^{\frac{p'}{2}}\cE\bigg[\bigg(\int_{0}^{\tau}e^{2\rho s}\big|f_s^0\big|^2ds\bigg)^{\frac{p'}{2}}\bigg] \\
	  	   &\leq \bigg(\frac{e^{-2(\rho-\eta')m}}{2(\rho-\eta')}\bigg)^{\frac{p'}{2}}\Big(\overline{f}^\dbP_{\rho,q,\tau}\Big)^{p'}.
	  \end{align*}
	Similarly, for $\eta<\eta'<\eta''<\rho$, we obtain that
	  \begin{align*}
	     \cE\bigg[\bigg(\int_{\tau_m}^{\tau_n} e^{\eta' s}\big|\widehat\sigma^{\top}_sz_s\big|ds \bigg)^{p'}\bigg]
	       &\leq \left(\frac{e^{-2(\eta''-\eta')m}}{2(\eta''-\eta')}\right)^{\frac{p'}{2}}\|\delta Z\|^{p'}_{\cH^{p'}_{\eta'',\tau}} \\
	       &\leq C\bigg(\frac{e^{-2(\eta''-\eta')m}}{2(\eta''-\eta')}\bigg)^{\frac{p'}{2}}
	             \Big(\big\|\xi{\mathbb 1}_{\{\tau<\infty\}}\big\|_{\cL^q_{\rho,\tau}}^{p'}+ \big\|S^+\big\|_{\cD^q_{\rho,\tau}}^{p'} \Big), 
	  \end{align*}
	and 
	  \begin{align*}
	    \cE\bigg[\bigg(\int_{\tau_m}^{\tau_n} e^{\eta' s}|y_s|ds\bigg)^{p'}\bigg]
	      &\leq \bigg(\frac{e^{-(\rho-\eta')m}}{\rho-\eta'}\bigg)^{p'}\|y\|^{p'}_{\cD^{p'}_{\rho,\tau}}  \\
	      &\leq C \bigg(\frac{e^{-(\rho-\eta')m}}{\rho-\eta'}\bigg)^{p'}
	         \Big(\big\|\xi{\mathbb 1}_{\{\tau<\infty\}}\big\|_{\cL^q_{\rho,\tau}}^{p'} + \big\|S^+\big\|_{\cD^q_{\rho,\tau}}^{p'}\Big). 
	  \end{align*}
	The last three estimates show that $\Delta_f\longrightarrow 0$ as $m,n\to\infty$, so that the required Cauchy property would follow from \eqref{yzuCauchy} once we establish that $\big\|Y^n\big\|_{\cD_{\eta',\tau}^p}$ is bounded uniformly in $n$. To see this, notice that $\big\|Y^n\big\|_{\cD_{\eta',\tau}^p}\le \big\|y\big\|_{\cD_{\eta',\tau}^p}+\big\|Y^n-y\big\|_{\cD_{\eta',\tau}^p}$, where $\big\|y\big\|_{\cD_{\eta',\tau}^p}$ is finite by Lemma \ref{lem:yz}, and thus it reduces our task to controlling $\big\|Y^n-y\big\|_{\cD_{\eta',\tau}^p}$. 
	To do this, we use \eqref{estimationDYeq1} to obtain
	  \begin{align*}
	  	 &\cE\bigg[\sup_{0\leq s\leq\tau}e^{p\eta' s}|Y_s^n-y_s|^p\bigg]
	  	   \leq C_{p,p'} \cE\bigg[\bigg(\int_0^n e^{\eta' s}|f_s(y_s,z_s)-\mu y_s|ds\bigg)^{p'}\bigg]^{\frac{p}{p'}} \\
	  	 &\qquad \le C_{p,p',\mu, L} \cE\bigg[\bigg(\int_0^\tau e^{\eta' s}\big|f_s^0\big|ds\bigg)^{p'} + \bigg(\int_0^\tau e^{\eta' s}|y_s|ds\bigg)^{p'}
	  	                    + \bigg(\int_0^\tau e^{\eta' s}\big|\widehat\sigma^{\top}_sz_s\big|ds\bigg)^{p'}\bigg]^{\frac{p}{p'}},     
	  \end{align*}
	and we argue as above to verify that the last bound is finite, using the integrability condition on $f^0$ in Assumption \ref{assum:bsdeF}, together with Lemma \ref{lem:yz}.
	
	\vspace{3mm}
	
	\noindent {\bf 2.} We next prove that the limit process $U$ is a c\`adl\`ag supermartingale with $[U,X]=0$. Theorem \ref{thm:rbsdecomp-stab} (ii) also implies that
	\begin{align*}
	   \cE\bigg[\sup_{0\leq t<\infty}\big|\overline{U}^n_{t\wedge\tau}-\overline{U}^m_{t\wedge\tau}\big|^p\bigg] \longrightarrow 0,
	   \quad\mbox{where}~~\overline{U}^n:=\int_0^.e^{\eta s}dU^n_s.
	\end{align*}    
	Then, there exists a limit process $\overline U\in\cD^p_{0,\tau}(\dbP)$. As $\overline U^n$ is a c\`adl\`ag $\dbQ-$uniformly integrable supermartingale for all $\dbQ\in\cQ_L$, we may deduce that its limit $\overline U$ is also a c\`adl\`ag $\dbQ-$uniformly integrable supermartingale for all $\dbQ\in\cQ_L$. Define $U_t:=\int_0^te^{-\eta s}d\overline U_s$, $t\geq 0$. Clearly, $U\in\cD^p_{\eta,\tau}(\dbP)$. As the integrand $e^{-\eta s}$ is positive, the process $U$ is a supermartingale. By Kunita-Watanabe inequality for semimartingales, we obtain 
	   \begin{align*}
	     \int_0^\tau e^{2\eta s}\big|d[U,X]_s\big| 
	       &\leq \int_0^\tau e^{2\eta s}\big|d[U-U^n,X]_s\big| + \int_0^\tau e^{2\eta s}\big|d[U^n,X]_s\big| \\
 	       &\leq \bigg(\int_0^\tau e^{2\eta s}d[U-U^n]_s\bigg)^{\frac{1}{2}}\bigg(\int_0^\tau e^{2\eta s}d[X]_s\bigg)^{\frac{1}{2}}.  
	   \end{align*}
	Theorem \ref{thm:rbsdecomp-stab} (ii) also states that the right-hand side converges a.s.~to $0$, at least along a subsequence, which implies that $[U,X]=0$.  
	
	\vspace{3mm}
	
	\noindent {\bf 3.} Clearly, $Y\ge S$, $\dbP-$a.s. In this step, we prove that the limit supermartingale $U$ satisfies the Skorokhod condition. To do this, denote $\varphi^n:=1\wedge(Y^n-S)$, $\varphi:=1\wedge(Y-S)$, and let us show that the convergence of $(Y^n,U^n)$ to $(Y,U)$ implies that
	  \begin{align*}
	  	 a_n := \dbE\bigg[\int_0^{\tau\wedge t}\varphi^n_{r-}dU^n_r\bigg]-\dbE\bigg[\int_0^{\tau\wedge t}\varphi_{r-}dU_r\bigg] \longrightarrow 0
	  	  ~~\mbox{as $n\to\infty$, ~~for all}~~t\geq 0.
	  \end{align*}
	
	For $\varepsilon>0$, let $\tau^\varepsilon_{0}=0$, 
	  \begin{align*}
	  	\tau^\varepsilon_{i+1} := \inf\big\{r>\tau^\varepsilon_{i}\,:\,|\varphi_r-\varphi_{\tau_{i}^\varepsilon}|\geq \varepsilon\big\},
	  \end{align*}
	  and 
	  \begin{align*}
	  	 \varphi^{\varepsilon} := \sum_{i\ge 0} \varphi_{\tau^{\varepsilon}_i}\Ind_{[\tau^\varepsilon_{i},\tau^\varepsilon_{i+1})}, 
	  	  \quad\mbox{so that }\, |\varphi-\varphi^\varepsilon|\le\varepsilon.
	  \end{align*}
	We first decompose 
	  \begin{align*}
	     a_n \leq \bigg|\dbE\bigg[\int_0^{\tau\wedge t}(\varphi^n-\varphi)_{r-}dU^n_r\bigg]\bigg| 
	              &+ \bigg|\dbE\bigg[\int_0^{\tau\wedge t}(\varphi-\varphi^\varepsilon)_{r-}d(U^n-U)_r\bigg]\bigg| \\
	              &+ \bigg|\dbE\bigg[\int_0^{\tau\wedge t}\varphi^\varepsilon_{r-}d(U^n-U)_r\bigg]\bigg|.
	  \end{align*}
	Since $\varphi^\varepsilon$ is piecewise constant, bounded by $1$, and $U^n\to U$ in $\cD^p_{\eta,\tau}$, we get 
	  \begin{equation*}
	   \lim_{n\to\infty}\dbE\bigg[\int_0^{\tau\wedge t}\varphi_{s-}^\varepsilon d(U^n_s-U_s)\bigg] = 0.
	  \end{equation*}
	For the second term, we have
	  \begin{align*}
	    0 & \leq \bigg|\dbE\bigg[\int_0^{\tau\wedge t}(\varphi_{s-}-\varphi^\varepsilon_{s-})d(U^n_s-U_s)\bigg]\bigg| \\
	      & \leq \bigg|\dbE\bigg[\int_0^{\tau\wedge t}(\varphi_{s-}-\varphi^\varepsilon_{s-})d(N^n_s-N_s)\bigg]\bigg| 
	            + \bigg|\dbE\bigg[\int_0^{\tau\wedge t}(\varphi_{s-}-\varphi^\varepsilon_{s-})d(K^n_s-K_s)\bigg]\bigg| \\
	      & = \varepsilon \dbE\big[K_{\tau\wedge t} + K_{\tau\wedge t}^n\big].
	  \end{align*}
	By \eqref{estimationKeq13} and $|f^{n,0}|\leq |f^{0}|$ we obtain that 
	  \begin{align*}
	   \dbE[K_{\tau\wedge t}^n] 
	     &\leq C\bigg(\|Y^n\|_{\cD^p_{\alpha,\tau}} + \dbE\bigg[\bigg(\int_0^{\tau\wedge t}e^{2\rho s}|f^{n,0}_s|^2ds\bigg)^{\frac{q}{2}}\bigg]\bigg) \\
	     &\leq C\bigg(\|Y^1\|_{\cD^p_{\alpha,\tau}} + \|Y\|_{\cD^p_{\alpha,\tau}} 
	         + \dbE\bigg[\bigg(\int_0^{\tau\wedge t}e^{2\rho s}|f^{0}_s|^2ds\bigg)^{\frac{q}{2}}\bigg]\bigg) < \infty.     
	  \end{align*}
	Hence, we may control the second term by choosing $\varepsilon$ arbitrarily small. 
	For the first term, we have 
	  \begin{align*}
	     0 &\leq \bigg|\dbE\bigg[\int_0^{\tau\wedge t}(\varphi^n_{s-}-\varphi_{s-})dU^n_s\bigg]\bigg| 
	         = \bigg|\dbE\bigg[\int_0^{\tau\wedge t}(\varphi^n_{s-}-\varphi_{s-})dK^n_s\bigg]\bigg| \\
	       &\leq \dbE\bigg[\bigg(\sup_{0\leq s\leq \tau\wedge t}|Y^n_{s}-Y_{s}|\wedge 1\bigg)K^n_{\tau\wedge t}\bigg] 
	        \leq \dbE\bigg[\sup_{0\leq s\leq \tau\wedge t}|Y^n_{s}-Y_{s}|^{p^*}\wedge 1\bigg]^{\frac{1}{p^*}}\dbE\big[(K^n_{\tau\wedge t})^p\big]^{\frac{1}{p}}.
	  \end{align*}
	Again we may show that $\dbE\big[(K^n_{\tau\wedge t})^p\big]$ is bounded by a constant, independent of $n\in\dbN$. As $Y^n\to Y$ in $\cD^p_{\eta,\tau}$, we have 
	  \begin{align*}
	  	\sup_{0\leq s\leq \tau\wedge t}|Y^n_{s}-Y_{s}|^{p^*} \to 0, \quad a.s.
	  \end{align*}
	By dominated convergence, we have 
	  \begin{align*}
	  	 \lim_{n\to\infty} \dbE\bigg[\sup_{0\leq s\leq \tau\wedge t}|Y^n_{s}-Y_{s}|^{p^*}\wedge 1\bigg] = 0. 
	  \end{align*}
	Hence, we have 
	  \begin{align*}
	  	 \lim_{n\to\infty}\dbE\bigg[\int_0^{\tau\wedge t}(\varphi^n_{s-}-\varphi_{s-})dU^n_s\bigg] = 0. 
	  \end{align*}
	All together, we have 
	  \begin{align*}
	  	 \lim_{n\to\infty}\dbE\bigg[\int_0^{\tau\wedge t}\varphi^n_{s-}dU^n_s-\int_0^{\tau\wedge t}\varphi_{s-}dU_s\bigg] = 0,
	  \end{align*}
	  and the assertion follows. 
	
	\vspace{3mm}
	
	\noindent {\bf 4.} We finally verify that $(Y,Z,U)$ satisfies the differential part of the RBSDE. 
	The following verification is reported for the convenience of the reader, and reproduces exactly the line of argument in \cite[Section 5.2, Step 3]{DP97}. 
	For any $\alpha\in\dbR$ and $t\geq 0$, we have by It\^o's formula and \eqref{RBSDEn} that, for $t$, $t'\in \mathbb{R}_+$, $t\leq t'$,
	  \begin{align*}
	     e^{\alpha(t\wedge\tau)}Y^n_{t\wedge\tau} 
	       &= e^{\alpha(t'\wedge\tau)}Y_{t'\wedge\tau } +\int_{t\wedge\tau}^{t' \wedge\tau}e^{\alpha s}\Big\{\big(f^n_s(Y_s^n,Z_s^n)-\alpha Y_s^n\big)ds 
	                     - \big(Z^n_s\cdot dX_s +dU^n_s\big)\Big\} \\
	       &= e^{\alpha(t'\wedge\tau)}Y_{t'\wedge\tau } +\int_{t\wedge\tau}^{t' \wedge \tau}e^{\alpha s}\Big\{\big(f_s(Y_s^n,Z_s^n)-\alpha Y_s^n\big)ds 
	                                                    - \big(Z^n_s\cdot dX_s +dU^n_s\big)\Big\} \\
	       & \hspace{23.5mm} - \int_{t\wedge\tau_n}^{t' \wedge \tau}e^{\alpha s}\big(f_s(Y_s^n,Z_s^n)+\mu Y_s^n\big)ds.
	  \end{align*}
	We choose $\alpha<\eta$. Then, it is easily seen that $e^{\alpha(t\wedge\tau)}Y^n_{t\wedge\tau} \longrightarrow e^{\alpha(t\wedge\tau)}Y_{t\wedge\tau}$, for all $t\ge 0$, and so that $e^{\alpha\tau}Y^n_{\tau} \longrightarrow e^{\alpha\tau}\xi$, on $\{\tau<\infty\}$. Moreover, 
	  \begin{align*}
	  	 \int_{t\wedge\tau}^{t' \wedge\tau}e^{\alpha s}dU^n_s \longrightarrow \int_{t\wedge\tau}^{t' \wedge \tau}e^{\alpha s}dU_s, \quad\mbox{in }~ \dbL^p.
	  \end{align*}
	By the BDG inequality, it also follows that 
	  \begin{align*}
	  	 \int_{t\wedge\tau}^{t' \wedge\tau} e^{\alpha s}Z^n_s\cdot dX_s \longrightarrow \int_{t\wedge\tau}^{t' \wedge\tau} e^{\alpha s}Z_s\cdot dX_s, 
	  	    \quad \mbox{in }~ \dbL^p, \quad \mbox{for all }~ t\ge 0.
	  \end{align*}
	Moreover, we have 
	  \begin{align*}
	  	 \int_{t\wedge\tau}^{t' \wedge \tau}e^{\alpha s} Y_s^n ds \longrightarrow \int_{t\wedge\tau}^{t' \wedge\tau}e^{\alpha s} Y_s ds,  \quad \mbox{in }~\dbL^p,
	  	   \quad \mbox{for all } ~t\ge 0, 
	  \end{align*}
	  due to the following estimate 
	  \begin{align*}
	       &\dbE\bigg[\bigg(\int_0^\tau e^{\alpha s}|Y_s^n-Y_s|ds\bigg)^p\bigg] \\
	       &\qquad \leq \dbE\bigg[\sup_{0\leq s<\infty} e^{p\eta (s\wedge\tau)}|Y_{s\wedge\tau}^n-Y_{s\wedge\tau}|^p\bigg(\int_0^\tau e^{-(\eta-\alpha) s}ds\bigg)^p\bigg] \\
	       &\qquad \leq \frac{1}{(\eta-\alpha)^p}\dbE\bigg[\sup_{0\leq s<\infty} e^{p\eta (s\wedge\tau)}|Y_{s\wedge\tau}^n-Y_{s\wedge\tau}|^p\bigg] 
	             \longrightarrow 0, \quad \mbox{ as } n\to\infty.
	  \end{align*}
	From a similar argument, we also have 
	  \begin{align*}
	  	\int_{t\wedge\tau_n}^{t' \wedge \tau}e^{\alpha s}\big(f_s(Y_s^n,Z_s^n)+\mu Y_s^n\big)ds \longrightarrow 0, \quad \mbox{in }~\dbL^p,
	  \end{align*}
	  and by Lipschitz continuity of $f$ we see that 
	  \begin{align*}
	  	 \int_{t\wedge\tau}^{t' \wedge\tau}e^{\alpha s}\big|f_s(Y_s^n,Z_s^n)-f_s(Y_s,Z_s)\big|ds \longrightarrow 0, \quad \mbox{in }~\dbL^p, \quad \mbox{for all }~t\ge 0.
	  \end{align*}
	Therefore, we have proved that 
	  \begin{align*}
	     e^{\alpha(t\wedge\tau)}Y_{t\wedge\tau} 
	      = e^{\alpha(t' \wedge \tau)} Y_{t' \wedge\tau} +\int_{t\wedge\tau}^{t' \wedge \tau}e^{\alpha s}\Big\{\big(f_s(Y_s,Z_s)-\alpha Y_s\big)ds 
	         - \big(Z_s\cdot dX_s+dU_s\big)\Big\}, 
	  \end{align*}
	  thus completing the proof by It\^o's formula.
\end{proof}

We now prove the comparison result. In particular, this justifies the uniqueness statement in Theorem \ref{thm:rbsde}.

\begin{proof}[Proof of Theorem \ref{thm:rbsdecomp-stab} {\rm (i)}]
	Denote by $\big\{\big(Y^n,Z^n,U^n\big)\big\}_{n\in\dbN}$ and $\big\{\big({Y'}^n,{Z'}^n,{U'}^n\big)\big\}_{n\in\dbN}$ the approximating sequence of $(Y,Z,U)$ and $(Y',Z',U')$, using the triples $(y,z,u)$ and $(y',z',u')$, respectively, as in the last proof. Since $\xi\leq\xi'$ and $S\leq S'$, we have $y_{\tau_n}\le y'_{\tau_n}$. By standard comparison argument of BSDEs, see e.g.~\cite[Proposition 3.2]{RS12}, this in turns implies that $Y^n_{\tau_0}\le Y'^n_{\tau_0}$ for all stopping time $\tau_0\le\tau$. The required result follows by sending $n\to\infty$.  
\end{proof}

\subsection{Special case: backward SDE}  \label{sect:specialbsde}

\begin{proof}[Proof of Theorem \ref{thm:bsde}]
	By setting $S=-\infty$, the existence and uniqueness results follow from Theorem \ref{thm:rbsde}. In particular, the Skorokhod condition implies in the present setting that $U=N$ is a $\dbP-$martingale orthogonal to $X$. It remains to verify the estimates \eqref{estznk}.
	
	Let $\eta\ge-\mu$, and observe that Assumption \ref{assum:bsdeF} implies that 
	  \begin{align*}
	  	\sgn(y)f_s(y,z)\le -\mu|y| + L\big|\widehat\sigma_s^{\top}z\big| + \big|f^0_s\big|\le \eta |y| + L\big|\widehat\sigma_s^{\top}z\big| + \big|f^0_s\big|.
	  \end{align*}
	Then, by It\^o's formula, together with Proposition \ref{TanakaIneq}, we have 
	  \begin{align*}
	  	& e^{\eta(n\wedge\tau)}|Y_{n\wedge\tau}| - e^{\eta(t\wedge\tau)}|Y_{t\wedge\tau}| \\
	  	&\qquad \geq \int_{t\wedge\tau}^{n\wedge\tau} e^{\eta s}\Big\{\eta|Y_s|ds - \sgn(Y_{s-})\big(f_s(Y_s, Z_s)ds - Z_s\cdot dX_s - dN_s \big)\Big\} \\
	  	&\qquad \geq \int_{t\wedge\tau}^{n\wedge\tau} e^{\eta s}\Big\{- L\big|\widehat\sigma_s^{\top} Z_s\big|ds -|f^0_s| ds + \sgn(Y_{s})Z_s\cdot dX_s + \sgn(Y_{s-})dN_s\Big\}.
	  \end{align*}
	Introduce 
	  \begin{equation*}
	  	 \widehat{\lambda}_s:= L\sgn(Y_s)\frac{\widehat\sigma_s^{\top}Z_s}{|\widehat\sigma_s^{\top}Z_s|}\Ind_{\{|\widehat\sigma_s^{\top}Z_s|\neq 0\}},
	  \end{equation*}
      and recall that $N$ remains a martingale under $\dbQ^{\widehat\lambda}$ by the orthogonality $[X,N]=0$. 
    Then, taking conditional expectation under $\dbQ^{\widehat\lambda}$, we obtain
	  \begin{align*} 
	     e^{\eta(t\wedge\tau)}|Y_{t\wedge\tau}| 
	      &\leq \lim_{n\to\infty} \dbE^{\dbQ^{\widehat\lambda}} \bigg[e^{\eta(n\wedge\tau)}|Y_{n\wedge\tau}| + \int_{0}^{n\wedge\tau}\!e^{\eta s}\big|f^0_s\big|ds \;\bigg|\cF^{+,\dbP}_{t\wedge\tau} \bigg]  \\
	      &= \dbE^{\dbQ^{\widehat\lambda}}\bigg[e^{\eta\tau}|\xi{\mathbb 1}_{\{\tau<\infty\}}| + \int_{0}^{\tau}\!e^{\eta s}\big|f^0_s\big|ds\;\bigg|\cF^{+,\dbP}_{t\wedge\tau}\bigg],
	  \end{align*}
	  by the uniform integrability of the process $\{e^{\eta s}Y_s\}_{s\ge 0}$. 
	By Lemma \ref{supDoob}, this provides  
	  \begin{align}  \label{BSDEapriori1}
	    \big\| Y \big\|_{\cD^p_{\eta,\tau}}^p 
	      &\leq \frac{p'}{p'-p}\cE\left[\bigg(e^{\eta \tau}|\xi{\mathbb 1}_{\{\tau<\infty\}}| + \int_{0}^{\tau} e^{\eta s}\big|f^0_s\big|ds\bigg)^{p'}\right]^{\frac{p}{p'}} \nonumber \\
	      &\leq C_{p,p',\eta,\eta'}\Big\{\big\|\xi{\mathbb 1}_{\{\tau<\infty\}}\big\|_{\cL_{\eta',\tau}^{p'}}^{p}+ \big(\overline{f}^0_{\eta',p',\tau}\big)^p\Big\}, 
	  \end{align}
	  for all $p'\in(p,q)$ and $-\mu\leq\eta<\eta'\leq\rho$ with some constant $C_{p,p',\eta,\eta'}$. 
	Next we can follow the lines of the proof of Proposition \ref{EstimationRBSDE} to show that
	  \begin{align*}
	  	 \big\|Z\big\|^p_{\cH^p_{\eta'',\tau}(\dbP)} + \big\|N\big\|^p_{\cN^p_{\eta'',\tau}(\dbP)} 
	  	  \leq C_{p,L,\eta,\eta''}\Big(\|Y\|^p_{\cD^p_{\eta,\tau}(\dbP)} + \big(\overline{f}^0_{\eta,p,\tau}\big)^p\Big),
	  \end{align*}
	  for $\eta''<\eta$. 
	Combined with \eqref{BSDEapriori1}, this induces the required estimate.
\end{proof}

\begin{proof}[Proof of Theorem \ref{thm:bsdecomp-stab}]
	The comparison and stability result follow from Theorem \ref{thm:rbsdecomp-stab} and Theorem \ref{thm:bsde} respectively. 
\end{proof}

For later use, we need a version of the stability result allowing for different horizons. This requires to extend the generator and the solution of the BSDE beyond the terminal time by: 
  \begin{align*}
  	f_{t\vee\tau}(y,z) = 0, \quad Y_{t\vee\tau} = \xi, \quad Z_{t\vee\tau} = 0, \quad N_{t\vee\tau} = 0, \quad \mbox{for all}~~t\ge 0. 
  \end{align*}

\begin{proposition}  \label{prop:BSDEstab2times}
	For finite stopping times $\tau$ and $\tau'$, suppose $(f,\xi,\tau)$ and $(f',\xi',\tau')$ safisfy Assumptions \ref{assum:bsdeF} with the same parameters $L$ and $\mu$.
	Let 
	  \begin{align*}
	  	\delta Y = Y - Y', \quad \delta Z = Z - Z', \quad \delta N = N - N', \quad \delta f = f - f', \quad \delta \xi = \xi - \xi'.
	  \end{align*}
	Then, for all stopping time $\tau_0\leq \tau\wedge\tau'$, and all $\eta\in[-\mu,\rho)$, $1<p<p'<q$, we have
	  \begin{align*}
	  	\big|e^{\eta\tau_0}\delta Y_{\tau_0}\big|  
	  	 \leq \esssup_{\dbQ\in\cQ_L}\dbE^{\dbQ}\left[\big|e^{\eta\tau}\xi-e^{\eta\tau'}\xi'\big| 
	  	       + \int_{\tau_0}^{\tau\vee\tau'}e^{\eta s}\big|\delta f_s(Y_s,Z_s)\big|ds\bigg|\cF_{\tau_0}^{+,\dbP}\right].
	  \end{align*}
\end{proposition}

\begin{proof}
	By Proposition \ref{TanakaIneq} and the Lipschitz and monotonicity conditions of Assumption \ref{assum:bsdeF}, 
	  \begin{align*}
	     \big|e^{\eta{\tau_0}}\delta Y_{\tau_0}\big| 
	       &\leq \big|e^{\eta\tau}\xi-e^{\eta\tau'}\xi'\big| +\int_{\tau_0}^{\tau\vee\tau'}e^{\eta s}\big|\delta f_s(Y_s,Z_s)\big|ds
	               + e^{\eta s}\sgn(\delta Y_s)\delta Z_s\cdot\big(dX_s - \widehat\sigma_s\widehat\lambda_sds\big) \\
	       &\hspace{26.5mm} - \int_{\tau_0}^{\tau\vee\tau'}e^{\eta s}\sgn(\delta Y_{s-})d\delta N_s,
	  \end{align*}
	  with 
	  \begin{align*}
	  	\widehat\lambda_s:=L\sgn(\delta Y_s)\frac{\widehat\sigma_s^\top\delta Z_s}{|\widehat\sigma_s^\top\delta Z_s|}\Ind_{\{|\widehat\sigma_s^\top\delta Z_s|\neq 0\}}.
	  \end{align*}
	Taking conditional expectation under $\dbQ^{\widehat\lambda}\in\cQ_L$ induces the required inequality. 
\end{proof}

\section{Second order backward SDE: representation and uniqueness}  \label{sect:2bsde}

We shall use the additional notation:
  \begin{align*}
  	 \cE^{\dbP,+}_{t}[\cdot] 
  	  := \esssup^\dbP_{\dbP' \in \cP_\dbP^+(t)}\esssup^{\dbP'}_{\dbQ\in \cQ_L(\dbP')}\dbE^{\dbQ}[\cdot |\cF^+_t], \quad \mbox{for all} \quad 
  	    t\geq 0,~~\dbP\in\cP_0.
  \end{align*}

\begin{remark} \label{rem:esupassum}
	It follows from Assumption \ref{assum:2bsde-integ} and Doob's inequality that for any $q'<q$
	 \begin{align*}
	 	&\sup_{\dbP\in\cP_0}\sup_{\dbQ\in\cQ_L(\dbP)}\dbE^{\dbQ}\bigg[\sup_{0\leq t\leq \tau}\!\!\cE^{\mathbb{P},+}_{t}\big[(e^{\rho\tau}|\xi|)^{q'}\big]\bigg] \\
	 	& \quad + \sup_{\dbP\in\cP_0}\sup_{\dbQ\in\cQ_L(\dbP)}\dbE^{\dbQ}\bigg[\sup_{0\leq t\leq \tau}\cE^{\dbP, +}_{t}\bigg[\bigg(\int^\tau_0\big|e^{\rho s}f_s^0\big|^2ds\bigg)^\frac{q'}{2}\bigg]\bigg]<\infty.
	 \end{align*}
	We also note that $\cE_t^{\dbP,+}\big[\Ind_{\{\tau\geq n\}}\big]$ is a $\dbP$-supermartingale. 
	Then, by Doob's martingale inequality, we have 
	  \begin{equation*}
	     \dbE^\dbP\left[\cE_t^{\dbP,+}\big[\Ind_{\{\tau\geq n\}}\big]\right] \leq C\cE^{\cP_0}\big[\Ind_{\{\tau\geq n\}}\big] \longrightarrow 0,
	  \end{equation*}
	  so that 
	  \begin{equation*}
	  	 \dbE^\dbP\left[\lim_{n\to\infty}\cE_t^{\dbP,+}\big[\Ind_{\{\tau\geq n\}}\big]\right] = 0,
	  \end{equation*}
	  by dominated convergence theorem, and therefore 
	  \begin{align*}
	  	\lim_{n\to\infty}\cE_t^{\dbP,+}\big[\Ind_{\{\tau\geq n\}}\big] = 0, \quad \dbP\mbox{-a.s.}
	  \end{align*}
\end{remark}

\vspace{3mm}

Similarly to Soner, Touzi \& Zhang \cite{STZ12}, the uniqueness follows from the representation of the $Y$ component of the 2BDSE \eqref{2bsdel} by means of the family of backward SDEs. For all $\mathbb{P}\in \cP_0$, we denote by $\cY^{\mathbb{P}}[\xi_0,\tau_0]$ the $Y-$component of the solution of the backward SDE:
  \begin{align} \label{bsdenonshift}
  	\cY^{\dbP}_{t\wedge \tau} 
  	  = \xi_0 + \int^{\tau_0}_{t\wedge\tau_0} F_s\big(\cY^{\dbP}_s, \cZ^\dbP_s, \widehat{\sigma}_s\big)ds 
  	          - \int^{\tau_0}_{t\wedge\tau_0} \big(\cZ^{\dbP}_s\cdot dX_s+d\cN^{\dbP}_s\big), \quad t\ge 0,\quad \dbP\mbox{-a.s.}
  \end{align}
  where $\xi_0$ is an $\cF_{\tau_0}-$measurable random variable for some stopping time $\tau_0\le\tau$. 
Under our conditions on $(F,\xi)$, the wellposedness of these BSDEs for $\xi_0\in\cL^p_{\eta,\tau_0}(\dbP)$ follows from Theorem \ref{thm:bsde}. 
Remark that in the sequel we always consider the version of $\cY^{\dbP}$ such that $\cY^{\dbP}_{t\wedge \tau}\in \mathcal{F}^{+}_{t\wedge \tau}$ by the result of Lemma \ref{lem:cYfinite}.

The following statement provides a representation for the 2BSDE, and justifies the comparison (and uniqueness) result of Proposition \ref{prop:2BSDEcomp}.

\vspace{3mm}

\begin{proposition}  \label{prop:representation}
	Let Assumptions \ref{assum:bsdeF} and Assumption \ref{assum:2bsde-integ} hold true, and let $(Y, Z)\in \cD^p_{\eta,\tau}\big(\cP_0, \dbF^{+, \cP_0}\big)\times \cH^p_{\eta,\tau}\big(\cP_0, \dbF^{\cP_0}\big)$ be a solution of the 2BSDE \eqref{2bsdel}, for some $p\in(1,q)$ and $\eta\in[-\mu,\rho)$. 
	Then,
	  \begin{align} 
	     Y_{t_1\wedge \tau} 
	      &= \esssup^\dbP_{\dbP' \in \cP_\dbP^+(t_1\wedge \tau)}\cY^{\dbP'}_{t_1\wedge \tau} \big[Y_{t_2\wedge \tau}, t_2\wedge\tau\big]  \label{2bsderepresentationt1t2} \\ 
	      &= \esssup^\dbP_{\dbP' \in \cP_\dbP^+(t_1\wedge \tau)} \cY^{\dbP'}_{t_1\wedge \tau} [\xi,\tau],
	         \quad \dbP\mbox{-a.s. for all}~~\dbP\in \cP_0,~0\leq t_1\leq t_2. \label{2bsderepresentationt1}
	  \end{align}
	In particular, the 2BSDE has at most one solution in $\cD^p_{\eta,\tau}\big(\cP_0, \dbF^{+, \cP_0}\big)\times \cH^p_{\eta,\tau}\big(\cP_0, \dbF^{\cP_0}\big)$, satisfying the estimate \eqref{apriori-2bsde}, and the comparison result of Proposition \ref{prop:2BSDEcomp} holds true.
\end{proposition}

\begin{proof}
 The uniqueness of $Y$ is an immediate consequence of \eqref{2bsderepresentationt1}, and implies the uniqueness of $Z$, $\widehat{a}_tdt\otimes\cP_0$-q.s.~by the fact that 
   \begin{equation*}
   	  \langle Y,X\rangle_t = \left\langle \int_0^\cdot Z_s\cdot dX_s,X\right\rangle_t = \int_0^t\widehat a_s Z_sds, \quad \dbP\mbox{-a.s.}
   \end{equation*}
 This representation also implies the comparison result as an immediate consequence of the corresponding comparison result of the BSDEs $\cY^{\dbP}[\xi,\tau]$.
 
 \vspace{3mm}
 
\noindent {\bf 1.} We first prove \eqref{2bsderepresentationt1t2}. 
 Fix some arbitrary $\dbP\in\cP_0$ and $\dbP'\in\cP_\dbP^+(t_1\wedge\tau)$. 
 By Definition \ref{def:2bsde} of the solution of the 2BSDE \eqref{2bsdel}, we see that $Y$ is a supersolution of the BSDE on $\llbracket t_1\wedge\tau, t_2\wedge\tau \rrbracket$ under $\dbP'$ with terminal condition $Y_{t_2\wedge\tau}$ at time $t_2\wedge\tau$. 
 By the comparison result of Theorem \ref{thm:bsdecomp-stab} (ii), see also Remerk \ref{rem:supersolution}, this implies that $Y_{t_1\wedge\tau}\ge \cY^{\dbP'}_{t_1\wedge\tau}\big[Y_{t_2\wedge \tau}, t_2\wedge\tau\big]$, $\dbP'$-a.s. 
 As $\cY_{t_1}^{\dbP'}$ is $\cF_{t_1\wedge\tau}^+$-measurable and $Y_{t_1}$ is $\cF_{t_1\wedge\tau}^{+,\cP_0}$-measurable, the inequality also holds $\dbP$-a.s., by definition of $\cP_\dbP^+(t_1)$ and the fact that measures extend uniquely to the completed $\sigma$-algebras. 
 Then, by arbitrariness of $\dbP'$, 
     \begin{align*}
     	Y_{t_1\wedge\tau} 
     	\geq \esssup_{\dbP'\in\cP_\dbP^+(t_1\wedge\tau)}^\dbP \cY^{\dbP'}_{t_1\wedge\tau}\big[Y_{t_2\wedge \tau}, t_2\wedge\tau\big], 
     	 \quad \dbP\mbox{-a.s. for all}~~ \dbP\in\cP_0.
     \end{align*}
 We next prove the reverse inequality. 
 Denote $\delta Y:=Y-\cY^{\dbP'}\big[Y_{t_2\wedge \tau}, t_2\wedge\tau\big]$, $\delta Z:=Z-\cZ^{\dbP'}\big[Y_{t_2\wedge \tau}, t_2\wedge\tau\big]$ and $\delta U:=U^{\dbP'}-\cN^{\dbP'}\big[Y_{t_2\wedge \tau}, t_2\wedge\tau\big]$. Recall that $U^{\dbP'}$ is a $\dbP'$-supermartingale with decomposition $U^{\dbP'}=N^{\dbP'}-K^{\dbP'}$. For $\alpha\in[-\mu,\eta]$, it follows by It\^o's formula, together with the Lipschitz property of $F$ in Assumption \ref{assum:bsdeF} that there exist two bounded processes $a^{\dbP'}$ and $b^{\dbP'}$, uniformly bounded by the Lipschitz constant $L$ of $F$, such that 
    \begin{align*}
    	e^{\alpha(t_1\wedge \tau)}\delta Y_{t_1\wedge\tau} 
    	 = \int_{t_1\wedge \tau}^{t_2\wedge \tau} e^{\alpha s}\big(a_s^{\dbP'}\delta Y_s + b^{\dbP'}_s\cdot\widehat\sigma_s^{\top}\delta Z_s\big)ds 
    	   - \int_{t_1\wedge \tau}^{t_2\wedge \tau}  e^{\alpha s}\big(\widehat\sigma_s^{\top}\delta Z_s\cdot dW_s + d\delta U_s\big),
    \end{align*}
    which implies that
     \begin{align*}
     	e^{\alpha(t_1\wedge\tau)}\delta Y_{t_1\wedge\tau} 
     	  = - \dbE^{\dbP'}\bigg[\int_{t_1\wedge\tau}^{t_2\wedge\tau}\Gamma^{\dbP'}_se^{\alpha s}d\delta U_s^{\dbP'}\bigg|\cF_{t_1\wedge\tau}^+\bigg]
     	  = \dbE^{\dbP'}\bigg[\int_{t_1\wedge\tau}^{t_2\wedge\tau}\Gamma^{\dbP'}_se^{\alpha s}dK_s^{\dbP'}\bigg|\cF_{t_1\wedge\tau}^+\bigg], 
     \end{align*}
    with 
     \begin{align*}
     	\Gamma_s^{\dbP'} := \exp\bigg(\int_{t_1\wedge\tau}^s\Big(a_u^{\dbP'}-\frac{1}{2}\big|b_u^{\dbP'}\big|^2\Big)du + \int_{t_1\wedge\tau}^s b_u^{\dbP'}\cdot dW_u\bigg).
     \end{align*}
 As $a_u^{\dbP'},b_u^{\dbP'}$ are uniformly bounded by $L$, it follows from the Doob maximal inequality that 
     \begin{align*}
        &\dbE^{\dbP'}\bigg[\bigg(\sup_{t_1\wedge\tau\leq s\leq t_2\wedge\tau}\Gamma^{\dbP'}_s\bigg)^{\frac{p+1}{p-1}}\bigg|\cF_{t_1\wedge\tau}^+\bigg] \\
        &\quad \leq e^{L(t_2-t_1)}C'_p \dbE^{\dbP'}\bigg[\exp\bigg(-\int_{t_1\wedge\tau}^{t_2\wedge\tau}\frac{p+1}{2(p-1)}\big|b_u^{\dbP'}\big|^2du 
 	      + \frac{p+1}{p-1}\int_{t_1\wedge\tau}^{t_2\wedge\tau} b_u^{\dbP'}\cdot dW_u\bigg)\bigg|\cF_{t_1\wedge\tau}^+\bigg] \\
        &\quad < C_p < \infty,
     \end{align*}
    where $C_p$ is a constant independent of $\mathbb{P}'$. 
 Then, it follows from H\"older's inequality that 
     \begin{align*}
        & e^{-|\alpha|t_1}\delta Y_{t_1\wedge\tau} \\
        & \quad \leq 
             \dbE^{\dbP'}\bigg[\bigg(\sup_{t_1\wedge\tau\leq s\leq t_2\wedge\tau}\Gamma^{\dbP'}_s\bigg)^{\frac{p+1}{p-1}}\bigg|\cF_{t_1\wedge\tau}^+\bigg]^{\frac{p-1}{p+1}}
             \dbE^{\dbP'}\bigg[\bigg(\int_{t_1\wedge\tau}^{t_2\wedge\tau}e^{\alpha s}dK_s^{\dbP'}\bigg)^{\frac{p+1}{2}}\bigg|\cF_{t_1\wedge\tau}^+\bigg]^{\frac{2}{p+1}} \\
        & \quad \leq
             C^{\frac{p+1}{p-1}}_{p}\left(C^{\dbP,p,\alpha}_{t_1}\right)^{\frac{1}{p+1}}
             \dbE^{\dbP'}\bigg[\int_{t_1\wedge\tau}^{t_2\wedge\tau}e^{\alpha s}dK_s^{\dbP'}\bigg|\cF_{t_1\wedge\tau}^+\bigg]^{\frac{1}{p+1}} \\
        & \quad \leq
             C^{\frac{p+1}{p-1}}_{p}\Big(C^{\dbP,p,\alpha}_{t_1}\Big)^{\frac{1}{p+1}} e^{(\alpha t_1)\vee(\alpha t_2)}
             \dbE^{\dbP'}\bigg[\int_{t_1\wedge\tau}^{t_2\wedge\tau}dK_s^{\dbP'}\bigg|\cF_{t_1\wedge\tau}^+\bigg]^{\frac{1}{p+1}},
     \end{align*}
     where
     \begin{align*}
     	C^{\dbP,p,\alpha}_{t_1}
     	 := \esssup_{\dbP'\in\cP_\dbP^+(t_1\wedge\tau)}^\dbP
     	      \dbE^{\dbP'}\bigg[\bigg(\int_{t_1\wedge\tau}^{t_2\wedge\tau}e^{\alpha s} dK_s^{\dbP'}\bigg)^p\bigg|\cF_{t_1\wedge\tau}^+\bigg].
     \end{align*}
 As it follows from the minimality condition in Definition \ref{def:2bsde} that 
     \begin{equation*}
     	K^{\dbP}_{t_1\wedge\tau} = \essinf_{\dbP'\in\cP_\dbP^+(t_1\wedge\tau)}^\dbP\dbE^{\dbP'}\big[K^{\dbP'}_{t_2\wedge\tau}\big|\cF_{t_1\wedge\tau}^+\big],
     \end{equation*}
     and $C^{\dbP,p,\alpha}_{t_1}<\infty$ (see \eqref{CPpst}), we obtain that 
     \begin{equation*}
     	Y_{t_1\wedge\tau} - \esssup_{\dbP'\in\cP_\dbP^+(t_1\wedge\tau)}^\dbP\cY_{t_1\wedge\tau}^{\dbP'} \leq 0,\quad \dbP\mbox{-a.s.}
     \end{equation*} 
     thus providing the required equality.
 
 \vspace{3mm}
 
 \noindent {\bf 2.} Given \eqref{2bsderepresentationt1t2}, we now show \eqref{2bsderepresentationt1} by proving that
   \begin{equation*} \label{eq:representation2}
   	  \lim_{n\to\infty}\esssup^{\dbP}_{\dbP'\in\cP_\dbP^+(t\wedge\tau)}\big|\delta\cY^{\dbP', n}_{t\wedge\tau}\big|=0, 
   	   \quad \dbP\mbox{-a.s.~~ where} \quad \delta\cY^{\dbP', n}:=\cY^{\dbP'}[\xi,\tau]-\cY^{\dbP'}[Y_{n\wedge\tau},n\wedge\tau]. 
   \end{equation*}
 By the stability result of Proposition \ref{prop:BSDEstab2times}, we have 
   \begin{align*}
      & \big|e^{\eta({t\wedge\tau})}\delta\cY^{\dbP', n}_{t\wedge\tau}\big| \\
      & \quad\leq \esssup^{\dbP'}_{\dbQ\in\cQ_L(\dbP')} 
              \dbE^\dbQ\left[\big|e^{\eta\tau}\xi - e^{\eta(n\wedge\tau)}Y_{n\wedge\tau}\big| 
              + \int_{n\wedge\tau}^\tau e^{\eta s}\big|F_s\big(\cY_s^{\dbP'}[\xi,\tau],\cZ_s^{\dbP'}[\xi,\tau],\widehat\sigma_s\big)\big|ds\bigg| \cF_{t\wedge\tau}^+ \right].
   \end{align*}
 Notice that 
   \begin{align*}
   	 \big|e^{\eta\tau}\xi - e^{\eta(n\wedge\tau)}Y_{n\wedge\tau}\big|
   	  =\Ind_{\{\tau\geq n\}}\big|e^{\eta\tau}\xi - e^{\eta(n\wedge\tau)}Y_{n\wedge\tau}\big|
   	  \leq 2\Ind_{\{\tau\geq n\}}e^{\eta \tau}\sup_{0\le s\le\tau}|Y_s|.
   \end{align*}
 Then, it follows from H\"older's inequality that for some $p'>p$,
   \begin{align*}
   	 \cE_{t\wedge\tau}^{\dbP,+}\Big[\big|e^{\eta\tau}\xi-e^{\eta(n\wedge\tau)}Y_{n\wedge\tau}\big|^p\Big|\cF^+_{t\wedge\tau}\Big]
   	  \leq 2 \cE_{t\wedge\tau}^{\dbP,+}\bigg[\sup_{0\leq s\leq\tau}e^{p'\eta s}|Y_s|^{p'}\bigg]^{\frac{p}{p'}}
   	          \cE_{t\wedge\tau}^{\dbP,+}\big[\Ind_{\{\tau\geq n\}}\big]^{\frac{p'-p}{p'}} \longrightarrow 0,
   \end{align*}
   as $n\to \infty$, due to the fact that $Y\in\cD_{\tau,\eta}^p(\cP_0)$ and $\cE_{t\wedge\tau}^{\dbP,+}\left[\Ind_{\{\tau\geq n\}}\right] \longrightarrow 0$ by Remark \ref{rem:esupassum}. 
 This leads to
   \begin{equation} \label{end:representation}
     \begin{aligned} 
       &\limsup_{n\to\infty}\big|e^{\eta({t\wedge\tau})}\delta\cY^{\dbP', n}_{t\wedge\tau}\big| \\
       &\quad \leq \limsup_{n\to\infty}\esssup^{\dbP'}_{\dbQ\in\cQ_L(\dbP')}
             \dbE^\dbQ\bigg[\int_{n\wedge\tau}^\tau \!\! e^{\eta s}\big|F_s\big(\cY_s^{\dbP'}[\xi,\tau],\cZ_s^{\dbP'}[\xi,\tau],\widehat\sigma_s\big)\big|ds \bigg|\cF_{t\wedge\tau}^+\bigg].
     \end{aligned}
   \end{equation}
 We next write $\cY_s^{\dbP'}:=\cY_s^{\dbP'}[\xi,\tau]$, $\cZ_s^{\dbP'}:=\cZ_s^{\dbP'}[\xi,\tau]$, and estimate that
   \begin{align*}
     & \int_{n\wedge\tau}^\tau e^{\eta s}\big|F_s\big(\cY_s^{\dbP'},\cZ_s^{\dbP'},\widehat\sigma_s\big)\big|ds \\
     &\quad \leq \int_{n\wedge\tau}^\tau e^{\eta s}\big|f^0_s\big|ds
              + L\int_{n\wedge\tau}^\tau e^{\eta s}\big|\cY_s^{\dbP'}\big|ds
              + L\int_{n\wedge\tau}^\tau e^{\eta s}\big|\widehat\sigma_s^{\top}\cZ_s^{\dbP'}\big|ds \\
     &\quad \leq \bigg(\frac{e^{-2(\eta'-\eta)n}}{2(\eta'-\eta)}\bigg)^{\frac{1}{2}}
              \bigg[\bigg(\int_{0}^\tau e^{2\eta' s}\big|f^0_s\big|^2ds\bigg)^{\frac12}
               +L\bigg(\int_{0}^\tau e^{2\eta' s}\big|\widehat\sigma_s^{\top}\cZ_s^{\dbP'}\big|^2ds\bigg)^{\frac12}\bigg] \\
     & \qquad\qquad +L\bigg(\frac{e^{-(\eta'-\eta)n}}{\eta'-\eta}\bigg)\sup_{0\leq s\leq \tau}e^{\eta's}\big|\cY_s^{\dbP'}\big|.
   \end{align*}
 By the integrability condition on $f^0$ in Assumption \ref{assum:bsdeF}, and the fact that $(\cY^{\dbP'},\cZ^{\dbP'})\in\cD_{\eta',\tau}^p(\dbP')\times\cH_{\eta',\tau}^p(\dbP')$ by the wellposedness result of backward SDEs in Theorem \ref{thm:bsde}, this implies that
  \begin{align*}
  	\cE^{\dbP,+}_{t\wedge\tau}\bigg[\bigg(\int_{n\wedge\tau}^\tau e^{\eta s}\Big|F_s\big(\cY_s^{\dbP'},\cZ_s^{\dbP'},\widehat\sigma_s\big)\Big|ds\bigg)^p\bigg] \longrightarrow 0, \quad \dbP\mbox{-a.s.}
  \end{align*}
  and therefore $\big|e^{\eta({t\wedge\tau})}\delta\cY^{\dbP', n}_{t\wedge\tau}\big|\longrightarrow 0,$ by \eqref{end:representation}.
 
 \vspace{3mm}
 
 \noindent {\bf 3.} We finally verify the estimate \eqref{apriori-2bsde}. 
 By the representation \eqref{2bsderepresentationt1} proved in the previous step, and following the proof of Proposition \ref{prop:V+DPP}, we may show that 
   \begin{equation*}
   	  \cE^{\cP_0}\bigg[\sup_{0\leq t\leq\tau}e^{p\eta t}|Y_t|^p\bigg] 
   	   \leq C_p \cE^{\cP_0}\bigg[\sup_{0\leq t\leq\tau}\cE_{t}^{\dbP,+}\bigg[\big|e^{\eta\tau}\xi\big|^p 
   	          + \bigg(\int_0^\tau e^{\eta s}\big|f^0_s\big|ds\bigg)^p\bigg]\bigg].
   \end{equation*}
 By Remark \ref{rem:esupassum} we obtain that 
   $ \|Y\|^p_{\cD^p_{\eta,\tau}(\cP_0)} \leq C_p\big(\|\xi\|^p_{\cL^q_{\rho,\tau}(\cP_0)}+(\overline{F}^0_{\rho,q,\tau})^p\big)$.
 As, for each $\dbP\in\cP_0$, $\big(Y,Z,U^\dbP\big)$ is a solution of the RBSDE \eqref{semimartRBSDE}, the required estimate for the $Z$ component follows from Proposition \ref{EstimationRBSDE}.
\end{proof}

\section{Second order backward SDE: existence} \label{sect:existence}

In view of the representation \eqref{2bsderepresentationt1} in Proposition \ref{prop:representation}, we follow the methodology of Soner, Touzi \& Zhang \cite{STZ12, STZ13} by defining the dynamic version of this representation (which requires the additional notations of the next section), and proving that the induced process defines a solution of the 2BSDE. In order to bypass the strong regularity conditions of \cite{STZ12, STZ13}, we adapt the approach of Possama\"{\i}, Tan \& Zhou \cite{PTZ15} to ensure measurability of the process of interest.

\subsection{Shifted space}

We recall the concatenation map of two paths $\omega,\omega'$ at the junction time $t$ defined by 
   \begin{align*}
   	  (\omega\otimes_t\omega')_s:=\omega_s\mathbb{1}_{[0, t)}(s)+(\omega_t+\omega'_{s-t})\mathbb{1}_{[t, \infty)}(s), 
   	   \quad s\ge 0,
   \end{align*}
 and we define the $(t,\omega)-$shifted random variable
   \begin{align*}
   	  \xi^{t, \omega}({\omega}') := \xi(\omega\otimes_t{\omega'}), \quad \mbox{for all }~\om'\in\Om.
   \end{align*}
By a standard monotone class argument, we see that $\xi^{t, \omega}$ is $\cF_{s}-$measurable whenever $\xi$ is $\cF_{t+s}$-measurable. 
In particular, for an $\dbF$-stopping time $\tau$, $t\leq \tau$, then $\vec{\tau}^{t, \omega}:=\tau^{t, \omega}-t$ is still an $\dbF$-stopping time.
Similarly, for any $\dbF$-progressively measurable process $Y$, the shifted process
  \begin{equation*}
     Y_s^{t, \om}(\om') := Y_{t+s}(\om\otimes_t \om'), \quad s\ge 0,
  \end{equation*} 
  is also $\dbF$-progressively measurable. The above notations can naturally be extended to $(\tau,\om)-$ shifting for any finite $\dbF$-stopping time $\tau$.

 \vspace{2mm}

\begin{lemma}\label{oto}
	The mapping $(\om, t, \om')\in  \Omega\times\dbR_+\times\Omega \longmapsto \omega\otimes_t\omega' \in \Om $ is continuous. In particular, if $\xi$ is $\cF_\infty$-measurable function, then $\xi^{\cdot, \cdot}(\cdot)$ is $\cF_\infty \otimes \mathcal{B}\big([0, \infty)\big)\otimes \cF_\infty$-measurable.
\end{lemma}

\begin{proof}
	We directly estimate that
	\begin{equation*}
	\|\omega\otimes_{t}\omega'-\overline{\omega}\otimes_{\overline{t}}\overline{\omega}'\|_\infty
	\leq \|\omega-\overline \omega\|_\infty+ \| \omega' - \overline{\omega}'\|_\infty+ \sup_{s\le |t-\overline t|} \Big( \| \om_{s+\cdot} -\om\|_\infty  + \| \om'_{s+\cdot} -\om'\|_\infty\Big).
	\end{equation*}
\end{proof}

For every probability measure $\mathbb{P}$ on $\Omega$ and $\mathbb{F}$-stopping time $\tau$, there exists a family of regular conditional probability distribution (for short r.c.p.d.) $(\mathbb{P}^\tau_\omega)_{\omega\in \Omega}$, see Theorem 1.3.4 in \cite{SV97}.
\footnote{By definition, an r.c.p.d. satisfies: 
	 \begin{itemize}
	 	\item For every $\omega\in \Omega$, $\mathbb{P}^\tau_\omega$ is a probability measure on $(\Omega, \cF)$;
	 	\item For every $A\in \cF$, the mapping $\omega\longmapsto\mathbb{P}^\tau_\omega(A)$ is $\cF_\tau$-measurable;
	 	\item The family $(\mathbb{P}^\tau_\omega)_{\omega\in \Omega}$ is a version of $\dbP |_{\cF_\tau}$, i.e.
	 	       $\dbE^\dbP[\xi|\cF_\tau](\omega)=\dbE^{\dbP^\tau_\omega}[\xi]$, $\dbP-{\rm a.s.}$ for all $\xi\in\dbL^1(\dbP)$.
	 	\item For every $\omega\in \Omega$, $\mathbb{P}^\tau_\omega(\Omega^{\omega}_\tau)=1$, where
	 	       $\Omega^\omega_\tau:=\{\omega' \in \Omega: \omega'_s=\omega_s,\ 0\leq s\leq \tau(\omega)\}.$
	 \end{itemize} }
The r.c.p.d.~$\mathbb{P}^\tau_\omega$ induces naturally a probability measure $\mathbb{P}^{\tau, \omega}$ on $(\Omega, \cF)$ such that 
    \begin{align*}
        \mathbb{P}^{\tau, \omega}(A):=\mathbb{P}^\tau_\omega(\omega\otimes_\tau A), \quad A\in \cF, 
         \quad {\rm where}\ \omega\otimes_\tau A:= \{\omega\otimes_\tau{\omega'}: {\omega'}\in A\}.
    \end{align*}
It is clear that $ \mathbb{E}^{\mathbb{P}^\tau_\omega}[\xi]=\mathbb{E}^{\mathbb{P}^{\tau, \omega}}[\xi^{\tau, \omega}]$, for every $\cF$-measurable random variable $\xi$.

\subsection{Backward SDEs on the shifted spaces}

For all $\mathbb{P}\in \cP(t, \omega)$, we introduce a family of random horizon BSDEs
   {\small 
   \begin{align} \label{bsdeshift}
   	 \cY^{t, \omega, \mathbb{P}}_{s\wedge \theta}
   	   = \xi^{t, \om} + \int^{\theta}_{s\wedge \theta} F^{t, \om}_r(\cY^{t, \omega, \mathbb{P}}_r, \cZ^{t, \omega, \mathbb{P}}_r, \widehat{\sigma}_r)dr 
   	     - \cZ^{t, \omega, \mathbb{P}}_r \cdot dX_r - d\cN^{t, \omega, \mathbb{P}}_r, \quad s\geq 0, \quad \dbP\mbox{-a.s.}
   \end{align}}

\noindent By Theorem \ref{thm:bsde}, this BSDE admits a unique solution.
Define the value function
  \begin{equation}\label{valfun}
     V_t(\omega):=\sup_{\mathbb{P}\in\cP(t, \omega)}\dbY^{t, \om,\dbP}[\xi,\tau], \quad \mbox{with} \quad\dbY^{t, \om,\dbP}[\xi,\tau] 
                := \mathbb{E}^\mathbb{P}\left[\cY^{t, \omega, \mathbb{P}}_0\right] .
  \end{equation}

In this section, we will prove the following measurability result, which is important for the discussion of the dynamic programming.

\begin{proposition}\label{prop:dbY}
	Under Assumptions \ref{assum:bsdeF}, the mapping 
	   \begin{equation*}
	   	 (t,\om,\dbP)\mapsto \dbY^{t, \om,\dbP}[\xi,\tau]
	   \end{equation*}
	   is $\cB\big([0,\infty)\big)\otimes\cF_\tau\otimes  \cB(\cM_1)$-measurable. 
\end{proposition}

 We will first review in Section \ref{subsec:finite} the finite horizon argument of \cite{PTZ15}, and we next adapt it to our random horizon setting in Section \ref{subsec:infinite}.

\subsubsection{Measurability - finite horizon} \label{subsec:finite} 

Let $\tau =T$, where $T$ is a finite deterministic time. For the convenience of the reader we repeat the argument in \cite{PTZ15} in order to prove the finite horizon version of Proposition \ref{prop:dbY}. 
For each $\mathbb{P}\in \cP_{loc}$, we consider the following shifted BSDE
  \begin{equation} \label{bsdeshift12}
    \cY^{t, \om, \dbP}_s 
      = \xi^{t, \om} + \int^{T-t}_s F^{t, \om}_r\big(\cY^{t, \om, \dbP}_r, \cZ^{t, \om, \dbP}_r, \widehat{\sigma}_r\big)dr
                     -\cZ^{t, \om, \dbP}_r \cdot dX_r - d\cN^{t, \om, \dbP}_r,
       \quad \dbP\mbox{-a.s.}
  \end{equation}
  for $s\in [0,T-t]$.

\vspace{2mm}

\begin{lemma}\label{lem:cYfinite}
	Let $\tau=T$ be a deterministic time. Then, there exists a version of $\cY^{t, \omega, \mathbb{P}}$ such that the mapping 
	$(t, \om,s, \om',\dbP)\in [0, \infty)\times\Om\times[0, \infty)\times \Omega\times  \cP_{loc}\mapsto \cY^{t, \omega, \mathbb{P}}_s(\omega')\in \mathbb{R}$ is $\mathcal{B}\big([0, \infty)\big)\times \cF_\infty\times \mathcal{B}\big([0, \infty)\big)\times\cF_\infty\times\mathcal{B}(\mathcal{M}_{1})$-measurable.
\end{lemma}

\begin{proof}
	We shall exploit the construction of the solution of the BSDE \eqref{bsdeshift12} by the Picard iteration, thus proving that for each step of the iteration, the induced process  $\cY^{n, t, \omega,\mathbb{P}}$  satisfies the required measurability. 
	
	\vspace{3mm}
	
	\noindent {\bf 1.} We start from the first step of the Picard iteration. Take the initial value $\cY^{0,t, \omega, \mathbb{P}}\equiv 0$ and $\cZ^{0,t, \omega, \mathbb{P}}\equiv 0$. 
	Define for all $t\le T$
	   \begin{align}
	     \overline{\cY}^{1,t, \omega, \mathbb{P}}_s 
	       &:= \mathbb{E}^{\dbP}\bigg[\xi^{t, \om} + \int^{T-t}_s F^{t, \om}_r(\cY^{0,t, \omega, \mathbb{P}}_r, \cZ^{0,t, \omega, \mathbb{P}}_r, \widehat{\sigma}_r)dr \bigg|\cF_{s}^+\bigg] \notag \\ 
	       &\,= \mathbb{E}^{\dbP}\bigg[\xi^{t, \om} + \int^{T-t}_0 F^{t, \om}_r(\cY^{0,t, \omega, \mathbb{P}}_r, \cZ^{0,t, \omega, \mathbb{P}}_r, \widehat{\sigma}_r)dr \bigg|\cF_{s}^+\bigg] \nonumber \\
	       &\quad \qquad -\int^s_0 F^{t, \om}_r(\cY^{0,t, \omega, \mathbb{P}}_r, \cZ^{0,t, \omega, \mathbb{P}}_r, \widehat{\sigma}_r)dr, 
	              \qquad s\in [0,T-t].  \label{defstep1}
	   \end{align}
	We extend the definition so that $\overline{\cY}^{1,t, \omega, \mathbb{P}}_s := \xi^{t, \om}$ on $\{s>T-t\}\cap\{t\le T\}$ and $\overline{\cY}^{1,t, \omega, \mathbb{P}}_s \equiv \xi(\om_{T\wedge \cdot})$ for $t>T$.
	By Lemma \ref{oto}, the mapping $\xi^{\cdot, \cdot}(\cdot): \Omega\times [0,T] \times\Omega\longrightarrow \mathbb{R}$ is $\cF_\infty\otimes\cB\big([0,\infty)\big)\otimes\cF_\infty$-measurable. 
	Similarly, the mapping
	  \begin{align*}
	  	 (t, \om, r,\om',\dbP)\mapsto F^{t, \om}_r\big(\om', \cY^{0,t, \omega, \mathbb{P}}_r(\omega'), \cZ^{0,t, \omega, \mathbb{P}}_r(\om'), \widehat{\sigma}_r(\om')\big)
	  \end{align*}
	is $\mathcal{B}\big([0, \infty)\big)\otimes\cF_\infty\otimes\mathcal{B}\big([0, \infty)\big)\otimes\cF_\infty\otimes\mathcal{B}(\cP_{loc})$-measurable, and by the Fubini theorem,
	  \begin{align*}
	  	(t, \om, \om',\dbP) \longmapsto 
	  	{\mathbb 1}_{\{t\le T\}}\int_0^{T-t} F^{t, \om}_r\big(\om', \cY^{0,t, \omega, \mathbb{P}}_r(\om'), \cZ^{0,t, \omega, \mathbb{P}}_r(\om'), \widehat{\sigma}_r(\om')\big)dr 
	  \end{align*}
	is $\mathcal{B}\big([0, \infty)\big)\otimes\cF_\infty\otimes\cF_\infty\otimes\mathcal{B}(\cP_{loc})$-measurable. It follows from Lemma 3.1 in \cite{NN14} that there exists a version, still noted by $\overline{\cY}^{1,t, \omega, \mathbb{P}}$, such that the mapping $(t, \om, \om',\dbP)\mapsto\overline{\cY}^{1,t, \omega, \mathbb{P}}_s(\omega')$ is $\mathcal{B}\big([0, \infty)\big)\otimes\cF_\infty\otimes\cF^+_{s}\otimes\mathcal{B}(\cP_{loc})$-measurable for each $s$. 
	
	\vspace{3mm}
	
	\noindent {\bf 2.} The function $\overline{\cY}^{1,t, \omega, \mathbb{P}}_s$ we just constructed is not necessarily $\dbP$-a.s.~c\`adl\`ag in $s$. 
	We next construct a version $\cY^{1,t, \omega, \mathbb{P}}$ (i.e., $\cY^{1,t, \om, \mathbb{P}}_s = \overline{\cY}^{1,t, \om, \mathbb{P}}_s$, $\mathbb{P}$-a.s.~for all $s$) which is measurable and $\dbP$-a.s.~c\`adl\`ag in $s$. Let $t^n_i:=i2^{-n}(T-t)$, and set for $s\ge 0$:
	   \begin{align*}
	   	 \cY^{1,t, \om, \dbP}_s := \limsup_{m\rightarrow \infty} \cY^{1,m,t,\om,\dbP}_s
	   	   ~~\mbox{with}~~
	   	  \cY^{1, m,t, \om, \dbP}_s
	   	   := \sum_{i=1}^{2^m}\overline{\cY}^{1,t, \om, \dbP}_{t^m_i}{\mathbb 1}_{[t^m_{i-1},t^m_i)}(s)+ {\xi}^{t, \om}{\mathbb 1}_{[T-t, \infty)}(s).
	   \end{align*}
	Clearly, $(t, \om,s,\om', \dbP)\mapsto {\cY}^{1, m, t, \omega, \mathbb{P}}_s(\om')$ is $\mathcal{B}\big([0, \infty)\big)\otimes\cF_\infty\otimes\mathcal{B}\big([0, \infty)\big)\otimes\cF_\infty\otimes\mathcal{B}(\cP_{loc})$-measurable, and so is $(t, \om,s, \om',\dbP)\mapsto {\cY}^{1,  t, \omega, \mathbb{P}}_s(\om')$. 
	Since the filtration $\dbF^{+,\dbP}$ satisfies the usual conditions and the conditional expectation in \eqref{defstep1} is an $\dbF^{+,\dbP}$-martingale, one can prove by a standard argument (see e.g.~\cite[Proposition I 3.14]{KS88}) that ${\cY}^{1,t, \omega, \mathbb{P}}$ is a $\dbP$-a.s.~c\`adl\`ag version of $\overline{\cY}^{1,t, \omega, \mathbb{P}}$.
	
	\vspace{3mm}
	
	\noindent{\bf 3.} Recall the inverse of a nonnegative-definite matrix in Footnote \ref{inverse}.  Define
	   \begin{align}  \label{defstep31}
	     {\cZ}^{1,t, \omega, \mathbb{P}}_s 
	       := \widehat{a}^{-1}_s\limsup_{n\rightarrow \infty}n\left(\langle {\cY}^{1,t, \omega, \mathbb{P}}, X\rangle_s- \langle{\cY}^{1,t, \omega, \mathbb{P}}, X\rangle_{(s-1/n)\vee 0}\right),
 	   \end{align}
	   where the $\limsup$ is componentwise. 
	Clearly, the mapping 
	  $(t, \om,s, \om',\dbP)\mapsto {\cZ}^{1,t, \omega, \mathbb{P}}_s(\om')$ 
	  is $\mathcal{B}\big([0, \infty)\big)\otimes\cF_\infty\otimes\mathcal{B}\big([0, \infty)\big)\otimes\cF_\infty\otimes\mathcal{B}(\cP_{loc})$-measurable. 
	Since ${\cY}^{1,t, \omega, \mathbb{P}}$ is c\`adl\`ag, by the uniqueness of the martingale representation (see e.g.~\cite[Lemma III 4.24]{JS03}), there exists an $\cF^{+, \mathbb{P}}$-martingale $\cN^{1,t, \omega, \mathbb{P}}$ orthogonal to $X$ under $\mathbb{P}$, such that for $t\le T$ and $s\in [0, T-t]$,
	  \begin{equation}\label{bsdeshift2}
	    \cY^{1,t, \om, \dbP}_s 
	      =  \xi^{t, \om} + \int^{T-t}_s F^{t, \om}_r(\cY^{0,t, \om, \dbP}_r, \cZ^{0, t, \om, \dbP}_r, \widehat{\sigma}_r)dr 
	                      - \cZ^{1,t, \om, \dbP}_r \cdot dX_r - d\cN_r^{1,t, \omega, \mathbb{P}},
	               ~\dbP\mbox{-a.s.}
	  \end{equation}
	
	\vspace{3mm}
	
	\noindent{\bf 4.} By replacing $\big(\cY^{0,t, \omega, \mathbb{P}}, \cZ^{0,t, \omega, \mathbb{P}}\big)$ in Steps \textbf{1} - \textbf{3} by $\big(\cY^{n,t, \omega, \mathbb{P}}, \cZ^{n,t, \omega, \mathbb{P}}\big)$, for an arbitrary $n\ge 1$, we may define $\big(\cY^{n+1,t, \omega, \mathbb{P}}, \cZ^{n+1,t, \omega, \mathbb{P}}, \cN^{n+1,t, \omega, \mathbb{P}}\big)$ such that the mappings 
	  \begin{align*}
	  	 (t, \om,s, \om',\dbP)\mapsto \big(\cY^{n+1,t, \omega, \mathbb{P}}(\om'), \cZ^{n+1,t, \omega, \mathbb{P}}(\om')\big)
	  \end{align*}
	  are $\mathcal{B}\big([0, \infty)\big)\otimes\cF_\infty\otimes\mathcal{B}\big([0, \infty)\big)\otimes\cF_\infty\otimes\mathcal{B}(\cP_{loc})$-measurable. 
	By the contracting feature of the Picard iteration, see e.g.~El Karoui, Peng \& Quenez \cite{EPQ97}, we have 
	  \begin{align*}
	  	 \big\|\cY^{n,t, \om, \dbP}-\cY^{t, \om, \dbP} \big\|_{\dbD^2_{T-t, \alpha}(\dbP)}\longrightarrow 0, ~\mbox{ as }~ n\to\infty. 
	  \end{align*}
	As before, we extend the definition so that ${\cY}^{t, \omega, \mathbb{P}}_s := \xi^{t, \om}$ on $\{s>T-t\}\cap\{t\le T\}$ and ${\cY}^{t, \omega, \mathbb{P}}_s \equiv \xi(\om_{T\wedge \cdot})$ for $t>T$.
	Then it follows from \cite[Lemma 3.2]{NN14} that there exists an increasing sequence $\{n^\dbP_k\}_{k\in \mathbb N}\subseteq \dbN$ such that $\dbP\longmapsto n^\dbP_k$ is measurable for each $k$ and
	  \begin{align*}
	  	 \lim_{k\rightarrow \infty}\sup_{0\leq s\leq T-t}\Big|\cY^{n^\dbP_k,t, \om, \dbP}_s-\cY^{t, \om, \dbP}_s\Big| = 0, \quad\dbP\mbox{-a.s}.
	  \end{align*}
	Besides, there exist $\cZ^{t, \omega, \mathbb{P}}\in \mathbb H^2_{T-t, \alpha}$ and $\cN^{t, \omega, \mathbb{P}}\in \mathbb N^2_{T-t, \alpha}$ as limits of the Picard sequence under each $(t, \omega, \mathbb{P})\in [0, T]\times\Omega\times \cP_{loc}$. 
	We conclude that $\big(\cY^{t, \omega, \mathbb{P}}, \cZ^{t, \omega, \mathbb{P}}, \cN^{t, \omega, \mathbb{P}}\big)$ is a solution to the BSDE \eqref{bsdeshift12}, and that $(t, \om, s,\om',\dbP)\mapsto \cY^{t, \omega, \mathbb{P}}_s(\om')$ is 
	  $\mathcal{B}\big([0, \infty)\big)\otimes\cF_\infty\otimes\mathcal{B}\big([0, \infty)\big)\otimes\cF_\infty\otimes\mathcal{B}(\cP_{loc})$-measurable.
	As $\cP_{loc}\subseteq\cB(\cM_1)$, the assertion follows. 
\end{proof}

\begin{remark} 
	In the finite horizon case, Proposition \ref{prop:dbY} is a direct corollary of Lemma \ref{lem:cYfinite}.
\end{remark}

\subsubsection{Measurability - random horizon} \label{subsec:infinite}

Let us return to our construction of the solution of the random horizon BSDE by means of a sequence of finite horizon BSDEs on $[0,\tau_n]$, $n\ge 1$, where $\tau_n:=n\wedge\tau$. For all $(t,\om)\in \llbracket 0, \tau\rrbracket$ and $\mathbb{P}\in \cP_{loc}$, consider the approximating sequence $\big(\cY^{n,t,\om,\dbP},\cZ^{n,t,\om,\dbP}, \cN^{n,t,\om,\dbP}\big)$ defined by: 
  \begin{equation}\label{bsden}
    \cY^{n,t,\om,\dbP}_s
      = \xi^{n,t,\om} + \int_s^{n-t}f^{t,\om}_s\big( \cY^{n,t,\om,\dbP}_s, \cZ^{n,t,\om,\dbP}_s\big)ds
                      - \cZ^{n,t,\om,\dbP}_s dX_s-d\cN^{n,t,\om,\dbP}_s,
  \end{equation} 
 $s\le n-t$, $\dbP^{t,\om}$-a.s., where $\vec{\tau}^{n,\om\otimes_t  X}:=(\tau^{n,\om\otimes_t X}-n)^+$,
  \begin{align*}
  	 \xi^{n,t,\om} := \dbE^{\dbP^{n,\om\otimes_t X}}\Big[ e^{-\mu\vec{\tau}^{n,\om\otimes_t X}}\xi^{n, \om\otimes_t {X}}\Big],
  	  \quad \mbox{and} \quad f^{t,\om}_s(y, z):=F^{t,\om}_s(y, z,\widehat\sigma_s){\mathbb 1}_{\{s\leq(\tau^{t,\om}-t)^+\}}
  \end{align*}
  satisfies Assumption \ref{assum:bsdeF}. 
Then $\big(\cY^{n,t,\om,\dbP},\cZ^{n,t,\om,\dbP}, \cN^{n,t,\om,\dbP}\big)$ is well-defined in $\cD^{p}_{\eta, \tau}(\dbP)\times \cH^{p}_{\eta, \tau}(\dbP)\times \cN^{p}_{\eta, \tau}(\dbP)$ for all $p\in(1,q)$ and $\eta\in[-\mu,\rho)$. 

\begin{proof}[Proof of Proposition \ref{prop:dbY}]
	As $\big(\cY^{n,t,\om,\dbP},\cZ^{n,t,\om,\dbP},\cN^{n,t,\om,\dbP}\big)$ is defined by the finite horizon BSDE, we may apply the results of previous subsection, thus obtaining a version of $\cY^{n, t, \omega, \mathbb{P}}$ such that
	$(t, \om, s,\om',\dbP)\longmapsto \cY^{n, t, \omega, \mathbb{P}}_s(\om')$ is $\mathcal{B}\big([0, \infty)\big)\otimes\cF_\infty\otimes\mathcal{B}\big([0, \infty)\big)\otimes\cF_\infty\otimes\mathcal{B}(\cP_{loc})$-measurable. This in turn implies that the mapping $(t, \om, \dbP)\longmapsto\overline{\mathbb{Y}}^{n, t, \om, \dbP}:=\dbE^\dbP\big[\cY^{n, t, \om, \dbP}_0\big]$ is $\mathcal{B}\big([0, \infty)\big)\otimes\cF_\infty\otimes\mathcal{B}(\cP_{loc})$-measurable. 
	
	By Proposition \ref{prop:rbsde-existence} (with $S=-\infty$), it follows that $\lim_{n\rightarrow \infty}\overline{\mathbb{Y}}^{n, t, \omega, \mathbb{P}}=\mathbb{Y}^{t, \omega, \mathbb{P}}[\xi, \tau].$ Then, the mapping $(t, \omega, \mathbb{P})\mapsto\mathbb{Y}^{t, \omega, \mathbb{P}}[\xi, \tau]$ is $\mathcal{B}\big([0, \infty)\big)\otimes\cF_\infty\otimes\mathcal{B}(\cP_{loc})$-measurable. 
	As $\cP_{loc}\subseteq\cB(\cM_1)$, the mapping $(t,\omega,\mathbb{P})\mapsto\mathbb{Y}^{t,\omega,\mathbb{P}}[\xi,\tau]$ is $\mathcal{B}\big([0, \infty)\big)\otimes\cF_\infty\otimes\mathcal{B}(\cM_1)$-measurable.
\end{proof}

\subsection{Dynamic programming principle} \label{sbsec: dppequa} 

The goal of this section is to prove that the dynamic value process $V$ satisfies the dynamic programming principle. We first focus on the underlying BSDEs for which the dynamic programming principle reduces to the following tower property, where we denote by $\cY[\xi_0,\tau_0]$ the $Y$ component of the solution of the BSDE with the terminal time $\tau_0$ and value $\xi_0$.

\begin{lemma} \label{lem:BSDEshift}
	Let Assumptions \ref{assum:bsdeF} and \ref{assum:bsde-integ} hold true. Then, for all stopping time $\tau_0\leq\tau$, and  $\mathbb{P}\in \cP_{loc}$:
	\begin{enumerate}[{\rm (i)}]
		\item $\dbE^\dbP\big[\cY^\dbP_{\tau_0}\big|\cF_{\tau_0}\big](\om) =\dbY^{\tau_0, \om, \dbP^{\tau_0,\om}}[\xi, \tau]$, for $\dbP$-a.e. $\om\in\Om$.
		\item $\cY^\dbP_{t\wedge\tau_0} [\xi, \tau] 
		= \cY^\dbP_{t\wedge\tau_0}\big[\cY^\dbP_{\tau_0}[\xi, \tau], \tau_0\big]
		= \cY^\dbP_{t\wedge \tau_0}\big[\dbE^\dbP\big[\cY^\dbP_{\tau_0}[\xi, \tau] \big|\cF_{\tau_0}\big], \tau_0\big],$ for all $t\ge 0$.
	\end{enumerate}
\end{lemma}

The proof is omitted as (i) is a direct consequence of the uniqueness of the solution to BSDE, and (ii) is similar to \cite[Lemma 2.7]{PTZ15}. In order to apply the classic measurable selection results, we need the following properties of the probability families $\{\cP(t,\om)\}_{(t,\om)\in\llbracket 0,\tau\rrbracket}$.

\vspace{2mm}

\begin{lemma}\label{lem:ppt-prob}
	The graph $\llbracket\cP\rrbracket:=\{(t, \om, \dbP): \mathbb{P}\in \cP(t, \omega)\}$, is Borel-measurable in $\dbR_+\times\Omega\times \cM_1$. Moreover for all $(t,\om)\in\llbracket 0, \tau\rrbracket$ and all stopping time $\tau_0$ valued in $[t,\tau]$, denoting $\vec{\tau}_0^{t, \omega}:=\tau_0^{t, \omega}-t$, we have:
	\begin{enumerate}[{\rm (i)}]
		\item $\cP(t, \omega)=\cP(t, \omega_{\cdot\wedge t})$, and for all $\dbP\in \cP(t, \omega)$, the r.c.p.d.~$\mathbb{P}^{\vec{\tau}_0^{t, \omega}, \omega' }\in \cP(\tau_0, \omega\otimes_t\omega')$, for $\mathbb{P}$-a.e.~$\omega'\in \Omega$.
		\item For any $\cF_{\vec{\tau}_0^{t, \omega}}$-measurable kernel $\nu:\Omega\to\mathcal{M}_1$ with $\nu({\omega'})\in \cP(\tau_0, \omega\otimes_t\omega')$ for $\mathbb{P}$-a.e. ${\omega'}\in \Omega$, 
		the map $\dbP':=\mathbb{P}\otimes_{\vec{\tau}_0^{t, \omega}} \nu$ defined by 
		\begin{align*}
		    \mathbb{P}'(A)=\int\int (\mathbb{1}_A)^{\vec{\tau}_0^{t, \omega},\omega'}(\omega'')\nu(d\omega''; {\omega'})\mathbb{P}(d\omega'), \quad A\in \cF,
		\end{align*}
		is a probability measure in $\cP(t, \omega)$.
	\end{enumerate}
\end{lemma} 

\begin{proof}
	This follows from \cite[Theorem 4.3]{NvH13}.
\end{proof}

\vspace{2mm}

\begin{theorem}[Dynamic programming for $V$]\label{thm:dpp}
	Let  Assumption \ref{assum:bsdeF} hold true. The mapping $\om\longmapsto V_{\tau_0}(\om)$ is $\cF_{\tau_0}^U$-measurable. 
	Moreover, for $(t,\om)\in\llbracket 0, \tau\rrbracket$, and an $\dbF$-stopping time ${\tau_0}$ with $t\wedge\tau\le\tau_0\le\tau$, we have, denoting $\vec{\tau}_0^{t,\om}:=\tau_0^{t,\om}-t$, for all $p\in(1,q),~\eta\in[-\mu,\rho)$, 
	  \begin{align}  \label{eq:uiV}
	  	\cE^{\cP(t, \omega)}\Big[\big|e^{\eta\vec{\tau}_0^{t, \omega}}(V_{\tau_0})^{t, \omega}\big|^p\Big]<\infty, \quad \mbox{and} \quad
	  	\sup_{\tau_0\le \tau}\cE^{\cP_0} \Big[\big|e^{\eta\tau_0}(V_{\tau_0}) \big|^p\Big]<\infty,
	  \end{align}
      and 
      \begin{align}
      	& V_t (\om) = \sup_{\dbP\in \cP(t, \omega)}\dbY^{t, \om, \dbP}\big[V_{\tau_0} , {\tau_0}\big], \label{dppeq} \\ 
      	& V_{t} = \esssup^\dbP_{\dbP'\in \cP_\dbP(t)}\dbE^{\dbP'}\Big[\cY^{\dbP'}_t\big[V_{\tau_0}, {\tau_0}\big]\Big| \cF_t\Big], 
      	           \quad \dbP\mbox{-a.s.}, ~~ \mbox{for all }~ \dbP\in \cP_0.  \label{dppesup}
      \end{align}
\end{theorem}

\begin{proof}
	Without loss of generality, we assume in the proof that $(t,\om)=(0, {\bf 0})$. 
	
	\vspace{3mm}
	
	\noindent {\bf 1.} It follows from Proposition \ref{prop:dbY} that $(t, \omega, \mathbb{P})\mapsto\mathbb{Y}^{t, \omega, \mathbb{P}}[\xi, \tau]$ is $\mathcal{B}\big([0, \infty)\big)\otimes\cF_\infty\otimes\mathcal{B}(\mathcal{M}_1)$-measurable, and from Lemma \ref{lem:ppt-prob} that $\llbracket\cP\rrbracket$ is analytic.
	By \cite[Poposition 7.47]{BS96}, we know that the mapping 
	  \begin{align*}
	  	 (t, \omega)\mapsto V_t (\omega):=\sup_{\mathbb{P}\in \cP(t, \omega)}\mathbb{Y}^{t,\omega, \mathbb{P}}[\xi, \tau] 
	  \end{align*}
	  is upper semi-analytic and thus universally measurable, i.e., $\mathcal{B}\big([0, \infty)\big)\otimes\cF^U_\infty$-measurable. 
	Finally, note that $V_t(\om)=V_t(\om_{t\wedge\cdot})$. 
	So, it follows from Galmarino's test that $V_{\tau_0} $ is $\cF^U_{{\tau_0}\wedge \tau}$-measurable.
	
	\vspace{3mm}
	
	\noindent {\bf 2.} We next pove \eqref{eq:uiV}. 
	By the measurable selection theorem (see, e.g., \cite[Proposition 7.50]{BS96}), for each $\varepsilon>0$, there exists an $\cF^U_{\tau_0}$-measurable kernel 
	$\nu^\varepsilon: \omega \mapsto \nu^\varepsilon(\omega)\in \cP\big({\tau_0}(\omega), \omega\big)$, such that for all $\omega\in \Omega$
	   \begin{align}  \label{eq:measurableselection}
	   	  e^{\eta{\tau_0}(\omega)} V_{\tau_0} (\omega) 
	   	   \leq e^{\eta\tau_0(\omega)}\mathbb{Y}^{{\tau_0}, \omega, \nu^\varepsilon(\omega)}[\xi, \tau]+\varepsilon.
	   \end{align}
	By Lemma \ref{lem:BSDEshift}, we have 
	   \begin{align*}
	   	  \dbY^{\tau_0, \omega, \nu^\varepsilon(\omega)}[\xi, \tau]
	   	    = \dbE^{\dbP\otimes_{\tau_0}\nu^\varepsilon}\Big[\cY^{\dbP\otimes_{\tau_0}\nu^\varepsilon}_{\tau_0}\Big|\cF_{\tau_0}\Big](\omega),
	   	      \quad \dbP\mbox{-a.s.}, \quad\mbox{for all }~\dbP\in \cP_0.
	   \end{align*}
    Therefore, for $\dbQ\in \cQ_L(\dbP)$, we have
	   \begin{align*}
	     \mathbb{E}^\mathbb{Q}\big[\big|e^{\eta{\tau_0}}V_{\tau_0} \big|^p\big]
	        &\leq \dbE^\dbQ\Big[\Big|\mathbb{E}^{\mathbb{P}\otimes_{\tau_0} \nu^\varepsilon}\Big[e^{\eta{\tau_0}}\cY^{\mathbb{P}\otimes_{\tau_0}\nu^\varepsilon}_{\tau_0}\Big|\cF_{\tau_0}\Big]+\varepsilon\Big|^p\Big] \\
	        &=C_p\Big(\mathbb{E}^{\mathbb{P}\otimes_{\tau_0}\nu^\varepsilon}\Big[\rmD_{\tau_0}^{\mathbb{Q}|\mathbb{P}}\Big|e^{\eta{\tau_0}}\cY^{\mathbb{P}\otimes_{\tau_0}\nu^\varepsilon}_{\tau_0}\Big|^p\Big]+\varepsilon^p\Big)  \\
	        &\leq C_p\bigg(\sup_{\mathbb{Q}'\in \mathcal{Q}_L(\mathbb{P}\otimes_{\tau_0}\nu^\varepsilon)}\mathbb{E}^{\mathbb{Q}'}\Big[\Big|e^{\eta{\tau_0}}\cY^{\mathbb{P}\otimes_{\tau_0} \nu^\varepsilon}_{\tau_0}\Big|^p\Big]+\varepsilon^p\bigg).
	   \end{align*}
	Then, by the estimate \eqref{estznk}, we obtain
	   \begin{align*}
	   	  \sup_{\dbP\in \cP_0} \sup_{\dbQ\in \cQ_L(\dbP)}\dbE^\dbQ\big[\big|e^{\eta{\tau_0}}V_{\tau_0}\big|^p\big]
	   	   \leq C_{p, q} \Big( \|\xi\|^{p}_{\cL^{q}_{\rho,\tau}(\cP_0)}+\big(\overline{F}^0_{\rho,q,\tau}\big)^{p}\Big) + C_p\varepsilon^p,
	   \end{align*}
	  which induces the required estimate by sending $\varepsilon \rightarrow 0$. 
	
	\vspace{3mm}
	
	\noindent {\bf 3.} To prove \eqref{dppeq}, we start by observing that, by the tower property in Lemma \ref{lem:BSDEshift}, we have
	  \begin{align*}
	  	 V_0 = \sup_{\dbP\in \cP_{0}}\dbE^{\dbP}\big[\cY^{\dbP}_0 [\xi, \tau] \big]
	  	     &= \sup_{\dbP\in \cP_{0}}\dbE^\dbP\Big[\cY^{\dbP}_0\big[\cY^{\dbP}_{\tau_0}[ \xi, \tau], \tau_0 \big]\Big] \\
	  	     &=\sup_{\dbP\in \cP_{0}}\dbE^\dbP\Big[\cY^{\dbP}_0\Big[\dbE^\dbP\big[\cY^{\dbP}_{\tau_0} [\xi, \tau] \big|\cF_{\tau_0}\big], \tau_0\Big]\Big].
	  \end{align*}
	Note that, for all $\dbP\in \cP_{0}$, we have by Lemma \ref{lem:BSDEshift} that for $\dbP$-a.e. $\om$,
	  \begin{align*}
	  	 V_{\tau_0} (\om) = \sup_{\mathbb{P}\in \cP_{0}}\mathbb{Y}^{\tau_0, \om, \mathbb{P}}[\xi, \tau] 
	  	   =\sup_{\mathbb{P}\in \cP_{0}}\mathbb{E}^\mathbb{P}\left[\mathcal Y^{\tau_0,\om,\mathbb{P}}_{0}\big[\xi, \tau\big]\right]
	  	   \geq \mathbb{E}^\mathbb{P}\left[\mathcal Y^{\mathbb{P}}_{\tau_0} [\xi, \tau]\Big|\cF_{\tau_0}\right](\omega).
	  \end{align*}
	By the comparison result of Theorem \ref{thm:bsdecomp-stab} (ii) for the BSDE \eqref{bsdeshift12}, we deduce that 
	  \begin{equation} \label{ineq-oneside-ms} 
	  	 V_0 \leq \sup_{\mathbb{P}\in \cP_{0}}\mathbb{E}^\mathbb{P}\left[\cY^{\mathbb{P}}_0\big[V_{\tau_0} , \tau_0 \big] \right] 
	  	     = \sup_{\mathbb{P}\in \cP_{0}}\mathbb{Y}^{0, {\bf 0}, \mathbb{P}}\left[V_{\tau_0} , \tau_0\right].
	  \end{equation}
	To prove the reverse inequality, we use again the measurable selection theorem to deduce the existence of an $\cF^U_{\tau_0}$-measurable kernel $\nu^\varepsilon: \omega \mapsto \nu^\varepsilon(\omega)\in \cP(\tau_0(\om),\om)$ 
	such that \eqref{eq:measurableselection} holds true for $\eta\in [-\mu,\rho)$.
	Define the concatenated probability $\overline{\mathbb{P}}:=\mathbb{P}\otimes_{\tau_0}\nu^\varepsilon$ and note that $\overline{\mathbb{P}}|_{\cF_{\tau_0}}=\mathbb{P}|_{\cF_{\tau_0}}$. 
	Then, by  the stability result of Theorem \ref{thm:bsdecomp-stab} (i) and Lemma \ref{lem:BSDEshift}, we have 
	  \begin{align*}
	  	V_0 \geq \mathbb{E}^{\overline{\mathbb{P}}}\big[\cY^{\overline{\mathbb{P}}}_0 [\xi, \tau]\big] 
	  	   &= \mathbb{E}^{\mathbb{P}}\Big[\cY^\mathbb{P}_0\Big[\mathbb{E}^{\overline{\mathbb{P}}^{\tau_0, \cdot}}\big[\cY^{\tau_0, \cdot,\overline{\mathbb P}^{\tau_0, \cdot}}_{0} [\xi, \tau]\big], \tau_0\Big]\Big]  \\
	  	   &= \mathbb{E}^{\mathbb{P}}\Big[\cY^\mathbb{P}_0\Big[\mathbb{E}^{\nu^\varepsilon(\cdot)}\big[\cY^{\tau_0 ,\cdot,\nu^\varepsilon(\cdot)}_{0}[\xi,\tau]\big], \tau_0\Big]\Big]. 
	  \end{align*}
	By \eqref{eq:measurableselection}, the right hand side is larger than $\mathbb{E}^{\mathbb{P}}\left[\cY^\mathbb{P}_0\left[V_{\tau_0} , \tau_0\right]\right]-C\varepsilon$ for some $C>0$ independent of $\varepsilon$.
	Therefore, $V_0 \ge \mathbb{Y}^{0, {\bf 0}, \mathbb{P}}\left[V_{\tau_0} , \tau_0\right]-C\varepsilon$, and we obtain by sending $\varepsilon\rightarrow 0$ that 
	  \begin{align*}
	     V_0 \geq \sup_{\mathbb{P}\in \cP_{0}}\mathbb{Y}^{0, {\bf 0}, \mathbb{P}}\left[V_{\tau_0} , \tau_0\right]. 
	  \end{align*}
	
	\vspace{3mm}
	
	\noindent {\bf 4.} We finally prove \eqref{dppesup}. Due to the previous step, we know
	   \begin{align*}
	   	  V_t(\om) \geq \dbY^{t,\om,\dbP'}\big[V_{\tau_0},\tau_0\big], \quad\mbox{for all}\quad \dbP'\in \cP(t,\om).
	   \end{align*}
	Now fix a probability measure $\dbP\in \cP_0$. 
	It follows from Lemma \ref{lem:ppt-prob} (i) that for all $\widetilde\dbP\in \cP_\dbP(t)\subseteq\cP_0$ we have $\widetilde\dbP^{t,\om}\in \cP(t,\om)$. 
	So $V_t(\om) \ge \dbY^{t,\om,\widetilde\dbP^{t,\om}}\big[V_{\tau_0},\tau_0\big]$. 
	By Lemma \ref{lem:BSDEshift}, this provides
	   \begin{align*}
	   	 V_t \geq \dbE^{\widetilde\dbP}\Big[\cY^{\widetilde\dbP}_t [V_{\tau_0},\tau_0] \Big|\cF_t\Big], \quad \dbP\mbox{-a.s.}, 
	   \end{align*}
	   and therefore 
	   \begin{align*}
	   	 V_t \geq \esssup^\dbP_{\widetilde\dbP\in \cP_\dbP(t)}\dbE^{\widetilde\dbP}\Big[\cY^{\widetilde\dbP}_t [V_{\tau_0},\tau_0] \Big|\cF_t\Big],
	   	    ~~\dbP\mbox{-a.s.}
	   \end{align*}
	To prove the reverse inequality, we apply the measurable selection theorem on the optimization problem \eqref{dppeq}, to find an $\cF^U_t$-measurable kernel $\nu^\varepsilon: \omega \mapsto \nu^\varepsilon(\omega)\in \cP(t,\omega)$ such that 
	   $V_t(\om) \leq \dbY^{t,\om,\nu(\om)}[V_{\tau_0},\tau_0]+\varepsilon$. 
	By Lemma \ref{lem:ppt-prob}, $\dbP^\varepsilon:=\dbP\otimes_t \nu^\varepsilon\in \cP_0$, and thus $\dbP^\varepsilon\in \cP_\dbP(t)$. 
	Together with Lemma \ref{lem:BSDEshift}, this provides
	  \begin{align*}
	  	V_t \leq \dbE^{\dbP^\varepsilon}\Big[\cY^{\dbP^\varepsilon}_t [V_{\tau_0},\tau_0] \Big|\cF_t\Big] + \varepsilon 
	  	    \leq \esssup^\dbP_{\widetilde\dbP\in \cP_\dbP(t)}\dbE^{\widetilde\dbP}\Big[\cY^{\widetilde\dbP}_t [V_{\tau_0},\tau_0] \Big|\cF_t\Big] + \varepsilon. 
	  \end{align*}
	The required inequality now follows by sending $\varepsilon\rightarrow 0$.	
\end{proof}

\subsection{A c\`adl\`ag version of the value function}

By \cite[Lemma 3.2]{PTZ15}, the right limit
   \begin{align*}
   	  V^+_t(\omega) := \lim_{r\in\mathbb{Q},r\downarrow t} V_r(\omega)
   \end{align*}
  exists $\cP_0$-q.s.~and the process $V^+$ is c\`adl\`ag $\cP_0$-q.s.~with $V_{\tau_0}^+\in\dbL^p_{\eta,\tau}(\dbQ)$ for all $\dbQ\in\bigcup_{\dbP\in\cP_0}\cQ_L(\dbP)$, $\eta\in[-\mu,\rho)$, $p\in(1,q)$, and all stopping times $\tau_0\leq \tau$.

\begin{proposition}[Dynamic programming for $V^+$]\label{prop:V+DPP}
	Under Assumption \ref{assum:bsdeF}, $V^+\in\cD^p_{\eta,\tau}(\cP_0)$ for any $\eta\in[-\mu,\rho)$, $p\in(1,q)$, and for all $\mathbb{F}^+$-stopping times $0\leq \tau_0\leq \tau_1\leq \tau$, and $\dbP\in \cP_0$, we have
	  \begin{align*}
	  	V^ +_{\tau_0} = \esssup^\dbP_{\dbP' \in \cP_\dbP^+(\tau_0)} \cY^{\dbP'}_{\tau_0}\big[V^+_{\tau_1}, \tau_1\big], \qquad \mathbb{P}\mbox{-a.s.}
	  \end{align*}
\end{proposition}

\begin{proof}
	{\bf 1.} For an $\dbF^+-$stopping time $\bar\tau\le\tau$, we introduce the approximating sequence of stopping times $\bar\tau^n := \frac{\lfloor 2^n \bar\tau\rfloor+1}{2^n}$, and we now verify that
	  \begin{align*}
	  	V^+_{\bar\tau}\in \cL^p_{\eta,\bar\tau}(\cP_0) \quad\mbox{and} \quad
	  	\lim_{n\rightarrow\infty}\cE^\dbP\big[\big|e^{\eta\bar\tau} V^+_{\bar\tau} -e^{\eta \bar\tau^n } V_{\bar\tau^n} \big|^p\big] = 0, 
	  	  \quad \mbox{for all }~\dbP\in\cP_0. 
	  \end{align*}
	Indeed, for all $\dbP\in\cP_0$, and $\dbQ\in\cQ_L(\dbP)$:
	  \begin{align*}
	     \dbE^\dbQ\big[\big(e^{\eta\bar\tau} V^+_{\bar\tau} \big)^p\big]
	        = \lim_{n\rightarrow\infty} \dbE^\dbQ\big[\big|e^{\eta \bar\tau^n } V_{\bar\tau^n}\big|^p\big] 
	        \leq \sup_{\tau'\le\tau}\cE^{\cP_0}\big[\big|e^{\eta \tau' } V_{\tau'}\big|^p\big]
	        =: v_p < \infty,
	  \end{align*}
	  by \eqref{eq:uiV} in Theorem \ref{thm:dpp}, implying that $\cE^{\cP_0}\big[\big|e^{\eta\bar\tau} V^+_{\bar\tau} \big|^p\big] \leq v_p$. 
	Then $\delta_n:=\big|e^{\eta\bar\tau} V^+_{\bar\tau} -e^{\eta \bar\tau^n } V_{\bar\tau^n}\big|$ satisfies for an arbitrary $m\ge 1$: 
	  \begin{align*}
	  	 \cE^\dbP\big[\delta_n^p\big]
	  	  \leq \cE^\dbP\big[\delta_n^p\Ind_{\{\tau\ge m\}}\big] + \cE^\dbP\big[\delta_n^p\Ind_{\{\tau<m\}}\big] 
	  	  \leq 2 v_{p'}^{\frac{p}{p'}} \cE^\dbP\big[\Ind_{\{\tau\ge m\}}\big]^{1-\frac{p}{p'}} + C_m\big( \dbE^\dbP\big[\delta_n^{p'}\big] \big)^{\frac{p}{p'}},
	  \end{align*}
	  which implies the required convergence.
	
	\vspace{3mm}
	
	\noindent {\bf 2.} We now prove that $V^+_{\tau_0}\ge \cY^{\dbP'}_{\tau_0}\big[V^+_{\tau_1},\tau_1\big]$, $\dbP-$a.s.~for all $\dbP'\in\cP_\dbP^+(\tau_0)$, where the right hand is well defined by the integrability of $V^+$ obtained in step \textbf{1}. 
	Recall from Theorem \ref{thm:dpp} that
	   \begin{align*}
	      V_{\tau_0^m} = \esssup_{\dbP'\in\cP_{\dbP}(\tau_0^m)}^\dbP\dbE^{\dbP'}\Big[\cY^{\dbP'}_{\tau_0^m}\big[V_{\tau_1^n},\tau_1^n\big]\Big|\cF_{\tau^m_0}\Big], \quad \dbP\mbox{-a.s.}
	   \end{align*}
	Since for each $m\in\mathbb{N}$, $\cP_\dbP(\tau_0^m)\subseteq\cP_\dbP^+(\tau_0)=\bigcup_{h>0}\cP_\dbP(\tau_0+h)$, we have for any $\dbP'\in\cP_\dbP^+(\tau_0)$ and for $m$ large enough that 
	   \begin{align*}
	       V_{\tau_0^m} 
	         \geq \dbE^{\dbP'}\Big[\cY^{\dbP'}_{\tau_0^m} \big[V_{\tau_1^n},\tau_1^n] \Big|\cF_{\tau_0^m}\Big], \quad \dbP\mbox{-a.s.},
	   \end{align*}
	   where $\tau_0^m$ and $\tau_1^n$ are defined from $\tau_0$ and $\tau_1$ as in the previous step.
	By the stability result of BSDEs in Proposition \ref{prop:BSDEstab2times}, and the result of Step \textbf{1} of the present proof, we have
	   \begin{align*}
	   	  &\lim_{n\rightarrow\infty}
	   	   \Big\|\cY^{\dbP'}_{\tau_0^m}\big[V_{\tau_1^n},\tau_1^n]-\cY^{\dbP'}_{\tau_0^m} \big[V^+_{\tau_1},\tau_1\big]\Big\|_{\dbL^p_{\eta,\tau_0^m}(\dbP')} \\
	   	  &\qquad\leq \lim_{n\rightarrow\infty}\Big\|\cY^{\dbP'}_{\tau_0^m} 
	   	         \big[V_{\tau_1^n},\tau_1^n]-\cY^{\dbP'}_{\tau_0^m}\big[V^+_{\tau_1},\tau_1\big]\Big\|_{\cL^p_{\eta,\tau_0^m}(\dbP')}
	   	     =0.
	   \end{align*}
	Then,  
	   \begin{align*}
	   	 V_{\tau_0^m} \geq \lim_{n\rightarrow\infty} \dbE^{\dbP'}\Big[\cY^{\dbP'}_{\tau_0^m} [V_{\tau_1^n},\tau_1^n]\Big|\cF_{\tau_0^m} \Big]
	   	              = \dbE^{\dbP'}\Big[\cY^{\dbP'}_{\tau_0^m} \big[V^+_{\tau_1},\tau_1\big]\Big| \cF_{\tau_0^m}\Big], \quad \dbP\mbox{-a.s.},
	   \end{align*} 
	  and therefore
	   \begin{align*}
	   	 V^+_{\tau_0} = \lim_{m\rightarrow\infty} V_{\tau_0^m} 
	   	   \geq \lim_{m\rightarrow\infty}  \dbE^{\dbP'}\Big[\cY^{\dbP'}_{\tau_0^m} \big[V^+_{\tau_1},\tau_1\big] \Big| \cF_{\tau_0^m}\Big]
	   	   = \dbE^{\dbP'}\Big[\cY^{\dbP'}_{\tau_0} \big[V^+_{\tau_1},\tau_1\big] \Big|\cF_{\tau_0}^+\Big],
	   \end{align*}
	where the last equality is due to $\cY^{\dbP'} \big[V^+_{\tau_1},\tau_1\big] \in \cD^p_{\eta,\tau_1}(\dbP')$.
	
	\vspace{3mm}
	
	\noindent {\bf 3.} We next prove the reverse inequality. By the comparison result together with the last step of the present proof, we have
	  \begin{align} \label{V+DPPstep}
	  	\esssup^\mathbb{P}_{\mathbb{P}'\in \cP_\dbP^+(\tau_0)} \cY^{\mathbb{P}'}_{\tau_0}\big[V^+_{\tau_1}, \tau_1\big]
	  	   \geq\esssup^\mathbb{P}_{\mathbb{P}'\in \cP_\dbP^+(\tau_0)} \cY^{\mathbb{P}'}_{\tau_0}\big[\cY^{\dbP'}_{\tau_1}[\xi,\tau], \tau_1\big] 
	  	   = \esssup^\mathbb{P}_{\mathbb{P}'\in \cP_\dbP^+(\tau_0)} \cY^{\mathbb{P}'}_{\tau_0} [\xi,\tau].
	  \end{align}
	So it remains to prove that
	  \begin{align} \label{remainstoprove}
	  	 V^+_{\tau_0} \leq \esssup^\mathbb{P}_{\mathbb{P}'\in \cP_\dbP^+(\tau_0)} \cY^{\mathbb{P}'}_{\tau_0} [\xi,\tau].
	  \end{align}
	In the remainder of Step \textbf{3}, we omit the parameter $[\xi,\tau]$ without causing confusion. For any $\eta\in[-\mu,\rho)$, we obtain by the dominated convergence theorem together with the estimate \eqref{eq:uiV} of Theorem \ref{thm:dpp} that
	  \begin{align} \label{longineqV+}
	     e^{\eta\tau_0}V^+_{\tau_0} 
	       &= \lim_{n\rightarrow\infty}\dbE\big[e^{\eta\tau_0^n}V_{\tau_0^n} \big| \cF_{\tau_0}^+\big]
	        = \lim_{n\rightarrow\infty}\dbE\bigg[e^{\eta\tau_0^n}\esssup^\mathbb{P}_{\mathbb{P}'\in \cP_\dbP(\tau_0^n)}\dbE^{\dbP'}\big[\cY^{\mathbb{P}'}_{\tau^n_0} \big|\cF_{\tau_0^n}\big]\bigg| \cF_{\tau_0}^+\bigg] \nonumber \\
	       &\leq \lim_{n\rightarrow\infty} \dbE\bigg[e^{\eta\tau_0^n}\esssup^\mathbb{P}_{\mathbb{P}'\in \cP_\dbP^+(\tau_0)}\dbE^{\dbP'}\big[\cY^{\mathbb{P}'}_{\tau^n_0} \big|\cF_{\tau_0^n}\big]\bigg| \cF_{\tau_0}^+\bigg] \nonumber \\
	       &= \lim_{n\rightarrow\infty} \esssup^\dbP_{\dbP'\in\cP_ \dbP^+(\tau_0)}\dbE^{\dbP'}\Big[e^{\eta\tau_0^n}\cY^{\dbP'}_{\tau_0^n} \Big| \cF^+_{\tau_0} \Big] \nonumber \\
	       &= \lim_{n\rightarrow\infty}\esssup^\dbP_{\dbP'\in\cP_\dbP^+(\tau_0)}\bigg\{e^{\eta\tau_0}\cY^{\dbP'}_{\tau_0}
	           + \dbE^{\dbP'}\bigg[\int_{\tau_0}^{\tau_0^n}e^{\eta s} \big(f_s(\cY^{\dbP'}_s, \cZ^{\dbP'}_s)+\eta\cY^{\dbP'}_s\big)ds\bigg| \cF^+_{\tau_0} \bigg]\bigg\}.
	  \end{align}
	By the Lipschitz property of $F$ in Assumption \ref{assum:bsdeF}, we estimate that
	   \begin{align*}
	   	  \dbE^{\dbP'}\bigg[\int_{\tau_0}^{\tau_0^n}e^{\eta s} \Big(f_s\big(\cY^{\dbP'}_s, \cZ^{\dbP'}_s\big)+\eta\cY^{\dbP'}_s\Big)ds\bigg| \cF^+_{\tau_0} \bigg]
	   	   \leq C 2^{-n}\Big(\|\xi\|_{\cL^q_{\rho,\tau}(\cP_0)}+\big(\overline{F}^0_{\rho,q,\tau}\big)\Big),
	   \end{align*}
	   which provides \eqref{remainstoprove} in view of \eqref{longineqV+}.
	
	\vspace{3mm}
	
	\noindent {\bf 4.} It remains to prove that $V^+$ inherits the integrability property of $V$. 
	By Proposition \ref{prop:V+DPP},
	   \begin{align*}
	   	  e^{p\eta t}\big|V^+_t\big|^p 
	   	     = e^{p\eta t}\bigg|\esssup^\dbP_{\dbP'\in\cP_\dbP^+(t)}\cY^{\dbP'}_t[\xi,\tau]\bigg|^p 
	   	     = \esssup^\dbP_{\dbP'\in\cP_\dbP^+(t)}e^{p\eta t}\big|\cY^{\dbP'}_t[\xi,\tau]\big|^p. 
	   \end{align*}
	As in the proof of Theorem \ref{thm:bsde}, we may find for each $\dbP'$ a measure $\dbQ^{\dbP'}\in\cQ_L$, such that 
	   \begin{align*}
	   	 e^{\eta t}\big|\cY_t^{\dbP'}[\xi,\tau]\big| 
	   	   \leq \dbE^{\dbQ^{\dbP'}}\bigg[e^{\eta\tau}|\xi|+\int_0^\tau e^{\eta s}\big|f^0_s\big|ds\bigg|\cF_t^{+}\bigg]
	   \end{align*}
	Then, 
	   \begin{align*}
	   	  e^{p\eta t}\big|V^+_t\big|^p 
	   	   \leq C_p \esssup_{\dbP'\in\cP_\dbP^+(t)}^\dbP\esssup^{\dbP'}_{\dbQ\in\cQ_L(\dbP')}\dbE^{\dbQ}\bigg[e^{p\eta\tau}|\xi|^p 
	   	         + \bigg(\int_0^\tau e^{\eta s}|f^0_s|ds\bigg)^p\bigg|\cF_t^{+}\bigg].
	   \end{align*}
	   and therefore, 
	   \begin{align*}
	   	  \cE^{\cP_0}\bigg[\sup_{0\leq t\leq\tau}e^{p\eta t}\big|V^+_t\big|^p\bigg] 
	   	   \leq C_p \cE^{\cP_0}\bigg[\sup_{0\leq t\leq\tau}\cE^{\dbP,+}_{t}\bigg[e^{p\eta\tau}|\xi|^p+\bigg(\int_0^\tau e^{\eta s}\big|f^0_s\big|ds\bigg)^p\bigg]\bigg]<\infty,
	   \end{align*}
	   which induces the required result by Remark \ref{rem:esupassum}. 
\end{proof}

\subsection{Proof of Theorem \ref{mainthm}: existence}

\begin{proof}
\noindent {\bf 1.} We first prove the existence of a process $Z$ and a family $(U^\dbP)_{\dbP\in\cP_0}$ such that for all $p\in(1,q)$ and $\eta\in[-\mu,\rho)$, $(Z, U^\dbP)\in \cH^p_{\eta,\tau}\big(\dbP,\dbF^{+,\dbP}\big)\times \cU^p_{\eta,\tau}\big(\dbP,\dbF^{+,\dbP}\big)$, and $U^\dbP$ is a c\`adl\`ag $\dbP$-supermartingale, $[U^\dbP, X]=0$, and 
   \begin{equation} \label{repre}
       V^+_{t\wedge \tau} 
         = \xi + \int_{t\wedge \tau}^\tau F_s\big(V^+_s,  Z^\dbP_s,\widehat \sigma_s\big)ds
               -\int_{t\wedge \tau}^\tau \big(Z^\dbP_s \cdot dX_s + dU^\dbP_s\big),~~t\ge 0,~~\dbP\mbox{-a.s.}
   \end{equation}
Fix $\dbP\in\cP_0$. 
As for any $p<p'<q$, $V^+\in \cD^{p'}_{\eta,\tau}(\cP_0)$, by Proposition \ref{prop:V+DPP}, it follows from Theorem \ref{thm:rbsde} that there exists an unique solution $(Y^\dbP,Z^\dbP,U^\dbP)\in\cD^{p}_{\eta,\tau}(\dbP)\times\cH^{p}_{\eta,\tau}(\dbP)\times\cU^{p}_{\eta,\tau}(\dbP)$ to the RBSDE:
   \begin{equation} \label{semimartRBSDE}
     \begin{cases}
       \displaystyle Y^\dbP_{\cdot \wedge\tau} = \xi +\int_{\cdot\wedge\tau}^\tau f_s(Y^\dbP_s,Z^\dbP_s)ds - \int_{\cdot\wedge\tau}^\tau(Z^\dbP_s\cdot dX_s + dU^\dbP_s), \\
        \vspace{2mm}
       Y^\dbP \geq V^+, \quad \dbP\mbox{-a.s.} \\
        \vspace{2mm}
       \displaystyle \dbE^\dbP\bigg[\int_0^{t\wedge\tau} \big(1\wedge(Y^\dbP_{r-}-V^+_{r-})\big)dU_r\bigg] = 0, \quad \mbox{for all }~t\ge 0.
     \end{cases}
   \end{equation}
Following the same argument as in \cite{STZ12}, see also \cite[Lemma 3.6]{PTZ15}, we now verify that $Y^\dbP = V^+$, $\dbP$-a.s. 
Indeed, assume to the contrary that $2\varepsilon:=Y^\dbP_0-V_0^+>0$ (without loss of generality), so that 
   \begin{equation*}
   	  \tau_{\varepsilon}:=\inf\big\{t>0 : e^{\eta_t}Y^\dbP_t\leq e^{\eta_t}V_t^++\varepsilon\big\}>0, \quad \dbP\mbox{-a.s.} 
   \end{equation*}
Notice that $\tau_{\varepsilon}\le\tau$, as the two processes are equal to $\xi$ at time $\tau$. From the Skorokhod condition, it follows  that $U^\dbP$ is a martingale on $[0,\tau_\varepsilon]$, thus reducing the RBSDE to a BSDE on this time interval. Denoting as usual by $\cY^\dbP\big[V^+_{\tau_\varepsilon},\tau_\varepsilon\big]$, we obtain by standard BSDE techniques that, for some probability measure $\dbQ\in\cQ_L(\dbP)$,
   \begin{align*}
   	  Y^\dbP_0 &\leq \cY_0^\dbP\big[V^+_{\tau_\varepsilon},\tau_\varepsilon\big] 
   	                 + \dbE^\dbQ\big[e^{\eta\tau_\varepsilon}\big(Y^\dbP_{\tau_\varepsilon}-V^+_{\tau_\varepsilon}\big)\big]
   	            \leq \cY_0^\dbP\big[V^+_{\tau_\varepsilon},\tau_\varepsilon\big] + \varepsilon 
   	            \leq V^+_0 + \varepsilon,    
   \end{align*}
   where the last inequality follows from the crucial dynamic programming principle of Proposition \ref{prop:V+DPP}. 
By the definition of $\varepsilon$, the last inequality cannot happen. 

Consequently $Y^\dbP=V^+$. In particular, $V^+$ is a c\`adl\`ag semimartingale which would satisfy \eqref{repre} once we prove that the family $\{Z^\dbP\}_{\dbP\in\cP_0}$ may be aggregated. By Karandikar \cite{KAR95}, the quadratic covariation process $\langle V^+,X\rangle$ may be defined on $\dbR_+\times\Om$. Moreover, $\langle V^+,X\rangle$ is $\cP_0$-q.s.~continuous and hence is $\dbF^{+,\cP_0}$-predictable, or equivalently $\dbF^{\cP_0}$-predictable. Similar to the proof of \cite[Theorem 2.4]{Nut15}, we can define a universal $\dbF^{\cP_0}$-predictable process $Z$ by $ Z_tdt:=\widehat{a}^{-1}_t d\langle V^+,X\rangle_t$, and by comparing to the corresponding covariation under each $\dbP\in\cP_0$, we see that $Z=Z^\dbP$, $\dbP$-a.s.~for all $\dbP\in\cP_0$. This completes the proof of \eqref{repre}.

\vspace{3mm}

\noindent {\bf 2.} It remains to prove that the family of supermartingales $\big\{U^\dbP\big\}_{\dbP\in\cP_0}$ satisfies the minimality condition. 
Let $0\le s\le t$, $\dbP\in\cP_0$, $\dbP'\in\cP_\dbP^+(s\wedge\tau)$, and denote by $\big(\cY^{\dbP'},\cZ^{\dbP'},\cN^{\dbP'}\big)$ the solution of the BSDE with parameters $(F,\xi)$. 
Define $\delta Y:=V^+-\cY^{\dbP'}$, $\delta Z:=Z-\cZ^{\dbP'}$ and $\delta U:=U^{\dbP'}-\cN^{\dbP'}$. 
By It\^o's formula, we have for $\alpha\in [-\mu, \rho)$, 
  \begin{align*}
     & e^{\alpha(s\wedge\tau)}\delta Y_{s\wedge\tau}  \\
     & \quad = \int_{s\wedge\tau}^\tau e^{\alpha(r\wedge\tau)}
                \Big\{\big(F_r\big(V^+_r,Z_r,\widehat\sigma_r\big)-F_r\big(\cY^{\dbP'}_r,\cZ^{\dbP'}_r,\widehat\sigma_r\big)-\alpha \delta Y_r\big)dr 
                    -\delta Z_r \cdot dX_r - d\delta U_r\Big\} \\
     & \quad = \int_{s\wedge\tau}^\tau e^{\alpha(r\wedge\tau)}
          \Big\{\big(a_r^{\dbP'}\delta Y_r + b_r^{\dbP'}\cdot\widehat\sigma_r^{\top}\delta Z_r\big)dr- \big(\widehat\sigma_r^{\top}\delta Z_r\cdot dW_r + d\delta U_r\big)\Big\}, 
  \end{align*}
for some bounded processes $a^{\dbP'}$ and $b^{\dbP'}$, by Assumption \ref{assum:bsdeF}. This provides that 
  \begin{align*}
  	 \Gamma^{\dbP'}_{t\wedge\tau}e^{\alpha(t\wedge\tau)}\delta Y_{t\wedge\tau} - \Gamma^{\dbP'}_{s\wedge\tau}e^{\alpha(s\wedge\tau)}\delta Y_{s\wedge\tau}
  	   = \int_{s\wedge\tau}^{t\wedge\tau}\Gamma^{\dbP'}_r e^{\alpha r}\Big\{\big(\delta Y_r b_r^{\dbP'}+\widehat\sigma_r^{\top}\delta Z_r\big)\cdot dW_r + d\delta U_r\Big\},
  \end{align*}
where 
  \begin{align*}
  	 \Gamma_r^{\dbP'}:=\exp\bigg(\int_{s\wedge\tau}^r\Big(a_u^{\dbP'}-\frac{1}{2}\big|b_u^{\dbP'}\big|^2\Big)du + \int_{s\wedge\tau}^r b_u^{\dbP'}\cdot dW_u\bigg).
  \end{align*}
Recall that $\delta Y\ge 0$, and $U^{\dbP'}$ is a $\dbP'-$supermartingale with decomposition $U^{\dbP'}=N^{\dbP'}-K^{\dbP'}$, for some $\dbP'-$martingale $N^{\dbP'}$ and nondecreasing process $K^{\dbP'}$. 
Then, taking conditional expectation $\dbE^{\dbP'}_{s\wedge\tau}[.]:=\dbE^{\dbP'}\big[.\big| \cF_{s\wedge\tau}^+\big]$, we obtain
  \begin{align*}
 	 e^{|\alpha|t}\delta Y_{s\wedge\tau} 
 	  \geq \dbE^{\dbP'}_{s\wedge\tau}\bigg[\int_{s\wedge\tau}^{t\wedge\tau}\Gamma^{\dbP'}_r dK^{\dbP'}_r\bigg] 
 	  \geq \dbE^{\dbP'}_{s\wedge\tau}\bigg[\gamma_{s,t}^{\dbP'}\int_{s\wedge\tau}^{t\wedge\tau} dK^{\dbP'}_r\bigg], 
 	   ~~\mbox{with}~~ 
 	 \gamma_{s,t}^{\dbP'}:=\inf_{s\wedge\tau \leq r\leq t\wedge \tau}\Gamma^{\dbP'}_r, 
  \end{align*}
  and we then obtain by the Cauchy-Schwarz and H\"older's inequality: 
  \begin{align*}
     0 \leq \dbE^{\dbP'}_{s\wedge\tau}\bigg[\int_{s\wedge\tau}^{t\wedge\tau} \!\!\!-d\delta U_r\bigg]
       &= \dbE^{\dbP'}_{s\wedge\tau}\bigg[\int_{s\wedge\tau}^{t\wedge\tau} \!\!dK^{\dbP'}_r\bigg] \\
       &\leq \dbE^{\dbP'}_{s\wedge\tau}\bigg[\gamma^{\dbP'}_{s,t} \int_{s\wedge\tau}^{t\wedge\tau} dK^{\dbP'}_r\bigg]^\frac{1}{2}  
                 \big(C^{\dbP,p}_{s,t}\big)^{\frac{1}{2p}}
                 \dbE^{\dbP'}_{s\wedge\tau}\Big[\big(\gamma_{s,t}^{\dbP'}\big)^{-\widetilde{p}}\Big]^{\frac{1}{2\widetilde{p}}} \\
       &\leq C e^{\frac{1}{2}|\alpha|t} \big(C^{\dbP,p}_{s,t}\big)^{\frac{1}{2p}}\;\big(\delta Y_{s\wedge\tau}\big)^{\frac{1}{2}}, 
  \end{align*}
  where $p\in(1,q)$, $p^{-1}+\widetilde{p}^{-1}=1$, and
  \begin{align}  \label{CPpst}
  	 C^{\dbP,p}_{s,t} 
  	   := \esssup^{\dbP'}_{\dbP'\in\cP^+_\dbP(s\wedge\tau)}\dbE^{\dbP'}_{s\wedge\tau}\bigg[\bigg(\int_{s\wedge\tau}^{t\wedge\tau} dK^{\dbP'}_r\bigg)^p\bigg].
  \end{align}
Now, the minimality condition in Definition \ref{def:2bsde} follows immediately from Proposition \ref{prop:V+DPP}, provided that $C^{\dbP,p}_{s,t}<\infty$, $\dbP-$a.s. which we now prove.

The family 
   \begin{align*}
   	  \bigg\{\dbE^{\dbP'}\bigg[\bigg|\int^{t\wedge \tau}_{s\wedge \tau} dK^{\dbP'}_{s}\bigg|^p\,\bigg|\cF^+_{s\wedge \tau}\bigg], ~~\dbP'\in\cP^+_\dbP(t\wedge\tau)\bigg\}
   \end{align*}
   is directed upward.\footnote{This follows from the same argument as in \cite[Theorem 4.3]{STZ12}. 
For $\dbP_1,\dbP_2\in\cP_\dbP^+(s\wedge\tau)$, denote $\kappa^{\dbP_i}:=\dbE^{\dbP_i}_{s\wedge \tau}\big[\big|\int^{t\wedge\tau}_{s\wedge\tau} dK^{\dbP'}_{r}\big|^p\big]$, and $A:=\{\kappa^{\dbP_1}>\kappa^{\dbP_2}\}$, and define $E\in\cF\longmapsto\dbP_3(E):= \dbP_1(A\cap E) + \dbP_2(A^c\cap E)$; clearly, $\dbP_3\in\cP^+_{\mathbb{P}}(t\wedge\tau)$, and $\kappa^{\dbP_3}=\kappa^{\dbP_1}\vee\kappa^{\dbP_2}$.}
Then, it follows from \cite[Proposition V-1-1]{Nev75} that the $\esssup$ in \eqref{CPpst} is attained as an increasing limit along some sequence $\{\dbP_n\}_{n\in\dbN}\subseteq\cP^+_\dbP(s\wedge\tau)$. By the monotone convergence theorem, we see that 
  \begin{align*}
  	 \dbE^\dbP\big[C^{\dbP,p}_{s,t}\big] 
  	   = \lim_{n\to\infty}\uparrow\dbE^\dbP\bigg[\dbE^{\dbP_n}_{s\wedge \tau} \bigg[\bigg(\int^{t\wedge\tau}_{s\wedge \tau} dK^{\mathbb{P}_n}_r\bigg)^p\bigg]\bigg]
  	   \leq C\big\| U \big\|_{\cU^p_{\eta,\tau}(\mathcal{P}_0)}^p
  	   < \infty,
  \end{align*}
  by Proposition \ref{EstimationRBSDE} together with the fact that $\|V^+\|_{\cD^{p'}_{\eta,\tau}(\mathcal{P}_0)}<\infty$ due to the wellposedness of the RBSDE. 
Hence, $C^{\dbP,p}_{s,t}<\infty$, $\dbP$-a.s.
\end{proof}

\section{Connection to a fully nonlinear elliptic path-dependent PDE}

In this section, we present an example of pricing under volatility uncertainty from the so-called robust finance.  
The canonical process $X$ represents the price process a financial asset. 
The objective is the hedging of the derivative security defined by the payoff $\xi(X)$ at some maturity ${\rm H}_Q$ defined as the exiting time from some set $Q$. 

In contrast with the standard approach, we assume that the volatility is uncertain. 
The probability space $(\Omega, \cF)$ is endowed with a family of probability measures $\cP^{\rm UVM}$, 
  \begin{align*}
  	\cP^{\rm UVM}:=\left\{\dbP: X \mbox{ is a continuous $\dbP$-martingale and } \frac{d\langle X\rangle_t}{dt}\in\big[\underline{\sigma}^2,\overline{\sigma}^2\big],~\dbP\mbox{-a.s.} \right\}.
  \end{align*}
We assume that $\underline{\sigma}>0$. 
The superhedging problem under volatility uncertainty was initially formulated by Denis and Martini \cite{DM06} and Neufeld and Nutz \cite{NN13}. 
Their superhedging result expresses the cost of robust superhedging as 
  \begin{align} \label{eq:defu0}
  	 u_0 := \cE^{\cP^{\rm UVM}}\left[e^{-r{\rm H}_Q}\xi(X_{{\rm H}_Q\wedge\cdot})\right], 
  \end{align}
  where $r>0$ is the discount rate, $Q$ is a bounded open convex subset of $\dbR^d$ containing ${\bf 0}$, and  
  \begin{equation*}
  	 {\rm H}_Q:=\inf\{t\geq 0:\omega_t\notin Q\}. 
  \end{equation*}

Since the family $\cP^{\rm UVM}$ is saturated, we consider the following saturated 2BSDE
  \begin{align} \label{eq:2BSDEI_example}
     Y_{t\wedge{\rm H}_Q} = \xi - \int_{t\wedge{\rm H}_Q}^{{\rm H}_Q} rY_sds - \int_{t\wedge{\rm H}_Q}^{{\rm H}_Q}Z_s\cdot dX_s + \int_{t\wedge{\rm H}_Q}^{{\rm H}_Q} dK_s, \quad \cP^{\rm UVM}\mbox{-q.s.}
  \end{align}

\vspace{2mm}

\begin{proposition}
	Let $(Y,Z)$ be the solution of the 2BSDE \eqref{eq:2BSDEI_example} above. 
	Then, 
	 \begin{align*}
	 	u_0 = \sup_{\dbP\in\cP^{\rm UVM}}\dbE^{\dbP}[Y_0]. 
	 \end{align*}
\end{proposition}

\begin{proof}
	By Proposition \ref{prop:representation}, the solution of the 2BSDE \eqref{eq:2BSDEI_example} can be represented as the supremum of the solution of BSDEs
	  \begin{align*}
	  	Y_0 = \esssup_{\dbP'\in\cP^+_\dbP(0)}^\dbP\cY^{\dbP'}_0, \quad \dbP\mbox{-a.s. for all }\dbP\in\cP^{\rm UVM}, 
	  \end{align*}
	  where for all $\dbP\in\cP^{\rm UVM}$ is the solution of the following BSDE under $\dbP$
	  \begin{equation*}
	  	\cY_0^\dbP = \xi - \int_0^{{\rm H}_Q} r\cY^\dbP_sds - \int_0^{{\rm H}_Q}\cZ_s^\dbP\cdot dX_s, \quad \dbP\mbox{-a.s.}
	  \end{equation*}
	By It\^o's formula, we obtain that 
	  \begin{align*}
	  	e^{-r{\rm H}_Q}\cY^\dbP_{{\rm H}_Q} = \cY_0^\dbP + \int_0^{{\rm H}_Q}e^{-rs}\cZ_s^\dbP\cdot dX_s, \quad \dbP\mbox{-a.s.}
	  \end{align*}
	Taking conditional expectation implies 
	  \begin{equation*}
	  	\cY^\dbP_0 = \dbE^{\dbP}\left[e^{-r{\rm H}_Q}\xi\big|\cF_0^+\right], \quad \dbP\mbox{-a.s.}, 
	  \end{equation*}
	  therefore,
	  \begin{equation*}
	  	\sup_{\dbP\in\cP^{\rm UVM}}\dbE^{\dbP}[Y_0] 
	  	  = \sup_{\dbP\in\cP^{\rm UVM}}\dbE^{\dbP}\left[\esssup^\dbP_{\dbP'\in\cP^+_\dbP(0)}\dbE^{\dbP'}\left[e^{-r{\rm H}_Q}\xi\big|\cF_0^+\right]\right] 
	  	  = u_0,
	  \end{equation*}
	  where the last equality follows from the fact, which is easy to show, that the family
	  \begin{equation*}
	  	\left\{\dbE^{\dbP'}\left[e^{-r{\rm H}_Q}\xi\big|\cF_0^+\right],~\dbP'\in\cP^+_\dbP(0)\right\}
	  \end{equation*} 
	  is upward directed. 
\end{proof}

\vspace{2mm}

Through the robust hedging problem, we obtain the connection between the random horizon 2BSDE above and the elliptic path-dependent PDE below. 
Under the assumption that $\xi:\Omega\to\dbR$ is bounded uniformly continuous on the boundary $\partial \cQ$, by \cite[Proposition 7.2]{Ren16}, the value function is also a viscosity solution to the following elliptic path-dependent HJB-equation
$$ ru - \sup_{\gamma\in[\underline{\sigma},\overline{\sigma}]}\frac{1}{2}\gamma^2\partial^2_{\omega\omega} u = 0 \quad \mbox{on } \cQ\quad \mbox{and} \quad u=\xi\quad \mbox{on }\partial\cQ, $$
where 
$$ \cQ:= \{\omega\in\Omega^e: \omega_t\in Q, ~~\forall t\geq 0\}\quad \mbox{and}\quad \Omega^e:=\{\omega\in\Omega:\omega =\omega_{t\wedge\cdot} \mbox{ for some }t\geq 0\}. $$
We refer the interested reader to \cite{Ren16} and the references therein for more details about the theory of path-dependent PDE. 

\vspace{2mm}

\begin{remark}
	Here we connect the random horizon 2BSDE to the elliptic path-dependent PDE via the value function of the stochastic control problem \eqref{eq:defu0}. 
	In order to verify that the value function is a viscosity solution to the path-dependent PDE, one need first prove it is uniformly continuous (according to the definition in \cite{Ren16}). 
	This regularity requirement is closely related to the generator and the boundary of the equation. 
	In the example above, we address the most simple case in which the generator is uniformly elliptic and the boundary is convex. 
	In such setting, one may prove the desired uniform continuity, using an elementary argument of which the key ingredient is to verify the uniform continuity of 
	  \begin{equation*}
	     x\mapsto {\rm H}^x_Q:=\inf\{t\geq 0: x+ \omega_t\notin Q\}
	  \end{equation*}
	  under a nonlinear expectation. 
	For more details, see \cite[Proposition 8.1]{Ren16}. 
\end{remark}




\ACKNO{We would like to thank two anonymous referees for their careful reading of our manuscript and their remarks.}

\end{document}